\documentclass[reqno, dvipsnames]{amsart}

\usepackage[utf8]{inputenc}

\usepackage{amscd, amssymb, amsthm, mathrsfs, amsmath, amsrefs, amsfonts}
\usepackage[all]{xy}
\usepackage{tikz, tikz-cd}
\usetikzlibrary{matrix,arrows,decorations.pathreplacing,calc} %% For big braces
\usepackage{graphicx}
\usepackage[colorlinks,linkcolor=red,citecolor=blue,urlcolor=blue,hypertexnames=false]{hyperref}
\usepackage{xcolor, colortbl}
\definecolor{lightgray}{gray}{0.94}

\usepackage{textcomp}

\numberwithin{equation}{section}

\newtheorem{thm}[equation]{Theorem}
\newtheorem*{mainthm*}{Main Theorem}
\newtheorem{prop}[equation]{Proposition}
\newtheorem{coro}[equation]{Corollary}
\newtheorem{lem}[equation]{Lemma}

\theoremstyle{definition}
\newtheorem{defi}[equation]{Definition}
\newtheorem{rem}[equation]{Remark}
\newtheorem{cons}[equation]{Construction}
\newtheorem{exa}[equation]{Example}

\newcommand{\cC}{\mathcal{C}}
\newcommand{\cD}{\mathcal{D}}
\newcommand{\cW}{\mathcal{W}}
\newcommand{\cF}{\mathcal{F}}
\newcommand{\card}{\mathrm{card}}
\renewcommand{\S}{\mathbb{S}}
\newcommand{\Z}{\mathbb{Z}}
\newcommand{\N}{\mathbb{N}}
\newcommand{\Hom}{\mathrm{Hom}}
\newcommand{\End}{\mathrm{End}}
\newcommand{\colim}{\operatornamewithlimits{colim}}
\newcommand{\Ex}{\mathit{Ex}}
\newcommand{\incl}{\mathrm{incl}}
\newcommand{\op}{\mathit{op}}
\newcommand{\scrA}{\mathscr{A}}
\newcommand{\id}{\mathrm{id}}
\newcommand{\can}{\mathrm{can}}
\newcommand{\ev}{\mathrm{ev}}

\newcommand{\cE}{\mathcal{E}}

\newcommand{\scrK}{\mathscr{K}}
\newcommand{\ob}{{\mathrm{ob}\,}}
\newcommand{\Res}{{\mathrm{Res}}}
\newcommand{\bbE}{\mathbb{E}}
\newcommand{\pr}{\mathrm{pr}}
\newcommand{\C}{\mathbb{C}}

\newcommand{\plus}{{\raisebox{.2\height}{\scalebox{.6}{+}}}}

\newcommand{\unit}{\theta}
\newcommand{\counit}{\rho}

\newcommand{\bK}{\mathbf{K}}
\newcommand{\Ktop}{K^\mathrm{top}}
\newcommand{\kk}{{kk}}
\newcommand{\KK}{{KK}}
\newcommand{\KH}{\mathrm{KH}}
\newcommand{\bKH}{\mathbf{KH}}

\newcommand{\FIN}{{\mathcal{F}in}}
\newcommand{\GF}{(G,\cF)}
\newcommand{\GFin}{(G, \FIN)}
\newcommand{\OrG}{{\mathrm{Or}(G)}}
\newcommand{\OrGF}{{\mathrm{Or}(G,\cF)}}
\newcommand{\OrGFin}{{\mathrm{Or}(G,\FIN)}}
\newcommand{\EFG}{{E_{\mathcal{F}}(G)}}
\newcommand{\EFinG}{E_\FIN(G)}

\newcommand{\sk}{\operatorname{sk}}

\newcommand{\Stab}{\mathrm{Stab}}
\newcommand{\Ind}{\mathrm{Ind}}
\newcommand{\supp}{\mathrm{supp}}

\newcommand{\Algl}{{\mathrm{Alg}_\ell}}
\newcommand{\GAlgl}{{G\mathrm{Alg}_\ell}}
\newcommand{\HAlgl}{{H\mathrm{Alg}_\ell}}
\newcommand{\GModl}{{G\mathrm{Mod}_\ell}}
\newcommand{\Catl}{{\mathrm{Cat}_\ell}}
\newcommand{\GCatl}{{G\mathrm{Cat}_\ell}}

\newcommand{\inj}{\mathrm{inj}}
\newcommand{\proj}{\mathrm{proj}}
\newcommand{\const}{\operatorname{const}}
\newcommand{\B}{B}
\newcommand{\scrC}{\mathscr{C}}
\newcommand{\bQ}{\bar{Q}}

\newcommand{\Idos}{L}
\newcommand{\Iuno}{M}
\newcommand{\ForgetfulFunctor}{u}
\newcommand{\Ustar}[1]{{(\ForgetfulFunctor_{#1})^\star}}
\newcommand{\Uexcl}[1]{{(\ForgetfulFunctor_{#1})_!}}
\newcommand{\Uast}[1]{{(\ForgetfulFunctor_{#1})^*}}
\newcommand{\Uastd}[1]{{(\ForgetfulFunctor_{#1})_*}}
\newcommand{\slice}[2]{{{#1}_{\hspace{-0.1em}/#2}}}
\newcommand{\intg}[1]{\int^\slice{\cO}{#1}}
\newcommand{\cO}{\mathcal{O}}
\newcommand{\J}{J}

\newcommand{\funcoend}{\mathscr{C}} % Functor B(O)^O->Spt^O de tomar coend
\newcommand{\clas}{\mathrm{clas}}
\newcommand{\diag}{\mathrm{diag}}
\newcommand{\bE}{\mathbf{E}}
\newcommand{\bF}{\mathbf{F}}
\newcommand{\bR}{\mathbf{R}}

\newcommand{\cA}{\mathcal{A}}
\newcommand{\cR}{\mathcal{R}}
\newcommand{\bbK}{\mathbb{K}}
\newcommand{\Spt}{\mathrm{Sp}}
\newcommand{\pt}{\mathrm{pt}}

\newcommand{\tri}{\mathscr{T}}

\usepackage{todonotes}

\newcommand{\cte}{\delta}

\pgfkeys{tikz/mymatrixenv/.style={decoration=brace,every left delimiter/.style=    {xshift=3pt},every right delimiter/.style={xshift=1pt}}}
\pgfkeys{tikz/mymatrix/.style={matrix of math nodes,left delimiter=(,right delimiter=),inner sep=2pt,column sep=1em,row sep=0.5em,nodes={inner sep=0pt}}}
\pgfkeys{tikz/mymatrixbrace/.style={decorate,thick}}
\newcommand\mymatrixbraceoffseth{0.8em}
\newcommand\mymatrixbraceoffsetv{0.6em}

\newcommand*\mymatrixbraceright[4][m]{
	\draw[mymatrixbrace] ($(#1.north west)!(#1-#3-1.south west)!(#1.south west)-    (\mymatrixbraceoffseth,0)$)
	-- node[left=2pt] {#4} 
	($(#1.north west)!(#1-#2-1.north west)!(#1.south west)-(\mymatrixbraceoffseth,0)$);
}
\newcommand*\mymatrixbracetop[4][m]{
	\draw[mymatrixbrace] ($(#1.north west)!(#1-1-#2.north west)!(#1.north east)+(0,\mymatrixbraceoffsetv)$)
	-- node[above=2pt] {#4} 
	($(#1.north west)!(#1-1-#3.north east)!(#1.north east)+(0,\mymatrixbraceoffsetv)$);
}
\begin{document}

\title{Algebraic $\kk$-theory and the $\KH$-isomorphism conjecture}
\author{E. Ellis}
\email{eellis@fing.edu.uy}
\address{IMERL. Facultad de Ingenier\'\i a. Universidad de la Rep\'ublica. Montevideo, Uruguay.}
\author{E. Rodr\'iguez Cirone}
\email{erodriguezcirone@cbc.uba.ar}
\address{Dep. Matemática -- CBC -- UBA, Buenos Aires, Argentina.}
% \thanks{}
\thanks{All authors were partially supported by grants ANII FCE-3-2018-1-148588, PICT-2021-I-A-00710 and UBACyT 2023 20020220300206BA. The first author was partially supported by ANII, CSIC and PEDECIBA. The second author was partially supported by grant PIP 2021-2023.}


\maketitle

\begin{abstract}
    We relate the Davis-L\"uck homology with coefficients in Weibel's homotopy $K$-theory to the equivariant algebraic $\kk$-theory using homotopy theory and adjointness theorems.
    We express the left hand side of the assembly map for the $\KH$-isomorphism conjecture in terms of equivariant algebraic $\kk$-groups.
\end{abstract}

%%%%%%%%%%%%%%%%%%%%%%%%%%%%%%%%%%%%%%%%%%%%%%%%%%%%%%%%%%%%%%%%%%%%%%%%%%%%%%%%%%%%%
%%%%%%%%%%%%%%%%%%%%%%%%%%%%%%%%%%% Introduction %%%%%%%%%%%%%%%%%%%%%%%%%%%%%%%%%%%%
%%%%%%%%%%%%%%%%%%%%%%%%%%%%%%%%%%%%%%%%%%%%%%%%%%%%%%%%%%%%%%%%%%%%%%%%%%%%%%%%%%%%%
\section{Introduction}
Kasparov $\KK$-theory, introduced in \cite{kasp}, is the major tool in noncommutative topology. To every pair $(A,B)$ of separable $C^*$-algebras it associates an abelian group $\KK(A,B)$ that is a common generalization both of topological $K$-theory $\Ktop_{*}(B)\cong\KK_{*}(\mathbb{C},B)$ and topological $K$-homology.
Kasparov $\KK$-theory was deeply studied by Cuntz, who gave an alternative description of the groups $\KK(A,B)$ and provided a new perspective on the theory \cite{newlook}. The groups $\KK(A,B)$ form the hom-sets of a category $\KK$ that is the target of the universal homotopy invariant, $C^{*}$-stable and split-exact functor from separable $C^*$-algebras into an additive category; these universal properties were established by Higson \cite{higson}. Kasparov developed in \cite{kaspG} an equivariant version of $\KK$-theory, $\KK^G$, for separable $C^*$-algebras with a an action of a group $G$ by $\ast$-automorphisms.
Equivariant $\KK$-theory was used in \cite{bacohi} to formulate the Baum-Connes conjecture with coefficients and is one of the basic tools in the proofs of results about it. Cuntz also analyzed $\KK$-theory in algebraic terms and defined a bivariant $K$-theory for all locally convex algebras \cite{cuntz}.

Motivated by the works of Cuntz and Higson, algebraic $\kk$-theory was introduced by Corti\~nas and Thom in \cite{cortho} as a completely algebraic counterpart of Kasparov $\KK$-theory.
To every pair $(A,B)$ of algebras over a commutative ring $\ell$ it associates an abelian group $\kk(A,B)$ that generalizes Weibel's homotopy $K$-theory $\KH$ defined in \cite{chuck}. By \cite{cortho}*{Main Theorem} we have $\KH_*(B)\cong \kk_*(\ell,B)$ for every algebra $B$. The groups $\kk(A,B)$ form the hom-sets of a category $\kk$ that is the target of the universal (polynomial) homotopy invariant, $M_\infty$-stable and excisive functor from $\Algl$ into a triangulated category. An equivariant version of algebraic $\kk$-theory, $\kk^G$, was developed in \cite{euge} for $\ell$-algebras with an action of a group $G$. By Garkusha's representability theorems \cite{garku}, the groups $\kk(A,B)$ can be recovered as the homotopy groups of certain spectra $\bbK(A,B)$. These spectra were constructed in \cite{garku} and were later extended in \cite{tesisema} to spectra $\bbK^G(A,B)$ representing $\kk^G(A,B)$.

We may summarize the above into a dictionary between Kasparov $\KK$-theory and algebraic $\kk$-theory:
\begin{center}
    \begin{tabular}{ c c c }
        \hline
        \rowcolor{lightgray}
        continuous homotopy invariance           & $\longleftrightarrow$ & polynomial homotopy invariance  \\
        $C^*$-stability                          & $\longleftrightarrow$ & $M_\infty$-stability            \\
        \rowcolor{lightgray}
        split exactness                          & $\longleftrightarrow$ & excision                        \\
        $\KK$-theory of $C^*$-algebras           & $\longleftrightarrow$ & $\kk$-theory of $\ell$-algebras \\
        \rowcolor{lightgray}
        $\KK(\mathbb{C},B)\cong \Ktop_*(B)$      & $\longleftrightarrow$ & $\kk_*(\ell,B)\cong \KH_*(B)$   \\
        $\KK^G$-theory of $G$-$C^*$-algebras     & $\longleftrightarrow$ & $\kk^G$-theory of $G$-algebras  \\
        \rowcolor{lightgray}
        $\begin{array}{c}
                 \text{formulation of the}          \\
                 \text{Baum-Connes conjecture with} \\
                 \text{coefficients using $\KK^G$-theory}
             \end{array}$ & $\longleftrightarrow$ & \textbf{?}                                                 \\
        \hline
    \end{tabular}
\end{center}
\medskip
This paper is concerned with the last line of this dictionary. The Baum-Connes conjecture with coefficients was originally formulated using $\KK^G$-theory, but no analogue is known on the algebraic side. We shed light on this by exploring the relation between $\kk^G$-theory and the $\KH$-isomorphism conjecture---the latter was introduced in \cite{balu} using the homotopical approach to isomorphism conjectures developed in \cite{dalu}. Our main theorem states that the domain of the $\KH$-assembly map can be described in terms of $\kk^G$-groups in way that is completely analogous to the case of the Baum-Connes assembly map. As explained in \ref{sec:KHiso}, a group $G$ and a $G$-algebra $B$ \emph{satisfy the $\KH$-isomorphism conjecture} if certain morphism
\[H^G_*(\EFinG;\bKH_B)\to \KH_*(B\rtimes G),\]
called the \emph{assembly map}, is an isomorphism. Here $\EFinG$ is the classifying space of $G$ with respect to its family of finite subgroups and $H^G_*(-;\bKH_B)$ is a $G$-homology theory such that $H^G_*(G/H;\bKH_B)\cong \KH_*(B\rtimes H)$ for every sugbroup $H\subseteq G$. We prove the following result; see \ref{sec:polyfunc} for the precise definition of $\ell^{(X)}$ and \ref{sec:isoconj} for the definition of $\GFin$-complex:

\begin{mainthm*}[cf. Theorem \ref{thm:mainthm} and Remark \ref{rem:mainthm}]
    Let $G$ be a group such that $|H|$ is invertible in the base ring $\ell$ for every finite subgroup $H\subseteq G$. Let $B$ be a $G$-algebra. Then for every $\GFin$-complex $Z$ there is a natural isomorphism
    \begin{equation}\label{eq:mainthmintro}H^G_*(Z;\bKH_B)\cong \colim_{\substack{X\subseteq Z \\ \text{$G$-finite}}}\kk^G_*(\ell^{(X)}, B)\end{equation}
    where $\ell^{(X)}$ denotes the algebra of finitely supported polynomial functions on $X$.
\end{mainthm*}

For $Z=G/H$ with $H\subseteq G$ a finite subgroup, the isomorphism \eqref{eq:mainthmintro} is easily understood using the adjointness theorems in $\kk$-theory. Let $H\subseteq G$ be a finite subgroup whose order is invertible in $\ell$. By the Green-Julg Theorem \cite{euge}*{Theorem 5.2.1} we have a natural isomorphism
\begin{equation}\label{gjmor}\kk^H(A^\tau, B) \cong \kk(A, B\rtimes H)\end{equation}
for $A\in\Algl$ and $B\in \HAlgl$. Here $A^\tau$ denotes the algebra $A$ considered as an $H$-algebra with trivial $H$-action. By the adjunction between induction and restriction \cite{euge}*{Theorem 6.14} we have a natural isomorphism
\begin{equation}\label{irmor}\kk^G(\Ind_H^G(A), B)\cong \kk^H(A, \Res_G^H(B))\end{equation}
for $A\in \HAlgl$ and $B\in\GAlgl$. Since $\Ind_H^G(A^\tau)=\bigoplus_{G/H}A=A^{(G/H)}$, upon composing \eqref{gjmor} and \eqref{irmor} we get a natural isomorphism
\begin{equation}\label{eq:compositeAdj}
    \kk^G(A^{(G/H)}, B)\cong \kk(A, B\rtimes H)
\end{equation}
for $A\in\Algl$ and $B\in\GAlgl$. As it will turn out, the identification \eqref{eq:mainthmintro} is obtained by glueing the isomorphisms \eqref{eq:compositeAdj} with homotopy theoretic techniques. In this process we will replace the groups $\kk$ and $\kk^G$ by the the spectra $\bbK$ and $\bbK^G$ that respectively represent them; see \cite{garku} and Appendix \ref{sec:mcs}.

The paper is organized as follows. In Section \ref{sec:prelim} we recall definitions and preliminaries that are used throughout the article.
In Section \ref{sec:adtr} we give explicit descriptions of the unit and counit of the adjunction \eqref{eq:compositeAdj}.
In Section \ref{sec:crossedpro} we define a triangulated functor $\cR(-\rtimes G/H):\kk^G\to\kk$ that is naturally isomorphic to the crossed product with $H$. This allows us to replace the right-hand side of \eqref{eq:compositeAdj} by an actual functor on $\OrG$; see Defintion \ref{defi:OrG}. In Section \ref{sec:natadj}, we prove that the isomorphism
\begin{equation}\label{eq:compositeAdjNat}
    \kk^G(A^{(G/H)}, B)\cong \kk(A, \cR(B\rtimes G/H))
\end{equation}
is natural in $G/H$; see Theorem \ref{thm:adj}. Section \ref{sec:bomba} is the technical core of this work and is devoted to lifting the isomorphism \eqref{eq:compositeAdjNat} to a weak equivalence of $\OrGFin$-spectra. By Lemma \ref{lem:natcounit}, we can describe the isomorphism \eqref{eq:compositeAdjNat} as the composite of the morphisms in the zig-zag \eqref{eq:decompAdj}. Upon replacing $\kk$ by $\bbK$ and $\kk^G$ by $\bbK^G$ we obtain a zig-zag:
\begin{equation}\label{eq:zzagpintro}\begin{gathered}\xymatrix@C=3.5em{
        \bbK(A, \cR(B\rtimes G/H))\ar[r]^-{(-)^{(G/H)}} & \bbK^G(A^{(G/H)}, \left[\cR(B\rtimes G/H)\right]^{(G/H)})\\
        \bbK^G(A^{(G/H)}, M_{G}B) & \bbK^G(A^{(G/H)}, \cR\left[(B\rtimes G/H)^{(G/H)}\right])\ar[u]_-{\rotatebox{90}{$\scriptstyle\sim$}}\ar[l]_-{\cR(\zeta_{G/H})}}\end{gathered}\end{equation}
Here the technical difficulties arise:
\begin{enumerate}
    \item\label{item:intro1} How to consider the spectra on the right column as covariant functors $\OrGFin\to\Spt$ (or to replace them by ones)?
    \item\label{item:intro2} Once the previous question has been addressed, are the morphisms in \eqref{eq:zzagpintro} natural in $G/H$?
\end{enumerate}
An answer to \eqref{item:intro1} is provided in sections \ref{sec:coend} and \ref{sec:coend2}.
In Section \ref{sec:modcatbi} we introduce a model category that allows us to build models for certain homotopy coends. An answer to \eqref{item:intro2} is given in Theorem \ref{thm:ZZAG}. The latter is the key technical result of this paper. In Section \ref{sec:prin} we finally prove our Main Theorem (Theorem \ref{thm:mainthm}). In Section \ref{sec:kktheoreticASS} we prove a first result towards obtaining a $\kk$-theoretic description of the $\KH$-assembly map.

\subsection*{Acknowledgements}
Both authors wish to thank Willie Corti\~nas for introducing them to the problem that led to the main result of this work as well as for many fruitful discussions on the topic. They would also like to thank the referee for carefully reading the paper and suggesting improvements.

%%%%%%%%%%%%%%%%%%%%%%%%%%%%%%%%%%%%%%%%%%%%%%%%%%%%%%%%%%%%%%%%%%%%%%%%%%%%%%%%%%%%%
%%%%%%%%%%%%%%%%%%% Notation, conventions and preliminaries %%%%%%%%%%%%%%%%%%%%%%%%%
%%%%%%%%%%%%%%%%%%%%%%%%%%%%%%%%%%%%%%%%%%%%%%%%%%%%%%%%%%%%%%%%%%%%%%%%%%%%%%%%%%%%%
\section{Notation, conventions and preliminaries}\label{sec:prelim}

Throughout this text, $\ell$ denotes a commutative ring with unit and $G$ denotes a group.
We write $\Algl$ for the category of not necessarily unital $\ell$-algebras and algebra homomorphisms. The objects of $\Algl$ are simply called \emph{algebras}. Tensor products over $\ell$ are denoted by $\otimes$. A \emph{$G$-algebra} is an algebra with a left action of $G$. We write $\GAlgl$ for the category of $G$-algebras and $G$-equivariant algebra homomorphisms. The category of simplicial sets is denoted by $\S$. A \emph{$G$-simplicial set} is a simplicial set endowed with a left action of $G$. We write $\S^G$ for the category of $G$-simplicial sets with equivariant morphisms.

\subsection{Algebras of polynomial functions on a simplicial set}\label{sec:polyfunc}

Let $B$ be an algebra. For $n\geq 0$, the algebra $B^{\Delta^n}$ of \emph{$B$-valued polynomial functions on the standard $n$-simplex} is defined as $B^{\Delta^n}:=B[t_0,\dots, t_n]/\langle t_0+\cdots + t_n-1\rangle$. For a simplicial set $X$, the algebra $B^X$ of \emph{$B$-valued polynomial functions on $X$} is defined as
\[B^X:=\Hom_\S(X,B^\Delta)\]
where $B^\Delta$ is the simplicial algebra $[n]\mapsto B^{\Delta^n}$. Note that the functor $B^{-}:\S\to\Algl$ sends colimits to limits. We summarize useful properties of this construction in the following lemma.

\begin{lem}[\cite{cortho}*{Lemma 3.1.2 and Proposition 3.1.3}]\label{lem:cortholX}
    Let $B$ be an algebra.
    \begin{enumerate}
        \item If $X\subseteq Y$ is an inclusion of simplicial sets, then $B^Y\to B^X$ is surjective.
        \item If $K$ is a finite simplicial set, then $\ell^K$ is free as $\ell$-module and there is a natural isomorphism of algebras $B^K\cong B\otimes \ell^K$.
    \end{enumerate}
\end{lem}

Let $X$ be a simplicial set. We recall the definition of the algebra $B^{(X)}$ of \emph{finitely supported $B$-valued polynomial functions on $X$}; see \cite{corel}*{Section 9.3} for details. The \emph{support} of a polynomial function $\phi\in B^X$ is defined as the simplicial subset of $X$ generated by
$\cup_{n\geq 0}\{\sigma\in X_n:\phi(\sigma)\ne 0\}$. Put:
\[B^{(X)}=\{\phi\in B^X:\supp(\phi) \text{ is finite}\}\]
Then $B^{(X)}\subseteq B^X$ is a two-sided ideal. If $K$ is a finite simplicial set, then we have $B^{(K)}=B^K$.

\begin{rem}
    The assignment $X\mapsto B^{(X)}$ is natural only for \emph{proper} maps of simplicial sets---i.e. morphisms $f:X\to Y$ such that $f^{-1}(K)$ is finite for every finite simplicial subset $K\subseteq Y$.
\end{rem}

\begin{rem}\label{rem:polynomialfcoprod}
    If $\{X_i\}_i$ is a family of simplicial sets then we have:
    \[B^{(\bigsqcup_iX_i)}\cong \bigoplus_iB^{(X_i)}\]
    The behaviour of $B^{(-)}$ with respect to general colimits of proper maps is analyzed in \cite{corel}*{Section 9.3, conditions $\partial_0$) and $\partial_1$)}.
    In particular, upon applying $B^{(-)}$ to a pushout square of simplicial sets where all the morphisms are proper we obtain a pullback square of algebras.
\end{rem}

If $X$ is a $G$-simplicial set, then $B^X$ and $B^{(X)}$ are $G$-algebras with the action defined by $(g\cdot \phi)(\sigma):=\phi(g^{-1}\cdot \sigma)$ for $\phi\in B^X$, $n\geq 0$ and $\sigma\in X_n$. If $X\to Y$ is a proper morphism in $\S^G$, then $B^{(Y)}\to B^{(X)}$ is a morphism in $\GAlgl$.

\begin{exa}\label{exa:GH}
    Let $H\subseteq G$ be a subgroup and let $G/H=\{uH: u\in G\}$ be the set of left cosets of $H$. We can consider the $G$-set $G/H$ as a $0$-dimensional $G$-simplicial set. For any algebra $A$ we have:
    \begin{align*}
        A^{G/H}   & =A^{\bigsqcup_{uH\in G/H}\{uH\}}\cong\prod_{uH\in G/H}A^{\{uH\}}\cong \prod_{G/H}A                         \\
        A^{(G/H)} & =A^{\left(\bigsqcup_{uH\in G/H}\{uH\}\right)}\cong \bigoplus_{uH\in G/H}A^{(\{uH\})}\cong \bigoplus_{G/H}A
    \end{align*}
    In particular, $\ell^{(G/H)}\cong \bigoplus_{G/H}\ell$ is a free $\ell$-module with base $\{\chi_{uH}: uH\in G/H\}$. When considering $\chi_{uH}$ as a function $G/H\to\ell$ with finite support, then $\chi_{uH}(vH)=\delta_{uH, vH}$ (Kronecker delta). Note that we have $g\cdot \chi_{uH}=\chi_{guH}$ for every $g\in G$. For any algebra $A$ we have:
    \[A^{(G/H)}\cong \bigoplus_{G/H}A\cong A\otimes \bigoplus_{G/H}\ell\cong A\otimes \ell^{(G/H)}\]
    For $a \in A$ and $uH\in G/H$ we sometimes write $a\chi_{uH}\in A^{(G/H)}$ for the element corresponding to $a\otimes \chi_{uH}\in A\otimes \ell^{(G/H)}$ under the isomorphism above.
\end{exa}

\subsection{Isomorphism conjectures}\label{sec:isoconj}
In this section we recall the homotopical aproach to the isomorphism conjectures developed in \cite{dalu}. By \cite{corel}*{Section 2} it is equivalent to work in the topological or in the simplicial setting; we choose to do the latter.

\begin{defi}\label{defi:OrG}
    Let $G$ be a group. The \emph{orbit category} of $G$ is the category $\OrG$ whose objects are the $G$-sets $G/H$ with $H\subseteq G$ a subgroup and whose morphisms are the $G$-equivariant functions.
\end{defi}

\subsubsection{Model structure on \texorpdfstring{$G$}{G}-simplicial sets determined by a family of subgroups}
For $X\in\S^G$ and $H\subseteq G$ a subgroup, let $X^H$ be the simplicial subset of $X$ formed by those simplices that are fixed by $H$. Note that
\[X^H\cong \Hom_G(G/H,X)\]
so that, for fixed $X$, $G/H\mapsto X^H$ is a functor on $\OrG$. Let $\cF$ be a nonempty family of subgroups of $G$ closed under conjugation and under taking subgroups.  As explained in \cite{corel}*{Section 2}, $\S^G$ admits a model structure where a mophism $f:X\to Y$ is a weak equivalence (resp. a fibration) if and only if $f^H:X^H\to Y^H$ is a weak equivalence (resp. a fibration) in $\S$ for every $H\in\cF$. Moreover, the cofibrant objects, called \emph{$\GF$-complexes}, are those $G$-simplicial sets $X$ such that the stabilizer of every simplex of $X$ is in $\cF$. These can be alternatively described as those $G$-simplicial sets $X$ that can be built from cells of the form $G/H\times\Delta^n$ with $H\in\cF$; see Appendix \ref{sec:Gsset}. We write $\EFG\to\pt$ for a cofibrant replacement of the point and call $\EFG$ a \emph{model for the classifying space of $G$ with respect to $\cF$}.

\subsubsection{Equivariant homology and assembly map}
Let $\Spt$ be the category of spectra; see Appendix \ref{sec:mcs} for details. An \emph{$\OrG$-spectrum} is a functor $\OrG\to\Spt$. Fix an $\OrG$-spectrum $\bE$ and define $H^G(-;\bE):\S^G\to\Spt$ as the coend:
\begin{equation}\label{eq:Ghomology}H^G(X; \bE):= \int^{G/H}X^H_+\wedge \bE(G/H)\end{equation}
Then the groups $H^G_*(X;\bE):=\pi_*H^G(X;\bE)$ assemble into a homology theory of $G$-simplicial sets such that
\[H^G_*(G/H; \bE)\cong \pi_*\bE(G/H)\]
for any $G/H\in\OrG$; see \cite{dalu}*{Lemma 6.1}.
The \emph{assembly map for the triple $(\bE, G, \cF)$} is the morphism
\begin{equation}\label{eq:ass}H^G_*(\EFG;\bE)\to H^G_*(\pt;\bE)\cong \pi_*\bE(G/G)\end{equation}
induced on homology by the cofibrant replacement $\EFG\to\pt$. We say that \emph{the isomorphism conjecture for the triple $(\bE, G, \cF)$ holds} if \eqref{eq:ass} is an isomorphism.

\begin{rem}
    Let $\OrGF$ be the full subcategory of $\OrG$ whose objects are those $G/H$ with $H\in\cF$. If the functor $\bE$ is only defined on $\OrGF$, the formula \eqref{eq:Ghomology} still makes sense but defines a homology theory in the full subcategory of $\S^G$ whose objects are the $\GF$-complexes.
\end{rem}

\begin{rem}\label{rem:natHG}
    Any morphism $f:\bE\to\bF$ of $\OrGF$-spectra induces a natural transformation $f_*:H^G(-;\bE)\to H^G(-;\bF)$ of functors from the category of $\GF$-complexes into $\Spt$.
    Upon taking homotopy groups, we get a natural transformation:
    \begin{equation}\label{eq:natHG}f_*:H^G_*(X; \bE)\to H^G_*(X; \bF)\end{equation}
    Moreover, if $f$ is an objectwise weak equivalence, then \eqref{eq:natHG} is an isomorphism by \cite{dalu}*{Lemma 4.6}.
\end{rem}

\subsubsection{The Baum-Connes conjecture with coefficients}
Let $G$ be a countable group and $B$ be a separable $G$-$C^*$-algebra. The Baum-Connes conjecture fits into the framework described above by considering the family $\FIN$ of finite subgroups and an $\OrG$-spectrum $\bK_B$ such that
\[\pi_*\bK_B(G/H)\cong \Ktop_*(B\rtimes_r H)\]
for every subgroup $H\subseteq G$. Here $\rtimes_r$ denotes the reduced crossed product. The equivalence of the corresponding assembly map with the original one in \cite{bacohi} was recently proved independently in \cite{kranz} and in \cite{bel}.

\subsubsection{The $\KH$-isomorphism conjecture}\label{sec:KHiso}
Let $G$ be a group and let $B$ be a $G$-algebra. The $\KH$-isomorphism conjecture was introduced in \cite{balu}*{Section 7} using the machinery described above; see \cite{bb}*{Section 15.3} for the status of this conjecture. The conjecture is obtained upon considering the family $\FIN$ of finite subgroups of $G$ and an $\OrG$-spectrum $\bKH_B$ such that
\begin{equation}\label{eq:KHBproperty}\pi_*\bKH_B(G/H)\cong \KH_*(B\rtimes H)\end{equation}
for every subgroup $H\subseteq G$. The original construction of $\bKH_B$ was done in \cite{balu}. In Section \ref{sec:prin} we provide a different construction using the spectra that represent algebraic $\kk$-theory \cite{garku}.

\subsection{Algebraic \texorpdfstring{$\kk$}{kk}-theory}
Let $\cC$ denote either the category $\Algl$ or $\GAlgl$.

\subsubsection{Homotopy invariance}
Two morphisms $f,g:A\to B$ in $\cC$ are \emph{elementary homotopic} if there exists a morphism $H:A\to B[t]$ such that $\ev_0\circ H =f$ and $\ev_1\circ H=g$. Here, $B[t]$ is the algebra of polynomials with coefficients in $B$ and $\ev_i:B[t]\to B$ is the evaluation at $i$. If $B$ is a $G$-algebra we let $G$ act trivially on $t$. The relation of elementary homotopy is easily seen to be reflexive and symmetric but not transitive. We consider the equivalence relation generated by elementary homotopy and call it \emph{homotopy}.

Let $\scrC$ be a category. A functor $F:\cC\to\scrC$ is \emph{homotopy invariant} if it sends homotopic morphisms to the same morphism.
It is easily verified that $F$ is homotopy invariant if and only if $F(A\subseteq A[t])$ is an isomorphism for every $A$.

\subsubsection{Matrix stability}
Let $A$ be an algebra and $X$ be an infinite set. We write $M_XA$ for the algebra of finitely supported matrices with coefficients in $A$ that are indexed over $X\times X$; when $A=\ell$ we just write $M_X$. Note that $M_XA\cong M_X\otimes A$. The assignment $X\mapsto M_X$ depends covariantly on $X$ with respect to injective functions.

Fix $x\in X$ and let $\iota_x:A\to M_X A$ be given by $\iota_x(a)=e_{x,x}\otimes a$. Let $\scrC$ be a category. A functor $F:\cC\to\scrC$ is \emph{$M_X$-stable} if $F(\iota_x:A\to M_X A)$ is an isomorphism for every $A$. This definition is independent of the choice of $x$ by the following result:

\begin{lem}[\cite{friendly}*{Proposition 2.2.6}]\label{lem:conju}
    Let $\scrC$ be a category and let $F:\GAlgl\to\scrC$ be an $M_2$-stable functor. Let $B\subseteq C$ be an inclusion of $G$-algebras. Suppose that $C$ is unital and let $V\in C$ be an invertible element such that $VB,BV^{-1}\subseteq B$ and $g\cdot V=V$ for all $g\in G$. Then the formula $\phi^{V}(b)=VbV^{-1}$ defines a $G$-algebra homomorphism $\phi^V:B\to B$ such that $F(\phi^V)=\id_{F(B)}$.
\end{lem}
\begin{proof}
    It is easily verified that $g\cdot V^{-1}=V^{-1}$ for all $g\in G$ and that $\phi^V$ defines a $G$-algebra homomorphism. The rest of the proof of \cite{friendly}*{Proposition 2.2.6} carries over verbatim.
\end{proof}

Let $S$ be a $G$-set and let $|S|$ be its underlying set. We write $M_S$ for the algebra $M_{|S|}$ endowed with the $G$-action defined by $g\cdot e_{s,t}=e_{g\cdot s, g\cdot t}$.
Let $G_+=G\bigsqcup\{\plus\}$ and let $\iota:\ell\to M_{G_+}$ (respectively $\iota':M_G\to M_{G_+}$) be the morphism induced by the inclusion $\{\plus\}\subset G_+$ (resp. $G\subset G_+$). Let $\scrC$ be a category. There is a notion of \emph{$G$-stability} for functors $\GAlgl\to\scrC$; see \cite{euge}*{Section 3.1} for details. It turns out that any $G$-stable functor sends the morphisms
\begin{equation}\label{eq:zigzag}
    \xymatrix{B\ar[r]^-{\iota} & M_{G_+}\otimes B & M_G\otimes B\ar[l]_-{\iota'}}
\end{equation}
to isomorphisms, for every $B\in\GAlgl$.

\begin{rem}
    For any $G$-algebra $B$, we have an isomorphism $R_{S,B}:(M_S\otimes B)\rtimes G\to M_{|S|}\otimes (B\rtimes G)$ defined by:
    \begin{equation}\label{eq:R}
        R_{S,B}((e_{s,t}\otimes b)\rtimes g)=e_{s,g^{-1}t}\otimes (b\rtimes g)\end{equation}
    This isomorphism is natural in $S$ with respect to injective morphisms of $G$-sets and natural in $B$ with respect to morphisms of $G$-algebras.
\end{rem}

\subsubsection{Excision}\label{subsec:excision}
An \emph{extension} of algebras (respectively of $G$-algebras) is a short exact sequence
\begin{equation}\label{eq:extension}
    \xymatrix{\cE: A \ar[r] & B \ar[r] & C \ar@{-->}@/_1pc/[l]
    }
\end{equation}
that splits in the category of $\ell$-modules (resp. $G$-modules).
Let $(\tri, \Omega)$ be a triangulated category.
A functor $F:\cC\to\tri$ is \emph{excisive} \cite{cortho}*{Section 6.6} if it associates to every extension \eqref{eq:extension} a morphism $\partial_\cE:\Omega F(C)\to F(A)$ that fits into a triangle:
\[
    \xymatrix{\Omega F(C)\ar[r]^-{\partial_\cE} & F(A) \ar[r] & F(B) \ar[r] & F(C)}
\]
These triangles are required to be natural with respect to morphisms of extensions.

\subsubsection{Bivariant \texorpdfstring{$K$}{K}-theory categories}\label{subsec:bivariant} Let $X$ be an infinite set. There exists a triangulated category $\kk$ (see \cite{cortho} for $X=\N$ and \cite{tesisema} for general $X$) endowed with a functor $j:\Algl\to\kk$ that is homotopy invariant, $M_X$-stable and excisive. This functor $j$ is moreover universal in the following sense: any functor $\Algl\to (\tri,\Omega)$ that is homotopy invariant, $M_X$-stable and excisive factors uniquely through $j$ \cite{cortho}*{Theorem 6.6.2}. By a theorem of Corti\~nas and Thom \cite{cortho}*{Main Theorem}, homotopy $K$-theory is representable in $\kk$ since we have
\[\kk_*(\ell, B)\cong \KH_*(B)\]
for every $B\in \Algl$. Here $\kk_n(A,B)$ is defined as $\kk(j(A), \Omega^nj(B))$. This representability theorem is a major computation of \cite{cortho}.

Let $G$ be a group. There exists a triangulated category $\kk^G$ (see \cite{euge} for countable $G$ and \cite{tesisema} for general $G$) endowed with a functor $j^G:\GAlgl\to\kk^G$ that is homotopy invariant, $G$-stable and excisive. This functor $j^G$ is moreover universal in the following sense: any functor $\GAlgl\to (\tri,\Omega)$ that is homotopy invariant, $G$-stable and excisive factors uniquely through $j^G$ \cite{euge}*{Theorem 4.1.1}.

Throughout this paper, by $\kk$ we shall mean the universal $M_{\N\times|G|}$-stable $\kk$-theory.

\begin{exa}
    Let $B$ be a $G$-algebra. By $G$-stability, the morphisms \eqref{eq:zigzag} induce an isomorphism $j^G(B)\cong j^G(M_GB)$ in $\kk^G$. To ease notation we often omit $j^G$ and write $B\cong M_GB$. By homotopy invariance, we have isomorphisms $B\cong B^{\Delta^n}$ in $\kk^G$, since $B^{\Delta^n}$ is isomorphic to the algebra of polynomials in $n$ variables with coefficients in $B$.
\end{exa}

%%%%%%%%%%%%%%%%%%%%%%%%%%%%%%%%%%%%%%%%%%%%%%%%%%%%%%%%%%%%%%%%%%%%%%%%%%%%%%%%%%%%%
%%%%%%%%%%%%%%%%%%%%%%%%%%%% Adjoint theorems revisited %%%%%%%%%%%%%%%%%%%%%%%%%%%%%
%%%%%%%%%%%%%%%%%%%%%%%%%%%%%%%%%%%%%%%%%%%%%%%%%%%%%%%%%%%%%%%%%%%%%%%%%%%%%%%%%%%%%
\section{Adjoint theorems revisited}\label{sec:adtr}

Let $G$ be a group and let $H\subseteq G$ be a finite subgroup whose order $n$ is invertible in $\ell$.
Recall the notation and conventions from Example \ref{exa:GH}.
By \cite{euge}*{Theorem 5.2.1} and \cite{euge}*{Theorem 6.14} we have an adjunction isomorphism
\begin{equation}\label{eq:adj}
    \kk^G(A^{(G/H)}, B)\cong \kk(A, B\rtimes H)
\end{equation}
for any $A\in\Algl$ and any $B\in\GAlgl$. For $A\in\Algl$,  let $\unit_A:A\to A^{(G/H)}\rtimes H$ be the algebra homomorphism  defined by:
\[\unit_A(a)=a\chi_H\rtimes \frac1n\sum_{h\in H}h\]
Put $\eta_A=j(\unit_A)\in \kk(A,A^{(G/H)}\rtimes H)$. We will show that $\eta_A$ is a unit for the adjunction \eqref{eq:adj}.
For $B\in\GAlgl$, let $\counit_B:(B\rtimes H)^{(G/H)}\to M_G\otimes B$ be the $G$-algebra homomorphism defined by:
\begin{equation}\label{eq:defipsi}
    \counit_B\left((b\rtimes h)\chi_{wH}\right)=\sum_{p\in wH}e_{p,ph}\otimes (p\cdot b)
\end{equation}
Let $\varepsilon_B\in \kk^G((B\rtimes H)^{(G/H)},B)$ be the following composite in $\kk^G$, where the isomorphism on the right is given by the zig-zag \eqref{eq:zigzag}:
\begin{equation}\label{eq:couniH}
    \xymatrix{(B\rtimes H)^{(G/H)}\ar[r]^-{j^G(\counit_B)} & M_G\otimes B\cong B}
\end{equation}
We will show that $\varepsilon_B$ is a counit for the adjunction \eqref{eq:adj}.

\begin{lem}\label{lem:unit}
    For any $B\in\GAlgl$ we have $(\varepsilon_B\rtimes H)\circ \eta_{B\rtimes H}=\id_{B\rtimes H}$ in $\kk$.
\end{lem}
\begin{proof}
    It is easily verified that the following diagram in $\kk$ commutes, where all the arrows are isomorphisms:
    \[
        \xymatrix{
        (M_G\otimes B)\rtimes H\ar[d]^-{\eqref{eq:R}}_-{R_{G,B}}\ar[r]^-{(\iota')_*} & (M_{G_\plus}\otimes B)\rtimes H\ar[d]^-{\eqref{eq:R}}_-{R_{G_\plus,B}} & B\rtimes H\ar[l]_-{\iota_*}\ar@/^1pc/[dl]^-{\iota_*} \\
        M_{|G|}\otimes (B\rtimes H)\ar[r]_-{(\iota')_*} & M_{|G_\plus|}\otimes (B\rtimes H)
        }
    \]
    Thus, the isomorphism $(M_G\otimes B)\rtimes H\cong B\rtimes H$ in $\kk$ induced by the zig-zag \eqref{eq:zigzag} equals the composite:
    \[
        \xymatrix@C=4em{
        (M_G\otimes B)\rtimes H\ar[r]_-{\cong}^-{j(R_{G,B})} & M_{|G|}\otimes (B\rtimes H) & B\rtimes H\ar[l]_-{j(e_{1,1}\otimes -)}^-{\cong}
        }
    \]
    To prove the lemma, it will be enough to show that the composite
    \[
        \xymatrix{
        B\rtimes H \ar[r]^-{\unit_{B\rtimes H}} & (B\rtimes H)^{(G/H)}\rtimes H\ar[r]^-{\counit_B\rtimes H} & (M_G\otimes B)\rtimes H\ar[r]^-{R_{G,B}} & M_{|G|}\otimes (B\rtimes H)
        }
    \]
    and the inclusion $e_{1,1}\otimes -:B\rtimes H\to M_{|G|}\otimes (B\rtimes H)$ induce the same morphism in $\kk$.
    Let $\Gamma$ be the algebra of matrices with coefficients in $\tilde{B}\rtimes G$ indexed by $G\times G$ that have only finitely many nonzero coefficients in each column and each row. Notice that $\Gamma$ is a unital algebra that contains $M_{|G|}\otimes (B\rtimes H)$ as a subalgebra. Put:
    \begin{equation}\label{eq:V}
        V=\sum_{g\in G}e_{g,g}\otimes (1\rtimes g)\in\Gamma
    \end{equation}
    We have:
    \begin{align*}
        [R_{G,B}\circ (\counit_B\rtimes H)\circ \unit_{B\rtimes H}](b\rtimes h) & =\frac1n\sum_{p,q\in H}e_{p,q}\otimes ((p\cdot b)\rtimes phq^{-1})     \\
                                                                                & =V\left(\frac1n\sum_{p,q\in H}e_{p,q}\otimes (b\rtimes h)\right)V^{-1}
    \end{align*}
    Moreover, $\frac1n\sum_{p,q\in H}e_{p,q}$ is a conjugate of $e_{1,1}$ in $M_{|G|}$; see \cite{euge}*{Remark 3.1.11}. By \cite{friendly}*{Proposition 2.2.6} we have
    \[j[R_{G,B}\circ (\counit_B\rtimes H)\circ \unit_{B\rtimes H}]=j(e_{1,1}\otimes -):B\rtimes H\to M_{|G|}\otimes (B\rtimes H)\]
    as we wanted to prove.
\end{proof}

\begin{lem}\label{lem:counit}
    For any $A\in\Algl$ we have
    $\varepsilon_{A^{(G/H)}}\circ\left[\left(\eta_A\right)^{(G/H)}\right]=\id_{A^{(G/H)}}$
    in $\kk^G$.
\end{lem}
\begin{proof}
    It is easily verified that the composite
    \[
        \xymatrix@C=3em{
        A^{(G/H)}\ar[r]^-{(\eta_A)^{(G/H)}} & \left(A^{(G/H)}\rtimes H\right)^{(G/H)}\ar[r]^-{\varepsilon_{A^{(G/H)}}} & M_G\otimes A^{(G/H)}\ar[r]^-{\iota'} & M_{G_\plus}\otimes A^{(G/H)}
        }
    \]
    is given by:
    \begin{equation}\label{eq:compo}
        a\chi_{wH}\mapsto \frac1n\sum_{p,q\in H}e_{wp,wq}\otimes a\chi_{wH}
    \end{equation}
    To prove the lemma, it suffices to show that the above formula induces the same morphism as $\iota:A^{(G/H)}\to M_{G_\plus}\otimes A^{(G/H)}$, $\iota(a\chi_{wH})=e_{\plus,\plus}\otimes a\chi_{wH}$, upon applying $j^G:\GAlgl\to \kk^G$. Let $\tilde{A}$ be the unitalization of $A$ and let $\Gamma_{G_+}(\tilde{A})$ be the set of those matrices with coefficients in $\tilde{A}$, indexed over $G_\plus\times G_\plus$, that have finitely many nozero coefficients in each row and each column. Then $\Gamma_{G_\plus}(\tilde{A})$ is a $G$-algebra with the usual matrix multiplication and the $G$-action defined by $(g\cdot a)_{x,y}:=a_{g^{-1}\cdot x, g^{-1}\cdot y}$.
    Moreover, we have inclusions of $G$-algebras as follows:
    \[
        M_{G_\plus}\otimes A^{(G/H)}\subseteq \left(M_{G_\plus}\otimes A\right)^{G/H}\subseteq \left(\Gamma_{G_\plus}(\tilde{A})\right)^{G/H}
    \]
    The $G$-action on the right is described by
    \[
        (g\cdot f)(tH):=g\cdot f(g^{-1}tH),
    \]
    where $f:G/H\to \Gamma_{G_\plus}(\tilde{A})$ is a function and $g\in G$. We will show that there exists an invertible $V\in(\Gamma_{G_\plus}(\tilde{A}))^{G/H}$ such that the following diagram commutes, where the horizontal morphism is given by \eqref{eq:compo}:
    \begin{equation}\label{eq:phiV}\begin{gathered}
            \xymatrix{
            A^{(G/H)}\ar[r]\ar@/_/[dr]_-{\iota} & M_{G_\plus}\otimes A^{(G/H)}\ar[d]^-{\phi^V} \\
            & M_{G_\plus}\otimes A^{(G/H)}
            }
        \end{gathered}\end{equation}
    Once this is done, the result will follow by Lemma \ref{lem:conju}. For $wH\in G/H$, define:
    \begin{equation}\label{eq:vsH}
        V_{wH}:=\sum_{x\in wH}\left(\tfrac{n-1}{n}e_{x,x}+e_{\plus,x}+e_{x,\plus}\right)-\sum_{x,y\in wH}\tfrac{1}{n}e_{x,y}+\sum_{g\in G\setminus wH}e_{g,g}\in \Gamma_{G_\plus}(\tilde{A})
    \end{equation}
    Note that $g\cdot V_{wH}=V_{gwH}$ for all $g\in G$. We can picture $V_{wH}$ as a block diagonal matrix having an identity block in the coordinates corresponding to $g\in G\setminus wH$, and the following block in the coordinates corresponding to elements of $wH\cup\{\plus\}$:
    \[
        \begin{tikzpicture}[mymatrixenv]
            \matrix[mymatrix] (m)  {
                \frac{n-1}{n} & -\frac{1}{n}  &
                \cdots        & -\frac{1}{n}  & -\frac{1}{n}  & [0.5em,between borders] 1                   \\
                -\frac{1}{n}  & \frac{n-1}{n} &
                \cdots        & -\frac{1}{n}  & -\frac{1}{n}  & 1                                           \\
                \vdots        &               & \ddots        &                           & \vdots & \vdots \\
                -\frac{1}{n}  & -\frac{1}{n}  &
                \cdots        & \frac{n-1}{n} & -\frac{1}{n}  & 1                                           \\
                -\frac{1}{n}  & -\frac{1}{n}  &
                \cdots        & -\frac{1}{n}  & \frac{n-1}{n} & 1                                           \\[0.7em]
                1             & 1             & \cdots        & 1                         & 1      & 0      \\
            };
            \draw[dashed]([xshift=0.8em]m.north west -| m-1-5.east) to ([xshift=0.8em] m.south east -| m-2-5.east);
            \draw[dashed]([xshift=-0.4em,yshift=0.7em)]m-6-1.north west) to ([yshift=0.7em,xshift=0.4em] m-6-6.north east);
            \mymatrixbraceright{1}{5}{$wH$}
            \mymatrixbracetop{1}{5}{$wH$}
        \end{tikzpicture}
    \]
    It is easily verified that $V_{gH}$ is invertible with inverse given by:
    \[V_{wH}^{-1}:=\sum_{x\in wH}\left(\tfrac{n-1}{n}e_{x,x}+\tfrac1ne_{\plus,x}+\tfrac1n e_{x,\plus}\right)-\sum_{x,y\in wH}\tfrac{1}{n}e_{x,y}+\sum_{g\in G\setminus wH}e_{g,g}\in \Gamma_{G_\plus}(\tilde{A})\]
    Again, we can think of $V_{wH}^{-1}$ as a block diagonal matrix having an identity block in the coordinates corresponding to $g\in G\setminus wH$, and the following block in the coordinates corresponding to elements of $wH\cup\{\plus\}$:
    \[
        \begin{tikzpicture}[mymatrixenv]
            \matrix[mymatrix] (m)  {
                \frac{n-1}{n} & -\frac{1}{n}  &
                \cdots        & -\frac{1}{n}  & -\frac{1}{n}  & [0.5em,between borders] \frac1n                    \\
                -\frac{1}{n}  & \frac{n-1}{n} &
                \cdots        & -\frac{1}{n}  & -\frac{1}{n}  & \frac1n                                            \\
                \vdots        &               & \ddots        &                                 & \vdots  & \vdots \\
                -\frac{1}{n}  & -\frac{1}{n}  &
                \cdots        & \frac{n-1}{n} & -\frac{1}{n}  & \frac1n                                            \\
                -\frac{1}{n}  & -\frac{1}{n}  &
                \cdots        & -\frac{1}{n}  & \frac{n-1}{n} & \frac1n                                            \\[0.7em]
                \frac1n       & \frac1n       & \cdots        & \frac1n                         & \frac1n & 0      \\
            };
            \draw[dashed]([xshift=0.8em]m.north west -| m-1-5.east) to ([xshift=0.8em] m.south east -| m-2-5.east);
            \draw[dashed]([xshift=-0.4em,yshift=0.7em)]m-6-1.north west) to ([yshift=0.7em,xshift=0.4em] m-6-6.north east);
            \mymatrixbraceright{1}{5}{$wH$}
            \mymatrixbracetop{1}{5}{$wH$}
        \end{tikzpicture}
    \]
    Define $V^{\pm 1}:G/H\to \Gamma_{G_\plus}(\tilde{A})$ by $V^{\pm 1}(wH):=V_{wH}^{\pm 1}$. Then $V,V^{-1}\in (\Gamma_{G_\plus}(\tilde{A}))^{G/H}$ are mutual inverses and we have:
    \[
        V(M_{G_\plus}\otimes A^{(G/H)}),(M_{G_\plus}\otimes A^{(G/H)})V^{-1}\subseteq (M_{G_\plus}\otimes A^{(G/H)})
    \]
    Moreover, $g\cdot V=V$ for all $g\in G$.  An easy computation shows that the triangle \eqref{eq:phiV} commutes. This finishes the proof.
\end{proof}

\begin{prop}[cf. \cite{euge}*{Theorems 5.2.1 and 6.14}]\label{prop:adj}
    Let $G$ be a group and let $H\subseteq G$ be a finite subgroup whose order is invertible in $\ell$. Then the morphisms $\eta_A\in kk(A,A^{(G/H)}\rtimes H)$ and $\varepsilon_B\in kk^G((B\rtimes H)^{(G/H)},B)$ defined above are respectively the unit and the counit of an adjunction:
    \begin{equation}\label{eq:adjnnat}
        \kk^G(A^{(G/H)}, B)\cong \kk(A, B\rtimes H)
    \end{equation}
\end{prop}
\begin{proof}
    It follows immediately from Lemmas \ref{lem:unit} and \ref{lem:counit}.
\end{proof}

%%%%%%%%%%%%%%%%%%%%%%%%%%%%%%%%%%%%%%%%%%%%%%%%%%%%%%%%%%%%%%%%%%%%%%%%%%%%%%%%%%%%%
%%%%%%%%%%%%%%%%%%%%%%%%%%%% Crossed product with G/H %%%%%%%%%%%%%%%%%%%%%%%%%%%%%%%
%%%%%%%%%%%%%%%%%%%%%%%%%%%%%%%%%%%%%%%%%%%%%%%%%%%%%%%%%%%%%%%%%%%%%%%%%%%%%%%%%%%%%
\section{Crossed product with \texorpdfstring{$G/H$}{G//H}}\label{sec:crossedpro}

Let $\OrGFin$ be the full subcategory of $\OrG$ (see Definition \ref{defi:OrG}) whose objects are those $G/H$ with finite $H$. We want to show that the adjunction \eqref{eq:adjnnat} is natural in $G/H$. The first thing to do is to replace the right-hand side of \eqref{eq:adjnnat} by an actual functor on $\OrGFin$.

\subsection{Non unital linear categories}\label{sec:lincat}

\begin{defi}
    A \emph{non-unital $\ell$-linear category} $\cC$ consists of:
    \begin{enumerate}
        \item a set of objects $\ob\cC$,
        \item an $\ell$-module $\cC(x,y)$ for every $x,y\in\ob\cC$, and
        \item $\ell$-module homomorphisms
              \begin{equation}\label{eq:compC} \circ:\cC(y,z)\otimes\cC(x,y)\to\cC(x,z) \end{equation}
              for every $x,y,z\in\ob\cC$, that are associative in the obvious way.
    \end{enumerate}
\end{defi}

Non-unital $\ell$-linear categories with only one object can be identified with (non-unital) algebras. In the sequel, we will refer to non-unital $\ell$-linear categories simply as \emph{linear categories}.

\begin{defi}
    Let $\cC$ and $\cD$ be linear categories. A \emph{linear functor} $F:\cC\to\cD$ consists of a function $F:\ob\cC\to\ob\cD$ together with $\ell$-module homorphisms
    \begin{equation}\label{eq:linearF} F_{x,y}:\cC(x,y)\to\cD(F(x),F(y)) \end{equation}
    that are compatible with the composition.
\end{defi}

We will write $\Catl$ for the category whose objects are linear categories and whose morphisms are linear functors.

\begin{exa}\label{exa:defiCrossedProduct}
    Let $B$ be a $G$-algebra and let $H\subseteq G$ be a subgroup. We proceed to define a linear category $B\rtimes G/H$. The objects of $B\rtimes G/H$ are the elements of $G/H$ and the hom-modules are defined as:
    \[(B\rtimes G/H)(uH, vH):=B\otimes \ell[vHu^{-1}]\]
    The composition law in $B\rtimes G/H$ is given by:
    \[
        \xymatrix@R=0em{
        (B\rtimes G/H)(vH, wH)\otimes (B\rtimes G/H)(uH, vH)\ar[r]^-{\circ} & (B\rtimes G/H)(uH, wH) \\
        (b\otimes g)\otimes (\tilde{b}\otimes \tilde{g})\ar@{|->}[r] & b(g\cdot \tilde{b})\otimes g\tilde{g}
        }
    \]
    In the terminology of \cite{corel}, $B\rtimes G/H$ is the crossed product of $B$ with the transport groupoid of $G/H$. It is easily verified that $B\rtimes G/H$ depends covariantly on $B$ and on $G/H$.
\end{exa}

Let $\cC$ and $\cD$ be linear categories. The \emph{tensor product} $\cC\otimes \cD$ is the linear category with objects $\ob(\cC)\times\ob(\cD)$ and such that:
\[(\cC\otimes\cD)((c,d),(\tilde{c}, \tilde{d}))=\cC(c,\tilde{c})\otimes\cD(d, \tilde{d})\] Composition is defined in the usual way, using the composition laws in $\cC$ and $\cD$ and the commutativity of the tensor product of $\ell$-modules.

We proceed to recall some definitions from \cite{corel}*{Section 3}. Let $\cC$ be a linear category. Put:
\[
    \cA(\cC)=\bigoplus_{x,y\in\ob \cC}\cC(x,y)
\]
If $f\in\cA(\cC)$, write $f_{y,x}$ for its component in $\cC(x,y)$. Then $\cA(\cC)$ is an algebra with multiplication given by:
\[
    (gf)_{y,x}=\sum_{z\in\ob\cC}g_{y,z}\circ f_{z,x}
\]

\begin{exa}
    Let $\cC$ and $\cD$ be linear categories. It is easily verified that:
    \[\cA(\cC\otimes\cD)\cong\cA(\cC)\otimes\cA(\cD)\]
\end{exa}

\begin{exa}\label{exa:subalgebraMGH}
    Let $B$ be a $G$-algebra and let $H\subseteq G$ be a subgroup.
    We can regard $\cA\left(B\rtimes G/H\right)$ as a subalgebra of $M_{|G/H|}(B\rtimes G)$ using the inclusion that sends $b\otimes g\in(B\rtimes G/H)(uH, vH)$ to $e_{vH,uH}\otimes (b\rtimes g)$.
\end{exa}

A disadvantage of $\cA(\cC)$ is that it is not natural with respect to all linear functors, but only with respect to those that are injective on objects; see \cite{corel}*{p.1231}. To fix this, one defines the algebra $\cR(\cC)$ \cite{corel}*{Section 3.4}. If $M$ is an $\ell$-module, write $T(M)=\oplus_{n\geq 1}M^{\otimes n}$ for the unaugmented tensor algebra. Put:
\[
    \cR(\cC)=T\left(\cA(\cC)\right)/\langle\left\{g\otimes f-g\circ f: f\in\cC(x,y),\, g\in\cC(y,z),\, x,y,z\in\ob\cC\right\}\rangle
\]
This defines a functor $\cR:\Catl\to\Algl$.

\begin{rem}
    Let $G$ be a group. One can define a \emph{$G$-category} as a linear category $\cC$ such that the hom-modules $\cC(x,y)$ are $G$-modules and the composition law \eqref{eq:compC} is $G$-equivariant, endowing $\cC(y,z)\otimes\cC(x,y)$ with the diagonal $G$-action. If $\cC$ and $\cD$ are $G$-categories, a \emph{$G$-functor} $F:\cC\to\cD$ is a linear functor such that the morphisms \eqref{eq:linearF} are $G$-equivariant. These definitions give rise to a category $\GCatl$ whose objects are $G$-categories and whose morphisms are $G$-functors.

    If $\cC$ is a $G$-category, then $\cA(\cC)$ and $\cR(\cC)$ are $G$-algebras in a natural way. Thus, we have a functor $\cR:\GCatl\to\GAlgl$.
\end{rem}

\begin{exa}\label{exa:rtensor}
    Let $\cC$ be a linear category and let $D$ be an algebra. We claim that there is a natural morphism:
    \[\cR(\cC\otimes D)\rightarrow \cR(\cC)\otimes D\]
    To see this, first note that there is a linear functor $\cC\to\cR(\cC)$ that takes $f\in\cC(x,y)$ to the class in $\cR(\cC)$ of $f\in\cC(x,y)\subseteq \cA(\cC)\subseteq T(\cA(\cC))$. Upon tensoring this functor with $D$ and then applying $\cR(-)$, we get the desired morphism:
    \[
        \cR(\cC\otimes D)\to \cR(\cR(\cC)\otimes D)=\cR(\cC)\otimes D
    \]
    If $\cC$ is a $G$-category and $D$ is a $G$-algebra, then this is a morphism of $G$-algebras.
\end{exa}

For a linear category $\cC$, there is a morphism $p:\cR(\cC)\to\cA(\cC)$ induced by multiplication in $\cA(\cC)$.

\begin{lem}[cf. \cite{corel}*{Lemma 3.4.3}]\label{lem:piso} Let $\cC$ be a linear category (respectively a $G$-category). Then the morphism
    \[p:\cR(\cC)\to\cA(\cC)\]
    induces an isomorphism in $\kk$ (resp. in $\kk^G$).
\end{lem}
\begin{proof}
    The proof of \cite{corel}*{Lemma 3.4.3} carries on verbatim in this setting.
\end{proof}

\begin{coro}\label{coro:rtensor}
    Let $\cC\in\Catl$ and $D\in\Algl$ (resp. $\cC\in\GCatl$ and $D\in\GAlgl$). Then the morphism
    \[\cR(\cC\otimes D)\rightarrow \cR(\cC)\otimes D\]
    of Example \ref{exa:rtensor} is an isomorphism in $\kk$ (resp. in $\kk^G$).
\end{coro}
\begin{proof}
    It is easily verified that the following diagram commutes in  $\Algl$ (resp. in $\GAlgl$):
    \[\xymatrix{\cR(\cC\otimes D)\ar[r]\ar[d]_-{p} & \cR(\cC)\otimes D\ar[d]^-{p\otimes D} \\
        \cA(\cC\otimes D)\ar[r]^-{\cong} & \cA(\cC)\otimes D}\]
    Indeed, it suffices to check commutativity on the generators $f\otimes d$, and this is immediate. The result follows from Lemma \ref{lem:piso}.
\end{proof}

\subsection{Crossed product with \texorpdfstring{$G/H$}{G/H}}\label{sec:crospro}
Fix $G/H\in\OrG$ and recall the definition of $B\rtimes G/H$ from Example \ref{exa:defiCrossedProduct}. We claim that the composite functor
\begin{equation}\label{eq:compFunc}
    \xymatrix@C=6em{\GAlgl \ar[r]^-{-\rtimes G/H} & \Catl\ar[r]^-{\cR} & \Algl\ar[r]^-{j} & \kk}
\end{equation}
factors through $j^G:\GAlgl\to\kk^G$. To prove this, it suffices to show that it is excisive, homotopy invariant and $G$-stable \cite{euge}*{Theorem 4.1.1}. Homotopy invariance and $G$-stability follow easily from the following.

\begin{lem}\label{lem:natisocompoGH}
    Let $G$ be a group, let $H\subseteq G$ be a subgroup and let $uH\in G/H$. Then there is a natural isomorphism
    \begin{equation}\label{eq:jrtimesH}
        \xymatrix{
        \nu_{uH}:j(-\rtimes uHu^{-1})\ar[r]^-{\cong} & j\cR(-\rtimes G/H)
        }
    \end{equation}
    of functors $\GAlgl\to\kk$.
\end{lem}
\begin{proof}
    Let $B\in\GAlgl$. Consider $B\rtimes uHu^{-1}\subseteq B\rtimes G/H$ as the full subcategory whose only object is $uH$. Upon applying $\cR$ to this inclusion, we get an algebra homomorphism:
    \[
        B\rtimes uHu^{-1}=\cR(B\rtimes uHu^{-1})\longrightarrow \cR\left(B\rtimes G/H\right)
    \]
    Let $\nu_{uH}$ be the image of this morphism in $\kk$. Clearly, $\nu_{uH}$ is a natural transformation $j(-\rtimes uHu^{-1})\to j\cR(-\rtimes G/H)$. To finish the proof, we will show that $\nu_{uH}$ is an isomorphism. We claim that there is an isomorphism $\alpha$ that fits in the following commutative diagram in $\kk$:
    \[
        \xymatrix@C=3em{
        B\rtimes uHu^{-1}\ar[r]^-{\nu_{uH}}\ar@/_3pc/[ddr]_-{e_{uH,uH}\otimes -} & \cR\left(B\rtimes G/H\right)\ar[d]^-{p}_-{\cong} \\
        & \cA\left(B\rtimes G/H\right)\ar[d]_-{\cong}^-{\alpha} \\
        & M_{|G/H|}(B\rtimes uHu^{-1})
        }
    \]
    Here, the bent arrow is induced by the inclusion into the $(uH,uH)$-coefficient and it is an isomorphism by matrix invariance. It follows that $\nu_{uH}$ is an isomorphism too. The isomorphism $\alpha$ is constructed as in the proof of \cite{corel}*{Lemma 3.2.6}. More precisely, let $s:G/H\to G$ be a section of the projection such that $s(uH)=u$. Write $\hat{g}=s(gH)$ for $g\in G$. For $b\otimes g\in(B\rtimes G/H)(sH,tH)$ put:
    \[\alpha(b\otimes g)=e_{tH,sH}\otimes ((u\hat{t}^{-1})\cdot b)\rtimes u\hat{t}^{-1}g\hat{s}u^{-1}\]
    It is straightforward to verify that this formula defines an isomorphism of algebras $\alpha:\cA\left(B\rtimes G/H\right)\to M_{|G/H|}(B\rtimes uHu^{-1})$.
\end{proof}

\begin{coro}[cf. \cite{euge}*{Proposition 5.1.2 and Section 6}]\label{coro:rtimesGHhtpyinv}
    Let $G$ be a group and let $H\subseteq G$ be a subgroup. Then the functor
    \[j\cR(-\rtimes G/H):\GAlgl\to\kk\]
    is homotopy invariant and $G$-stable.
\end{coro}
\begin{proof}
    Write $F:\GAlgl\to\kk$ for the functor $F=j(-\rtimes H)$. By Lemma \ref{lem:natisocompoGH} it suffices to show that $F$ is homotopy invariant and $G$-stable. The functor $-\rtimes H:\GAlgl\to\Algl$ is easily seen to send homotopic morphisms in $\GAlgl$ to homotopic morphisms in $\Algl$. It follows that $F$ is homotopy invariant. Recall the definition of $G$-stable functor from \cite{euge}*{Section 3.1}. Let $(\cW_1, B_1)$ and $(\cW_2, B_2)$ be $G$-modules by locally finite automorphisms such that $\card(B_i)\leq \card(\N\times G)$ for $i=1,2$ and let $A$ be a $G$-algebra. Then the inclusion
    \[
        \left(A\otimes \End_{\ell}^F(\cW_1)\right)\rtimes H\overset{\tilde{\iota}}{\to}\left(A
        \otimes \End_{\ell}^F(\cW_1\oplus \cW_2)\right)\rtimes H
    \]
    is identified with
    \[
        (A\rtimes H)\otimes \End_{\ell}^F(\cW_1)\overset{\tilde{\iota}}{\to}(A\rtimes H)
        \otimes \End_{\ell}^F(\cW_1\oplus \cW_2)
    \]
    by \cite{euge}*{Proposition 5.1.1}.
    The latter becomes an isomorphism upon applying $j$ by $M_{\N\times|G|}$-stability of $\kk$. It follows that $F$ is $G$-stable.
\end{proof}

We still have to show that the composite of the functors in \eqref{eq:compFunc} is excisive.

\begin{cons}\label{cons:partial}
    Let $\cE:A\to B\to C$ be an extension in $\GAlgl$ (that splits in $\GModl$). We will to construct a triangle in $\kk$ as follows, that is natural with respect to morphisms of extensions:
    \[
        \xymatrix@C=2.7em{
        \Omega(\cR(C\rtimes G/H))\ar[r]^-{\partial_{\cE\rtimes G/H}} & \cR(A\rtimes G/H)\ar[r] & \cR(B\rtimes G/H)\ar[r] & \cR(C\rtimes G/H)
        }
    \]
    Let $\nu_H:j(-\rtimes H)\to j\cR(-\rtimes G/H)$ be the natural isomorphism \eqref{eq:jrtimesH}. Note that $j(-\rtimes H)$ is excisive, since $-\rtimes H:\GAlgl\to\Algl$ preserves extensions. To simplify notation, we omit explicit mention to $j$ for the rest of the proof. We have the following commutative diagram of solid arrows in $\kk$, where the top row is a triangle:
    \[
        \xymatrix@C=2.7em{
        \Omega(C\rtimes H)\ar[r]^-{\partial_{\cE\rtimes H}}\ar[d]_-{\cong}^-{\Omega(\nu_H)} & A\rtimes H\ar[d]_-{\cong}^-{\nu_H}\ar[r] & B\rtimes H\ar[d]_-{\cong}^-{\nu_H}\ar[r] & C\rtimes H\ar[d]_-{\cong}^-{\nu_H} \\
        \Omega(\cR(C\rtimes G/H))\ar@{-->}[r] & \cR(A\rtimes G/H)\ar[r] & \cR(B\rtimes G/H)\ar[r] & \cR(C\rtimes G/H)
        }
    \]
    Define $\partial_{\cE\rtimes G/H}:\Omega(\cR(C\rtimes G/H))\to \cR(A\rtimes G/H)$ to be the dashed arrow that makes the left square commute. Then the bottom row becomes a triangle too. This triangle is clearly natural with respect to the extension $\cE$. These morphisms $\partial_{\cE\rtimes G/H}$ make $j\cR(-\rtimes G/H):\GAlgl\to\kk$ into an excisive homology theory.
\end{cons}

\begin{prop}\label{prop:existencertimesGH}
    Let $G$ be a group and let $H\subseteq G$ be a subgroup. Then there exists a unique triangulated functor $\overline{-\rtimes G/H}:\kk^G\to\kk$ making the following diagram commute:
    \[
        \xymatrix@C=6em{
        \GAlgl\ar[d]_-{j^G}\ar[r]^-{\cR(-\rtimes G/H)} & \Algl\ar[d]^-{j} \\
        \kk^G\ar[r]_-{\overline{-\rtimes G/H}} & \kk
        }
    \]
    Moreover, for every extension $\cE:A\to B\to C$ in $\GAlgl$ and every $uH\in G/H$, the following square in $\kk$ commutes:
    \begin{equation}\label{eq:squarepartial}\begin{gathered}
            \xymatrix@C=4em{
            \Omega(C\rtimes uHu^{-1})\ar[d]_-{\Omega(\nu_{uH})}^-{
            \scriptsize{{\begin{array}{@{}c@{}}
                            \cong \\
                            \eqref{eq:jrtimesH}
                        \end{array}}}
            }\ar[r]^-{\partial_{\cE \rtimes uHu^{-1}}} & A\rtimes uHu^{-1}\ar[d]^-{\nu_{uH}}_-{
            \scriptsize{{\begin{array}{@{}c@{}}
                            \cong \\
                            \eqref{eq:jrtimesH}
                        \end{array}}}
            } \\
            \Omega(\cR(C\rtimes G/H))\ar[r]^-{\partial_{\cE\rtimes G/H}} & \cR(A\rtimes G/H)
            }
        \end{gathered}\end{equation}
\end{prop}
\begin{proof}
    The functor $j\cR(-\rtimes G/H):\GAlgl\to\kk$ is homotopy invariant and $G$-stable by Corollary \ref{coro:rtimesGHhtpyinv}. Moreover, endowed with the morphisms $\partial_{\cE\rtimes G/H}$ defined in Construction \ref{cons:partial}, it becomes an excisive homology theory. Then the existence of $\overline{-\rtimes G/H}:\kk^G\to\kk$ follows from \cite{euge}*{Theorem 4.1.1}.

    To prove the assertion about \eqref{eq:squarepartial}, consider the following diagram in $\kk$:
    \begin{equation}\label{eq:partialuH}\begin{gathered}
            \xymatrix@C=4em{\Omega(C\rtimes uHu^{-1})\ar[d]_-{\Omega(\nu_{uH})}^-{\cong}\ar[r]^-{\partial_{\cE \rtimes uHu^{-1}}} & A\rtimes uHu^{-1}\ar[d]^-{\nu_{uH}}_-{\cong} \\
            \Omega(\cR(C\rtimes G/H))\ar[r]^-{\partial_{\cE\rtimes G/H}}\ar[d]_-{\Omega(\nu_H)^{-1}}^-{\cong} & \cR(A\rtimes G/H)\ar[d]^-{(\nu_H)^{-1}}_-{\cong} \\
            \Omega(C\rtimes H)\ar[r]^-{\partial_{\cE\rtimes H}} & A\rtimes H}
        \end{gathered}\end{equation}
    The bottom square commutes by definition of $\partial_{\cE\rtimes G/H}$; see Construction \ref{cons:partial}. Thus, the commutativity of \eqref{eq:squarepartial} is equivalent to that of the outer square in \eqref{eq:partialuH}. We will see that the latter commutes since it is induced by a morphism of extensions in $\Algl$. For $D\in\GAlgl$, let $\varphi_u:D\rtimes uHu^{-1}\to D\rtimes H$ be the algebra homomorphism defined by $\varphi_u(a\rtimes g)=(u^{-1}\cdot a)\rtimes u^{-1}gu$. We have a morphism of extensions
    \[
        \xymatrix{A\rtimes uHu^{-1}\ar[d]^-{\varphi_u}\ar[r] & B\rtimes uHu^{-1}\ar[d]^-{\varphi_u}\ar[r] & C\rtimes uHu^{-1}\ar[d]^-{\varphi_u} \\
        A\rtimes H\ar[r] & B\rtimes H\ar[r] & C\rtimes H}
    \]
    that induces a morphism of triangles in $\kk$. In particular, there is a commutative square in $\kk$ as follows:
    \[\xymatrix@C=4em{\Omega(C\rtimes uHu^{-1})\ar[d]_-{\Omega j(\varphi_u)}\ar[r]^-{\partial_{\cE\rtimes uHu^{-1}}} & A\rtimes uHu^{-1}\ar[d]^-{j(\varphi_u)} \\
        \Omega(C\rtimes H)\ar[r]^-{\partial_{\cE\rtimes H}} & A\rtimes H}\]
    This square turns out to be the outer square of \eqref{eq:partialuH}. To see this, it suffices to show that $j(\varphi_u)=(\nu_H)^{-1}\circ \nu_{uH}$ or, equivalently, that
    \begin{equation}\label{eq:nuHnuuH}j(p)\circ\nu_H\circ j(\varphi_u)=j(p)\circ\nu_{uH}
    \end{equation}
    where $p:\cR\left(D\rtimes G/H\right)\longrightarrow\cA\left(D\rtimes G/H\right)$ is the morphism of Lemma \ref{lem:piso}. Each side of \eqref{eq:nuHnuuH} is induced by an algebra homomorphism $D\rtimes uHu^{-1}\to \cA\left(D\rtimes G/H\right)$. We will show that both morphisms are conjugate in $M_{|G/H|}(D\rtimes G)$, regarding $\cA\left(D\rtimes G/H\right)$ as a subalgebra of $M_{|G/H|}(D\rtimes G)$ with the inclusion of Example \ref{exa:subalgebraMGH}. A straightforward verification shows that the left- and the right-hand sides of \eqref{eq:nuHnuuH} are induced, respectively, by the algebra homomorphisms $\lambda$ and $\rho$ defined by:
    \begin{align*}
        \lambda(d\rtimes g) & =e_{H,H}\otimes (u^{-1}\cdot d)\rtimes u^{-1}gu \\
        \rho(d\rtimes g)    & =e_{uH, uH}\otimes d\rtimes g
    \end{align*}
    Now put:
    \[V=\sum_{vH\in G/H}e_{uvH, vH}\otimes 1\rtimes u \in\Gamma_{|G/H|}(\tilde{D}\rtimes G)\supset M_{|G/H|}(D\rtimes G) \]
    It is easily seen that $\rho=V\lambda V^{-1}$. This proves the equality \eqref{eq:nuHnuuH} and concludes the proof of the proposition.
\end{proof}

\begin{prop}
    Let $G$ be a group. Then every morphism $G/H\to G/K$ in $\OrG$ induces a natural transformation $\overline{-\rtimes G/H}\to \overline{-\rtimes G/K}$ of triangulated functors $\kk^G\to\kk$. Moreover, these assemble into a functor $-\rtimes -:\kk^G\times \OrG\to \kk$.
\end{prop}
\begin{proof}
    Let $f:G/H\to G/K$ be a morphism in $\OrG$. Clearly, $f$ induces a natural transformation $j\cR(-\rtimes G/H)\to j\cR(-\rtimes G/K)$ of functors $\GAlgl\to\kk$. We will prove that this natural transformation is compatible with the excisive homology theory structures. More precisely, let $\cE:A\to B\to C$ be an extension in $\GAlgl$. We will prove that the following square in $\kk$ commutes:
    \begin{equation}\label{eq:deltafunc}\begin{gathered}
            \xymatrix@C=4em{\Omega(\cR(C\rtimes G/H))\ar[d]_-{f_*}\ar[r]^-{\partial_{\cE\rtimes G/H}} & \cR(A\rtimes G/H)\ar[d]^-{f_*} \\
            \Omega(\cR(C\rtimes G/K))\ar[r]^-{\partial_{\cE\rtimes G/K}} & \cR(A\rtimes G/K)}
        \end{gathered}
    \end{equation}
    Supose for a moment that this square commutes. Put $\scrA=\kk^I$ where $I$ is the interval category. Then $f$ induces a functor $\gamma_f:\GAlgl\to\scrA$ defined by:
    \[D\mapsto \left(f_*:\cR(D\rtimes G/H)\to \cR(D\rtimes  G/K) \right)\]
    The commutativity of \eqref{eq:deltafunc} implies that $\gamma_f$ is a homotopy invariant and $G$-stable $\delta$-functor in the sense of Definition \ref{def:deltafunctor}. Thus, it factors uniquely through $\kk^G$ by the universal property of $\kk^G$ stated in Theorem \ref{thm:molesto}:
    \[\xymatrix{\GAlgl\ar[r]^-{j^G}\ar@/_/[dr]_-{\gamma_f} & \kk^G\ar@{-->}[d]^-{\exists !\,\bar{\gamma}_f} \\
        & \scrA}\]
    The functor $\bar{\gamma}_f$ corresponds to the desired natural transformation.
    Let us now show that \eqref{eq:deltafunc} commutes. The morphism $f:G/H\to G/K$ is determined by $f(H)=uK$ for some $u\in G$ with $H\subseteq uKu^{-1}$. We have the following commutative square of linear categories:
    \[\xymatrix@C=4em{A\rtimes G/H\ar[r]^-{f_*} & A\rtimes G/K \\
        A\rtimes H\ar[u]\ar[r]^-{\incl} & A\rtimes uKu^{-1}\ar[u]}\]
    The bottom arrow is an inclusion of algebras. The left and right arrows are the inclusions of the full subcategories whose only objects are $H$ and $uK$, respectively. Upon applying $j\circ\cR:\Catl\to\kk$ the vertical arrows become isomorphisms and we get:
    \[f_*=\nu_{uK}\circ j(\incl)\circ (\nu_H)^{-1}:\cR(A\rtimes G/H)\longrightarrow \cR(A\rtimes G/K)\]
    Thus, the commutativity of \eqref{eq:deltafunc} is equivalent to that of the outer square in the following diagram in $\kk$:
    \[\xymatrix@C=4em{
        \Omega(\cR(C\rtimes G/H))\ar[d]_-{\Omega(\nu_H)^{-1}}^-{\cong}\ar[r]^-{\partial_{\cE\rtimes G/H}} & \cR(A\rtimes G/H)\ar[d]_-{\cong}^-{(\nu_H)^{-1}} \\
        \Omega(C\rtimes H)\ar[r]^-{\partial_{\cE\rtimes H}}\ar[d]_-{\Omega j(\incl)} & A\rtimes H\ar[d]^-{j(\incl)} \\
        \Omega(C\rtimes uKu^{-1})\ar[d]_-{\Omega(\nu_{uK})}^-{\cong}\ar[r]^-{\partial_{\cE\rtimes uKu^{-1}}} & A\rtimes uKu^{-1}\ar[d]^-{\nu_{uK}}_-{\cong} \\
        \Omega(\cR(C\rtimes G/K))\ar[r]^-{\partial_{\cE\rtimes G/K}} & \cR(A\rtimes G/K)
        }\]
    Here, the bottom and top squares commute by Proposition \ref{prop:existencertimesGH}. The middle square commutes because it fits into the morphism of triangles induced by the inclusion $H\subseteq uKu^{-1}$.
\end{proof}

%%%%%%%%%%%%%%%%%%%%%%%%%%%%%%%%%%%%%%%%%%%%%%%%%%%%%%%%%%%%%%%%%%%%%%%%%%%%%%%%%%%%%
%%%%%%%%%%%%%%%%%%%%%%%%%%%% A natural adjunction %%%%%%%%%%%%%%%%%%%%%%%%%%%%%%%%%%%
%%%%%%%%%%%%%%%%%%%%%%%%%%%%%%%%%%%%%%%%%%%%%%%%%%%%%%%%%%%%%%%%%%%%%%%%%%%%%%%%%%%%%
\section{A natural adjunction}\label{sec:natadj}

Let $G$ be a group and let $H$ be a finite subgroup of $G$. Recall from Lemma \ref{lem:natisocompoGH} that there is a natural isomorphism $\nu_{H}:-\rtimes H\to \overline{-\rtimes G/H}$ of functors $\kk^G\to\kk$. If $|H|$ is invertible in $\ell$, we have isomorphisms
\begin{equation}\label{eq:natadjdefi}\kk^G(A^{(G/H)}, B)\cong \kk(A,B\rtimes H)\cong \kk(A, \cR(B\rtimes G/H)),\end{equation}
where the isomorphism on the right is induced by $\nu_H$ and the one on the left is that of Proposition \ref{prop:adj}. In other words, there is an adjunction:
\begin{equation}\label{eq:natadj}\begin{gathered}\begin{tikzcd}
            (-)^{(G/H)}:\kk \arrow[r,shift left]
            &
            \kk^G:\overline{-\rtimes G/H} \arrow[l,shift left]
        \end{tikzcd}\end{gathered}\end{equation}
We will show that this adjunction is natural in $G/H$. For the rest of this paper, we assume that $G$ satisfies the following property: \begin{equation}\label{eq:hipG}\text{$|H|$ is invertible in $\ell$ for every }H\in\FIN\end{equation}

We will prove that for every morphism $f:G/H\to G/K\in\OrGFin$ the following square commutes:
\begin{equation}\label{eq:natadjsq}\begin{gathered}\xymatrix{\kk(A,\cR(B\rtimes G/H))\ar[r]^-{\cong}\ar[d]_-{f_*} & \kk^G(A^{(G/H)}, B)\ar[d]^-{f_*} \\
        \kk(A, \cR(B\rtimes G/K))\ar[r]^-{\cong} & \kk^G(A^{(G/K)}, B)}\end{gathered}\end{equation}
The commutativity of this square is not obvious a priori since the middle term of \eqref{eq:natadjdefi} is not a functor on $\OrGFin$. Let $\alpha\in\kk(A, \cR(B\rtimes G/H))$ and $f:G/H\to G/K$ in $\OrGFin$. The commutativity of \eqref{eq:natadjsq} is equivalent to that of the outer square in the following diagram in $\kk$, where $\epsilon_{G/H}$ is the counit of \eqref{eq:natadj}:
\[
    \xymatrix@C=4em{
    A^{(G/H)}\ar[r]^-{\alpha^{(G/H)}} & \left[\cR(B\rtimes G/H)\right]^{(G/H)}\ar[r]^-{\epsilon_{G/H}} & B \\
    A^{(G/K)}\ar[u]^-{A^{(f)}}\ar[r]_-{\alpha^{(G/K)}} & \left[\cR(B\rtimes G/H)\right]^{(G/K)}\ar[u]_-{f^*}\ar[r]_-{f_*} & \left[\cR(B\rtimes G/K)\right]^{(G/K)}\ar[u]_-{\epsilon_{G/K}}}
\]
The square on the left clearly commutes. The commutativity of the square on the right follows easily from the following result.

\begin{lem}\label{lem:natcounit}
    Let $G$ be a group satisfying \eqref{eq:hipG}. Then:
    \begin{enumerate}
        \item\label{item:defizetaGH} For every $G/H\in\OrGFin$ there is a $G$-functor
              \begin{equation}\label{eq:defiZeta}
                  \zeta_{G/H}:(B\rtimes G/H)^{(G/H)}\to M_G\otimes B\end{equation}
              that sends $(b\otimes g)\chi_{sH}\in \Hom(uH,vH)=B\otimes\ell[vHu^{-1}]^{(G/H)}$ to:
              \[\sum_{p\in sHv^{-1}}e_{p,pg}\otimes (p\cdot b)\]
        \item\label{item:zetacompat} For any $f:G/H\to G/K$ in $\OrGFin$, the following square in $\GCatl$ commutes:
              \begin{equation}\label{eq:zetacompat}\begin{gathered}\xymatrix{(B\rtimes G/H)^{(G/H)}\ar[r]^-{\zeta_{G/H}} & M_G\otimes B\\
                      (B\rtimes G/H)^{(G/K)}\ar[u]_-{f^*}\ar[r]^-{f_*} & (B\rtimes G/K)^{(G/K)}\ar[u]_-{\zeta_{G/K}}}\end{gathered}\end{equation}
        \item\label{item:couniadj} The counit $\epsilon_{G/H}$ of the adjunction \eqref{eq:natadj} fits into the following commutative diagram in $\kk^G$:
              \[\xymatrix@C=5em{
                  \left[\cR\left(B\rtimes G/H\right)\right]^{(G/H)}\ar@/_1pc/[ddr]_-{\epsilon_{G/H}} & \cR\left[\left(B\rtimes G/H\right)^{(G/H)}\right]\ar[d]^-{\cR(\zeta_{G/H})}\ar[l]_-{\text{Cor. \ref{coro:rtensor}}}^-{\cong} \\
                  & M_G\otimes B\ar@{--}[d]^-{\eqref{eq:zigzag}}_-{\cong} \\
                  & B}
              \]
    \end{enumerate}
\end{lem}
\begin{proof}
    Let us verify that $\zeta_{G/H}$ as defined in \eqref{item:defizetaGH} is compatible with composition. Let $f_1=(b_1\otimes g_1)\chi_{sH}\in \Hom(uH, vH)$ and $f_2=(b_2\otimes g_2)\chi_{tH}\in \Hom(vH, wH)$. We have
    \[f_2\circ f_1=\delta_{sH,tH} (b_2(g_2\cdot b_1)\otimes g_2g_1)\chi_{tH}\]
    where $\delta_{sH,tH}$ is Kronecker's delta. Then:
    \[\zeta_{G/H}(f_2\circ f_1)=\delta_{sH,tH}\sum_{p\in tHw^{-1}}e_{p,pg_2g_1}\otimes p\cdot (b_2(g_2\cdot b_1))\]
    On the other hand, we have:
    \begin{align*}\zeta_{G/H}(f_2)\zeta_{G/H}(f_1) & =\left(\sum_{q\in tHw^{-1}}e_{q, qg_2}\otimes (q\cdot b_2)\right)\left(\sum_{p\in sHv^{-1}}e_{p, pg_1}\otimes (p\cdot b_1)\right) \\
                                               & =\sum_{\substack{p\in sHv^{-1}                                                                                                    \\ q\in tHw^{-1}}}e_{q,qg_2}e_{p,pg_1}\otimes (q\cdot b_2)(p\cdot b_1)\\
                                               & =\delta_{sH,tH}\sum_{q\in tHw^{-1}}e_{q,qg_2g_1}\otimes (q\cdot b_2)((qg_2)\cdot b_1)\end{align*}
    The appearance of Kronecker's delta in the last line is explained as follows: $qg_2\in tHv^{-1}$ and $p\in sHv^{-1}$ can be equal if and only if $sH=tH$. This shows that $\zeta_{G/H}$ is indeed a well-defined functor.

    Let us prove \eqref{item:zetacompat}. The square commutes on objects since $M_G\otimes B$ has only one object. A morphism $f:G/H\to G/K$ is determined by $f(H)=xK$ with $x$ such that $H\subseteq xKx^{-1}$. Let $(b\otimes g)\chi_{sK}\in \Hom(uH,vH)$. We have:
    \begin{align*}
        \zeta_{G/K}\left(f_*\left((b\otimes g)\chi_{sK}\right)\right) & =\zeta_{G/K}\left((b\otimes g)\chi_{sK}\right)     \\
                                                                      & =\sum_{p\in sK(vx)^{-1}}e_{p,pg}\otimes (p\cdot b)
    \end{align*}
    On the other hand, we have:
    \begin{align*}
        \zeta_{G/H}\left(f^*\left((b\otimes g)\chi_{sK}\right)\right) & =\zeta_{G/H}\left(\sum_{tH\in f^{-1}(sK)}(b\otimes g)\chi_{tH}\right)  \\
                                                                      & =\sum_{tH\in f^{-1}(sK)}\sum_{p\in tHv^{-1}}e_{p,pg}\otimes (p\cdot b) \\
                                                                      & =\sum_{p\in sKx^{-1}v^{-1}}e_{p,pg}\otimes (p\cdot b)
    \end{align*}
    Here, the last equality follows from the fact that $sKx^{-1}$ is the disjoint union of $f^{-1}(sK)$. This finishes the proof of \eqref{item:zetacompat}.

    Let us prove \eqref{item:couniadj}. It follows from \eqref{eq:natadjdefi} that the counit $\epsilon_{G/H}$ of the adjunction \eqref{eq:natadj} equals $\varepsilon_H\circ \left[(\nu_H)^{-1}\right]^{(G/H)}$, where $\varepsilon_H$ is the counit of \eqref{eq:adjnnat}. Recall the definition of $\varepsilon_H$ from \eqref{eq:couniH}. Then $\epsilon_{G/H}$ equals the following composite in $\kk$:
    \[\xymatrix@C=4.3em{\left[\cR(B\rtimes G/H)\right]^{(G/H)}\ar[r]^-{\left[(\nu_H)^{-1}\right]^{(G/H)}} & (B\rtimes H)^{(G/H)}\ar[r]^-{\counit_B}_-{\text{\eqref{eq:defipsi}}} & M_G\otimes B\underset{\eqref{eq:zigzag}}{\cong} B}\]
    The statement in \eqref{item:couniadj} now follows from the commutativity of the following diagram of $G$-algebras:
    \[\xymatrix@C=5em{\left[\cR(B\rtimes G/H)\right]^{(G/H)} & \cR\left[(B\rtimes G/H)^{(G/H)}\right]\ar[l]_-{\text{Cor. \ref{coro:rtensor}}}\ar[d]^-{\cR(\zeta_{G/H})} \\
        (B\rtimes H)^{(G/H)}\ar[r]^-{\counit_B}\ar[u]^-{(\nu_H)^{(G/H)}}\ar[ur] & M_G\otimes B}\]
    Here the diagonal morphism is the one induced by the inclusion
    \[(B\rtimes H)^{(G/H)}\subseteq (B\rtimes G/H)^{(G/H)}\]
    as the full subcategory on the object $H$. This finishes the proof.
\end{proof}

We collect the main results of this section in the following theorem.

\begin{thm}\label{thm:adj}
    Let $G$ be a group satisfying \eqref{eq:hipG} and let $H\subseteq G$ be a finite subgroup. Then we have an adjunction:
    \[\begin{tikzcd}
            (-)^{(G/H)}:\kk \arrow[r,shift left]
            &
            \kk^G:\overline{-\rtimes G/H} \arrow[l,shift left]
        \end{tikzcd}\]
    Moreover, the adjunction isomorphism
    \[\kk^G(A^{(G/H)}, B)\cong \kk(A, \cR(B\rtimes G/H))\]
    is natural in $G/H\in\OrGFin$ and the counit of the adjunction is described by Lemma \ref{lem:natcounit}.
\end{thm}

%%%%%%%%%%%%%%%%%%%%%%%%%%%%%%%%%%%%%%%%%%%%%%%%%%%%%%%%%%%%%%%%%%%%%%%%%%%%%%%%%%%%%
%%%%%%%%%%%%%%%%%%%%%%%%% Lifting the adjunction to spectra %%%%%%%%%%%%%%%%%%%%%%%%%
%%%%%%%%%%%%%%%%%%%%%%%%%%%%%%%%%%%%%%%%%%%%%%%%%%%%%%%%%%%%%%%%%%%%%%%%%%%%%%%%%%%%%
\section{Lifting the adjunction to spectra}\label{sec:bomba}

\subsection{The primitive zig-zag}\label{sec:zzagp}

Let $A\in\Algl$, $B\in\GAlgl$ and $G/N\in\OrGFin$. By Theorem \ref{thm:adj} we have an adjunction isomorphism
\begin{equation}\label{eq:natAdjGroups}\kk^G(A^{(G/N)}, B)\cong \kk(A, \cR(B\rtimes G/N))\end{equation}
that is natural in $G/N$.
Let $\bbK$ and $\bbK^G$ be, respectively, the spectra representing $\kk$-theory and $\kk^G$-theory. These were defined in \cite{garku} and in \cite{tesisema}; see Section \ref{sec:sptkkg} for details. We would like to lift the isomorphism \eqref{eq:natAdjGroups} to a natural weak equivalence of spectra:
\[\xymatrix{ \bbK(A, \cR(B\rtimes G/N))\ar@{-->}[r] & \bbK^G(A^{(G/N)}, B)}
\]
Here, we want the dashed arrow to represent a zig-zag of $\OrGFin$-spectra inducing the isomorphism \eqref{eq:natAdjGroups} upon taking homotopy groups. As a starting point, let us recall how to obtain this adjunction. By Lemma \ref{lem:natcounit} \eqref{item:couniadj}, the isomorphism \eqref{eq:natAdjGroups} equals the following composite:
\begin{equation}\label{eq:decompAdj}\begin{gathered}\xymatrix@C=3.5em{
        \kk(A, \cR(B\rtimes G/N))\ar[r]^-{(-)^{(G/N)}} & \kk^G(A^{(G/N)}, \left[\cR(B\rtimes G/N)\right]^{(G/N)})\\
        \kk^G(A^{(G/N)}, M_{G}B)\ar[d]_-{\rotatebox{90}{$\scriptstyle\cong$}} & \kk^G(A^{(G/N)}, \cR\left[(B\rtimes G/N)^{(G/N)}\right])\ar[u]^-{\text{Cor. \ref{coro:rtensor}}}_-{\rotatebox{-90}{$\scriptstyle\cong$}}\ar[l]_-{\cR(\zeta_{G/N})}^-{\eqref{eq:defiZeta}} \\
        \kk^G(A^{(G/N)}, M_{G_\plus}B) & \kk^G(A^{(G/N)}, B)\ar[l]_-{\cong}\\
        }\end{gathered}\end{equation}
The last two morphisms are easily lifted to spectra. Indeed, the $G$-stability zig-zag \eqref{eq:zigzag} induces a zig-zag of weak equivalences that is clearly natural in $G/N$:
\[\xymatrix{
    \bbK^G(A^{(G/N)}, M_{G}B)\ar[r]^-{\sim} & \bbK^G(A^{(G/N)}, M_{G_\plus}B) & \bbK^G(A^{(G/N)}, B)\ar[l]_-{\sim}
    }\]
Lifting the rest of \eqref{eq:decompAdj} is somewhat more delicate. If we simply replace groups by spectra, we get:
\begin{equation}\label{eq:zzagRecortado}
    \begin{gathered}\xymatrix@C=3.5em{
        \bbK(A, \cR(B\rtimes G/N))\ar[r]^-{(-)^{(G/N)}} & \bbK^G(A^{(G/N)}, \left[\cR(B\rtimes G/N)\right]^{(G/N)})\\
        \bbK^G(A^{(G/N)}, M_{G}B) & \bbK^G(A^{(G/N)}, \cR\left[(B\rtimes G/N)^{(G/N)}\right])\ar[u]^-{\text{Cor. \ref{coro:rtensor}}}_-{\rotatebox{90}{$\scriptstyle\sim$}}\ar[l]_-{\cR(\zeta_{G/N})}^-{\eqref{eq:defiZeta}}}\end{gathered}\end{equation}
This is what we call the \emph{primitive zig-zag}. While the spectra on the left are covariant functors on $\OrGFin$, this is not the case for those on the right.
We should start by replacing the latter by covariant functors on $\OrGFin$ if we expect a zig-zag that is natural in $G/N$. In the following sections---taking the primitive zig-zag as our model---we proceed to construct a zig-zag of spectra that depends covariantly on $G/N$ and that induces the isomorphism \eqref{eq:natAdjGroups} upon taking homotopy groups.

\subsection{Notation and preliminary definitions}\label{sec:notpreli}
To ease notation, the category $\OrGFin$ will be denoted by $\cO$ for the rest of this section. Its objects---the orbit spaces corresponding to finite subgroups of $G$---will be denoted by letters $r$, $s$ and $t$.

Let $\cC$ be a category and $\Gamma$ be a small category. We will write $\B(\Gamma, \cC)$ for the category $\cC^{\Gamma^\op\times\Gamma}$ of bifunctors $\Gamma^\op\times\Gamma\to\cC$. Let $f:\Gamma\to \Lambda$ be a functor between small categories. Then we can restrict along $f$ either of the variables of a bifunctor $\Lambda^\op\times\Lambda\to\cC$, or both of them, as shown by the following commutative diagram of categories:
\[\xymatrix{
    \B(\Lambda,\cC)\ar[r]^-{f^*}\ar[d]_-{f^*}\ar[dr]^-{f^\star} & \cC^{\Lambda^\op\times\Gamma}\ar[d]^-{f^*} \\
    \cC^{\Gamma^\op\times\Lambda}\ar[r]^-{f^*} &
    \B(\Gamma,\cC)
    }\]
Here, $f^*$ denotes restriction of one of the variables (either the covariant or the contravariant one) and $f^\star$ denotes restriction of both variables.

Define functors $J\in\Spt^\cO$ and $\Iuno, \Idos\in \B(\cO,\Spt)^\cO$ by
\begin{align}
    J(t)         & :=\bbK(A, \cR(B\rtimes t)) \nonumber                                \\
    \Idos_t(s,r) & :=\bbK^G(A^{(t)}, \cR[(B\rtimes r)^{(s)}]) \label{eq:defisFuntores} \\
    \Iuno_t(s,r) & :=\bbK^G(A^{(t)}, [\cR(B\rtimes r)]^{(s)}) \nonumber
\end{align}
for $r,s,t\in\cO$. The $\kk^G$-equivalence $\cR[(B\rtimes r)^{(s)}]\to \left[\cR(B\rtimes r)\right]^{(s)}$ of Corollary \ref{coro:rtensor} induces, upon applying $\bbK^G(A^{(t)}, -)$, a natural transformation $\psi:\Idos\to \Iuno$ that is an objecwise weak equivalence of spectra. With this notation, the primitive zig-zag \eqref{eq:zzagRecortado} becomes:
\begin{equation}\label{eq:zzagPrimit}\xymatrix{
    J(t) \ar[r]^-{(-)^{(t)}} &
    \Iuno_t(t, t) & \Idos_t(t,t)\ar[l]^-{\sim}_-{\psi}\ar[r]^-{\cR(\zeta_t)} &
    \bbK^G(A^{(t)}, M_G B)
    }\end{equation}

\subsection{Coends enter the game}\label{sec:coend}
Fix $t\in\cO$. The commutativity of \eqref{eq:zetacompat} suggests that the morphism induced by $\cR(\zeta_t)$ in \eqref{eq:zzagPrimit} could be replaced by a morphism from a certain coend, as we proceed to explain. If $f:r\to s$ is a morphism in $\cO$, the following square commutes by Lemma \ref{lem:natcounit} \eqref{item:zetacompat}:
\[
    \xymatrix{(B\rtimes r)^{(r)}\ar[r]^-{\zeta_r} & M_G\otimes B\\
    (B\rtimes r)^{(s)}\ar[u]_-{f^*}\ar[r]^-{f_*} & (B\rtimes s)^{(s)}\ar[u]_-{\zeta_{s}}}
\]
Upon applying $\bbK^G(A^{(t)}, \cR(-))$, we get a commutative diagram:
\begin{equation}\label{eq:sqZetaK1}\begin{gathered}\xymatrix{
        \Idos_t(r,r)\ar[r]^-{\zeta_r} & \bbK^G(A^{(t)}, M_GB)\\
        \Idos_t(s,r)\ar[u]^-{f^*}\ar[r]^-{f_*} & \Idos_t(s,s)\ar[u]_-{\zeta_s}
        }\end{gathered}\end{equation}
For reasons that will become clear later on (see Remark \ref{rem:choiceOfSlice1} and Lemma \ref{lem:nat:isofun}) we will take coends over the slice category $\slice{\cO}{t}$ of orbit spaces over $t$. We will denote by $\ForgetfulFunctor_t$ the forgetful functor $\slice{\cO}{t}\to \cO$. Let now $f:\alpha\to\beta$ be a morphism in $\slice{\cO}{t}$, where $\alpha:r\to t$ and $\beta:s\to t$. Then \eqref{eq:sqZetaK1} equals the square:
\[\xymatrix{
    \left[\Ustar{t} \Idos_t \right](\alpha, \alpha)\ar[r]^-{\zeta_r} & \bbK^G(A^{(t)}, M_GB)\\
    \left[\Ustar{t} \Idos_t \right](\beta, \alpha)\ar[u]^-{f^\star }\ar[r]^-{f_*} &
    \left[\Ustar{t} \Idos_t \right](\beta, \beta)\ar[u]_-{\zeta_s}
    }\]
By the universal property of the coend, there is a unique morphism $\zeta$ making the following triangle commute for all objects $\alpha:r\to t$ of $\slice{\cO}{t}$:
\begin{equation}\label{eq:defiZetaNat}\begin{gathered}\xymatrix{
        \intg{t}\Ustar{t} \Idos_t \ar[r]^-{\zeta} & \bbK^G(A^{(t)}, M_GB) \\
        \left[\Ustar{t} \Idos_t \right](\alpha, \alpha)\ar[u]^-{\can_\alpha}\ar@/_1.5pc/[ur]_-{\zeta_r}
        }\end{gathered}\end{equation}
Here the vertical morphism is the structural morphism into the coend corresponding to $\alpha$.
In the next section we will prove that $\zeta$ depends covariantly on $t\in\cO$.

\subsection{Defining \texorpdfstring{$\cO$}{O}-spectra as objectwise coends} \label{sec:coend2} Let us show that the morphisms $\zeta$ of \eqref{eq:defiZetaNat} assemble, for varying $t$, into a morphism of $\cO$-spectra. We first prove some preliminary lemmas.

\begin{lem}\label{lem:natcoend}
    Let $\cC$ be a cocomplete category, $f:\Gamma\to\Lambda$ be a functor between small categories and $T\in\B(\Lambda ,\cC)$. Then there is a unique morphism
    \begin{equation}\label{eq:natcoend}\xymatrix{\int^{\Gamma} f^\star T\ar[r] & \int^{\Lambda}  T}\end{equation}
    making the following square in $\cC$ commute, for every object $\gamma$ of $\Gamma$:
    \[\xymatrix{
        \int^{\Gamma} f^\star T\ar[r] & \int^{\Lambda}  T \\
        (f^\star T)(\gamma, \gamma)\ar@{=}[r]\ar[u]^-{\can_{\gamma}} & T(f(\gamma), f(\gamma))\ar[u]_-{\can_{f(\gamma)}}\ar[u]
        }\]
    Here the vertical arrows are the structural morphisms into the coends. Moreover, \eqref{eq:natcoend} is natural in $T$.
\end{lem}
\begin{proof}
    This is immediate from the universal property of coends.
\end{proof}

\begin{rem}\label{rem:natcoend2}
    For composable functors $\Gamma\xrightarrow{f}\Lambda\xrightarrow{g}\Sigma$ and $T\in\B(\Sigma,\cC)$, the morphisms \eqref{eq:natcoend} clearly fit into the following commutative triangle:
    \[\xymatrix{
            \int^\Gamma (g\circ f)^\star T=\int^\Gamma f^\star g^\star T\ar[r] & \int^\Lambda g^\star T\ar[d] \\
            & \int^\Sigma T
            \ar@/_1pc/ "1,1"+<-2em,-0.8em>;"2,2"
        }\]
\end{rem}

\begin{lem}\label{lem:nat}
    Let $\cC$ be a cocomplete category and $\Gamma$ be a small category. Then there is a functor $\funcoend:\B(\Gamma,\cC)^{\Gamma}\to \cC^\Gamma$ described as follows.
    \begin{enumerate}
        \item\label{item:lemnat1} Let $V$ be an object of $\B(\Gamma,\cC)^{\Gamma}$, $t\in\Gamma\mapsto V_t\in \B(\Gamma,\cC)$. For $t\in\Gamma$, we have:
              \[\funcoend(V)(t)=\textstyle\int^{\slice{\Gamma}{t}} \Ustar{t}V_t\]
              For a morphism $f:t\to t'$ in $\Gamma$, $\funcoend(V)(f)$ equals the composite:
              \[\xymatrix@C=1.7em{
                  \int^{\slice{\Gamma}{t}} \Ustar{t}V_t \ar[r]^-{V_f} &
                  \int^{\slice{\Gamma}{t}} \Ustar{t}V_{t'}=\int^{\slice{\Gamma}{t}} f^\star\Ustar{t'}V_{t'} \ar[r]^-{\eqref{eq:natcoend}} & \int^{\slice{\Gamma}{t'}}\Ustar{t'}V_{t'}}\]
        \item For a morphism $h:V\to W$ in $\B(\Gamma,\cC)^\Gamma$, the natural transformation $\funcoend(h)$ has components:
              \[\xymatrix@C=6em{\int^{\slice{\Gamma}{t}}\Ustar{t}V_t\ar[r]^-{\int^{\slice{\Gamma}{t}}\Ustar{t}h_t} & \int^{\slice{\Gamma}{t}}\Ustar{t}W_t}\]
    \end{enumerate}
\end{lem}
\begin{proof}
    The fact that the equalities in \eqref{item:lemnat1} indeed define a functor $\funcoend(V)\in\cC^\Gamma$ boils down to the naturality of \eqref{eq:natcoend} and Remark \ref{rem:natcoend2}. The fact that $\funcoend(h)$ is indeed a natural transformation follows as well from the naturality of \eqref{eq:natcoend}.
\end{proof}

\begin{lem}\label{lem:natZeta}
    The morphism $\zeta$ defined by \eqref{eq:defiZetaNat} is a morphism of $\cO$-spectra.
\end{lem}
\begin{proof}
    The codomain of $\zeta$ is clearly an $\cO$-spectrum. Its domain is an $\cO$-spectrum as well; indeed, it is $\funcoend(t\mapsto \Idos_t)$
    where $\scrC$ is the functor of Lemma \ref{lem:nat}. Let $f:t\to t'$ be a morphism in $\cO$. We claim that following square commutes:
    \[\xymatrix{
        \intg{t} \Ustar{t} \Idos_t\ar[r]^-{\zeta}\ar[d]_-{f} & \bbK^G(A^{(t)}, M_GB)\ar[d]^-{f} \\
        \intg{t'}\Ustar{t'} \Idos_{t'}\ar[r]^-{\zeta} & \bbK^G(A^{(t')}, M_GB)
        }\]
    Indeed, by the universal property of the coend, it suffices to show that the square commutes when precomposed with the structural morphisms
    \[\left[\Ustar{t} \Idos_t\right](\alpha, \alpha)\to \textstyle\intg{t}\Ustar{t} \Idos_t\]
    for every object $\alpha:r\to t$ of $\slice{\cO}{t}$. Upon precomposing with the latter we get the following square, that clearly commutes:
    \[\xymatrix{
        \bbK^G(A^{(t)}, \cR\left[(B\rtimes r)^{(r)}\right])\ar[r]^-{\zeta_{r}}\ar[d]_-{f} & \bbK^G(A^{(t)}, M_GB)\ar[d]^-{f} \\
        \bbK^G(A^{(t')}, \cR\left[(B\rtimes r)^{(r)}\right])\ar[r]^-{\zeta_{r}} & \bbK^G(A^{(t')}, M_GB)
        }\]
    This proves the lemma.
\end{proof}

\subsection{The morphism \texorpdfstring{$\varphi$}{phi}}\label{sec:morfi}
In this section we define a morphism $\varphi$ that will be part of the natural zig-zag of Theorem \ref{thm:ZZAG}.
Let $\cte:\Spt^\cO\to\B(\cO, \Spt)$ be the functor that adds a constant contravariant variable, defined by
\begin{equation}\label{eq:defiDelta}\cte_F(s,r)=F(r)\end{equation}
for $F\in\Spt^\cO$ and $r,s\in\cO$. Let $\Uast{t}:\B(\cO, \Spt)\to \Spt^{(\slice{\cO}{t})^\op \times \cO}$ be the restriction of the contravariant variable along the forgetful functor. Recall the definitions of $\J\in\Spt^\cO$ and $\Iuno_t\in\B(\cO, \Spt)$ from \eqref{eq:defisFuntores}. Define a morphism
\begin{equation}\label{eq:phiSharp}\varphi^\sharp:\Uast{t}\cte_\J\to \Uast{t}\Iuno_t\end{equation}
as follows. For objects $r$ of $\cO$ and $\alpha:s\to t$ of $\slice{\cO}{t}$, let
\begin{equation}\label{eq:defiPhi}\varphi^\sharp_{(\alpha, r)}:\left[\Uast{t}\cte_\J\right](\alpha, r)\to \left[\Uast{t}\Iuno_t\right](\alpha, r)\end{equation}
be the following composition:
\[
    \xymatrix{\bbK(A, \cR(B\rtimes r)) \ar[r]^-{(-)^{(t)}} &
    \bbK^G(A^{(t)}, [\cR(B\rtimes r)]^{(t)})\ar[r]^-{\alpha^*}
    & \bbK^G(A^{(t)}, [\cR(B\rtimes r)]^{(s)})}
\]
It is easily verified that $\varphi^\sharp$ is a natural transformation of bifunctors. Indeed, for morphisms $f:r\to r'$ in $\cO$ and $g:\alpha\to\alpha'$ in $\slice{\cO}{t}$, the following diagrams commute:

\[
    \xymatrix@C=-1em{&\bbK(A, \cR(B\rtimes r))\ar@/_2.5pc/[ddl]_-{\varphi^\sharp_{(\alpha, r)}} \ar[d]^{(-)^{(t)}} \ar@/^2.5pc/[ddr]^-{\varphi^\sharp_{(\alpha', r)}}& \\
    &\bbK^{G}(A^{(t)},[\cR(B\rtimes r)]^{(t)})\ar[dr]^{(\alpha')^{*}} \ar[dl]_{\alpha^{*}}&\\
    \bbK^{G}(A^{(t)},[\cR(B\rtimes r)]^{(s)})
    && \bbK^{G}(A^{(t)},[\cR(B\rtimes r)]^{(s')})\ar[ll]^-{g^{*}}
    }
\]

\[
    \xymatrix@C=3em{
    \bbK(A, \cR(B\rtimes r))\ar@/_6.7pc/[dd]_(.25){\varphi^\sharp_{(\alpha, r)}}\ar[d]^-{(-)^{(t)}}\ar[r]^-{f_*} & \bbK(A, \cR(B\rtimes r'))\ar[d]_-{(-)^{(t)}}\ar@/^6.7pc/[dd]^(.25){\varphi^\sharp_{(\alpha, r')}} \\
    \bbK^G(A^{(t)}, [\cR(B\rtimes r)]^{(t)})\ar[d]^-{\alpha^*}\ar[r]^-{f_*} & \bbK^G(A^{(t)}, [\cR(B\rtimes r')]^{(t)})\ar[d]_-{\alpha^*} \\
    \bbK^G(A^{(t)}, [\cR(B\rtimes r)]^{(s)})\ar[r]^-{f_*} & \bbK^G(A^{(t)}, [\cR(B\rtimes r')]^{(s)})
    }
\]

\begin{rem}\label{rem:choiceOfSlice1}
    The definition of the components of $\varphi^\sharp$ \eqref{eq:defiPhi} makes use of the structural morphism $\alpha:s\to t$ and it is not clear how to define $\varphi^\sharp$ as a morphism $\cte_\J\to \Iuno_t$ in $\B(\cO, \Spt)$. This is one of the reasons that motivated our choice of $\slice{\cO}{t}$ as the indexing category for coends.
\end{rem}

\begin{cons}\label{const:nat2}
    By Remark \ref{rem:resExt} we have, for each $t\in\cO$, a pair of adjoint functors:
    \begin{equation}\label{eq:adjResPF}\xymatrix{
        \Uexcl{t}:\Spt^{(\slice{\cO}{t})^\op\times \cO}\ar@<.5ex>[r]\ar@<-.5ex>@{<-}[r] & \B(\cO, \Spt):\Uast{t}
        }\end{equation}
    Let $f:t\to t'$ be a morphism in $\cO$ and write $f^*:\Spt^{(\slice{\cO}{t'})^\op\times\cO}\to\Spt^{(\slice{\cO}{t})^\op\times\cO}$ for the restriction of the contravariant variable along $f:\slice{\cO}{t}\to\slice{\cO}{t'}$. Note that $f^*\circ \Uast{t'}=\Uast{t}$ and consider the following diagram of solid arrows for $F\in\Spt^\cO$ and $H\in\B(\cO, \Spt)$:
    \[\xymatrix{
        \B(\cO,\Spt)\left(\Uexcl{t'}\Uast{t'}\cte_F, H\right)\ar[r]^-{\cong}\ar@{-->}[d] &
        \Spt^{(\slice{\cO}{t'})^\op\times\cO}\left(\Uast{t'}\cte_F, \Uast{t'}H\right)\ar[d]^-{f^*} \\
        \B(\cO, \Spt)\left(\Uexcl{t}\Uast{t}\cte_F, H\right)\ar[r]^-{\cong} &
        \Spt^{(\slice{\cO}{t})^\op\times\cO}\left(\Uast{t}\cte_F, \Uast{t}H\right)
        }\]
    Let the dashed arrow complete the diagram to a commutative square. By the Yoneda Lemma, the dashed arrow is induced by precomposition with a unique morphism $\Uexcl{t}\Uast{t}\cte_F\to \Uexcl{t'}\Uast{t'}\cte_F $. The latter, for varying $f$, assemble into a functor $\cO\to\B(\cO, \Spt)$, $t\mapsto \Uexcl{t}\Uast{t}\cte_F $. This construction is, moreover, clearly natural in $F$, so that we get a functor:
    \[\xymatrix@R=0em{\Spt^\cO\ar@{->}[r] & \B(\cO,\Spt)^\cO \\
        F\ar@{|->}[r] & (t\mapsto \Uexcl{t}\Uast{t}\cte_F )}\]
\end{cons}

\begin{rem}\label{rem:constr}
    Let $F\in\Spt^\cO$. Later on, it will be useful to have an explicit description of the functor $\cO\to \B(\cO, \Spt)$, $t\mapsto \Uexcl{t}\Uast{t}\cte_F $ mentioned in Construction \ref{const:nat2}. Let us first describe the bifunctor $\Uexcl{t}\Uast{t}\cte_F \in\B(\cO,\Spt)$ for fixed $t\in\cO$. For $r,s\in\cO$ we have:
    \[\left[\Uexcl{t}\Uast{t}\cte_F \right](s,r)=
        \coprod_{\alpha\in\cO(s,t)}\left[\Uast{t}\cte_F \right](\alpha, r)=
        \coprod_{\cO(s,t)} F(r)\]
    For a morphism $f:r\to r'$ in $\cO$, the induced morphism $f_*=\left[\Uexcl{t}\Uast{t}\cte_F \right](s,f)$ equals the morphism:
    \[\coprod F(f):\coprod_{\cO(s,t)} F(r)\to \coprod_{\cO(s,t)} F(r')\]
    For a morphism $g:s'\to s$ in $\cO$, the induced morphism $g^*=\left[\Uexcl{t}\Uast{t}\cte_F \right](g,r)$ is the unique morphism making the following triangle commute for all $\beta\in\cO(s,t)$:
    \[\xymatrix@R=1.5em@C=4em{F(r)\ar[r]^-{\can_\beta}\ar@/_1pc/[dr]_-{\can_{\beta\circ g}} &
        \displaystyle\coprod_{\cO(s,t)}F(r)\ar[d]^-{g^*} \\
        & \displaystyle\coprod_{\cO(s',t)}F(r)
        }\]
    Now let $h:t\to t'$ be a morphism in $\cO$. Then the components of the induced natural transformation $h_*:\Uexcl{t}\Uast{t}\cte_F \to \Uexcl{t'}\Uast{t'}\cte_F $ are the unique morphisms making the following triangle commute for all $\beta\in\cO(s,t)$:
    \[\xymatrix@R=1.5em@C=4em{F(r)\ar[r]^-{\can_\beta}\ar@/_1pc/[dr]_-{\can_{h\circ \beta}} &
        \displaystyle\coprod_{\cO(s,t)}F(r)\ar[d]^-{h_*} \\
        & \displaystyle\coprod_{\cO(s,t')}F(r)
        }\]
\end{rem}

\begin{lem}
    Fix $t\in\cO$ and let $\varphi^\sharp:\Uast{t}\cte_\J \to \Uast{t}\Iuno_t$ be the morphism in $\Spt^{(\slice{\cO}{t})^\op\times\cO}$ defined in \eqref{eq:phiSharp}. Under the adjunction \eqref{eq:adjResPF}, $\varphi^\sharp$ corresponds to a morphism $\varphi:\Uexcl{t}\Uast{t}\cte_\J \to M_t$ in $\B(\cO,\Spt)$.
    Explicitly, for $r,s\in\cO$, the component $\varphi_{(s,r)}:[\Uexcl{t}\Uast{t}\delta_J](s,r)\to \Iuno_t(s,r)$ is the unique morphism making the following triangle commute for every $\alpha\in\cO(s,t)$:
    \[\xymatrix{\displaystyle\coprod_{\cO(s,t)}J(r)\ar[r]^-{\varphi_{(s,r)}} &
        M_t(s,r) \\
        J(r)\ar[u]^-{\can_{\alpha}}\ar@/_1pc/[ur]_-{\varphi^\sharp_{(\alpha,r)}} &
        }\]
    Then the latter, for varying $t$, assemble into a morphism in $\B(\cO,\Spt)^\cO$---where the domain of $\varphi$ is considered as an object of $\B(\cO,\Spt)^\cO$ as explained in Construction \ref{const:nat2}.
\end{lem}
\begin{proof}
    Let $f:t\to t'$ be a morphism in $\cO$. We have to show that the following square in $\B(\cO,\Spt)$ commutes:
    \[\xymatrix{\Uexcl{t}\Uast{t}\cte_\J \ar[r]^-{\varphi}\ar[d]_-{f_*} &
        \Iuno_t\ar[d]^-{f_*} \\
        \Uexcl{t'}\Uast{t'}\cte_\J \ar[r]^-{\varphi} &
        \Iuno_{t'}
        }\]
    By Remark \ref{rem:constr}, for $r,s\in\cO$, we have $[\Uexcl{t}\Uast{t}\cte_\J ](s,r)=\coprod_{\cO(s,t)}J(r)$.
    Thus, it suffices to show that the following square commutes for $r,s\in\cO$:
    \[\xymatrix{\displaystyle\coprod_{\cO(s,t)}J(r)\ar[r]^-{\varphi}\ar[d]_-{f_*} &
        \Iuno_t(s,r)\ar[d]^-{f_*} \\
        \displaystyle\coprod_{\cO(s,t')}J(r)\ar[r]^-{\varphi} &
        \Iuno_{t'}(s,r)
        }\]
    By the universal property of the coproduct it suffices to show that this square commutes when precomposed with every structural morphism into the coproduct in the upper left corner. Let $\alpha\in\cO(s,t)$. Upon precomposing the latter square with $\can_\alpha:J(r)\to \coprod_{\cO(s,t)}J(r)$ we get:
    \[\xymatrix{J(r)\ar[r]^-{\varphi^\sharp_{(\alpha,r)}}\ar@/_1pc/[dr]_-{\varphi^\sharp_{(f\circ \alpha,r)}} &
        \Iuno_t(s,r)\ar[d]^-{f_*} \\
        & \Iuno_{t'}(s,r)
        }\]
    Unravelling the definition of $\varphi^\sharp$, this triangle becomes:
    \[\xymatrix@C=1.3em{
        \bbK(A, \cR(B\rtimes r))\ar[r]^-{(-)^{(t)}}\ar[d]_-{(-)^{(t')}} &
        \bbK^G(A^{(t)}, [\cR(B\rtimes r)]^{(t)})\ar[d]^-{f_*}\ar[r]^-{\alpha^*} &
        \bbK^G(A^{(t)}, [\cR(B\rtimes r)]^{(s)})\ar[d]^-{f_*} \\
        \bbK^G(A^{(t')}, [\cR(B\rtimes r)]^{(t')})\ar[r]^-{f^*} &
        \bbK^G(A^{(t')}, [\cR(B\rtimes r)]^{(t)})\ar[r]^-{\alpha^*} &
        \bbK^G(A^{(t')}, [\cR(B\rtimes r)]^{(s)})
        }\]
    The square on the right clearly commutes. The one on the left commutes by Lemma \ref{lem:sqSptComm}.
\end{proof}

\subsection{Model structure on categories of bifunctors}\label{sec:modcatbi}
Let $\cC$ be a model category and let $\Gamma$ be a small category. In this section, we endow $\B(\Gamma, \cC)$ with a model structure that will allow us to build models for the \emph{homotopy coends} of certain morphisms of bifunctors. We will need these later on to define the morphisms going backwards in the zig-zag of $\cO$-spectra of Theorem \ref{thm:ZZAG}.

\begin{prop}[\cite{luriehtt}*{Proposition A.2.8.2}]\label{prop:htt}Let $\cC$ be a combinatorial model category and let $\Gamma$ be a small category. Then there exist two combinatorial model structures on $\cC^\Gamma$:
    \begin{itemize}
        \item The \emph{injective model structure}, denoted $\cC^{\Gamma}_\inj$, where weak equivalences and cofibrations are defined objectwise.
        \item The \emph{projective model structure}, denoted $\cC^{\Gamma}_\proj$, where weak equivalences and fibrations are defined objectwise.
    \end{itemize}
\end{prop}

We will always consider $\B(\Gamma,\cC)$ as a model category with the structure $(\cC_\inj^{\Gamma^\op})_\proj^{\Gamma}$ whenever this structure exists. The model structure on $\B(\Gamma, \Spt)$ exists for any $\Gamma$ by Lemma \ref{lem:sptCombi} and Proposition \ref{prop:htt}.

\begin{thm}[\cite{daneses}*{Theorem 4.1}]\label{thm:coendquillen}
    Let $\cC$ be a model category and $\Gamma$ be a small category such that the model structure on $\B(\Gamma, \cC)$ exists. Then the functor
    \[\textstyle\int^{\Gamma}:\B(\Gamma, \cC)\to \cC\]
    is a left Quillen functor.
\end{thm}
\begin{proof}
    This is \cite{daneses}*{Theorem 4.1}; we sketch the proof for completeness. For $c\in\cC$, define $R(c):\Gamma^\op\times\Gamma\to\cC$ by:
    \[R(c)(s,r):=\prod_{\alpha\in\Gamma(r, s)}c\]
    For morphisms $f:r\to r'$ and $g:s'\to s$ in $\Gamma$, let $f_*$ and $g^*$ be the unique morphisms making the following diagrams commute for all $\alpha\in\Gamma(r',s)$ and all $\beta\in\Gamma(r,s')$:
    \[
        \xymatrix@C=4em{R(c)(s,r)\ar[r]^-{f_*}\ar@/_0.5pc/[dr]_-{\can_{\alpha\circ f}} & R(c)(s,r')\ar[d]^-{\can_{\alpha}} \\
        & c
        }
        \hspace{2em}
        \xymatrix@C=4em{R(c)(s,r)\ar[r]^-{g^*}\ar@/_0.5pc/[dr]_-{\can_{g\circ\beta}} & R(c)(s',r)\ar[d]^-{\can_{\beta}} \\
        & c
        }
    \]
    It is easily verified that $R:\cC\to\cC^{\Gamma^\op\times\Gamma}$ is right adjoint to $\int^\Gamma:\cC^{\Gamma^\op\times\Gamma}\to\cC$. Thus, proving the theorem is equivalent to showing that
    \[R:\cC\to\B(\Gamma, \cC)\]
    preserves fibrations and trivial fibrations. Let $c\to c'$ be a (trivial) fibration in $\cC$. To prove that $R(c)\to R(c')$ is a (trivial) fibration in $\B(\Gamma, \cC)=(\cC^{\Gamma^\op}_\inj)^{\Gamma}_\proj$, it suffices to show that $R(c)(-,r)\to R(c')(-,r)$ is a (trivial) fibration in $\cC^{\Gamma^\op}_\inj$ for every $r\in\Gamma$. But the latter holds since, for every $r$, there is a Quillen adjunction
    \[\ev_r:\cC^{\Gamma^\op}_\inj\rightleftarrows \cC:R(-)(-,r)\]
    where $\ev_r$ is the evaluation at $r$ \cite{daneses}*{Corollary 2.3 (iii)}. Indeed, this adjunction is Quillen since $\ev_r$ clearly preserves cofibrations and trivial cofibrations.
\end{proof}

\begin{lem}[\cite{luriehtt}*{Proposition A.2.8.7}]\label{lem:resLeftQuillen1}
    Let $\cC$ be a model category and let $f:\Gamma\to\Lambda$ be a functor. Then the restriction functor $f^*:\cC^\Lambda\to\cC^\Gamma$ fits into the following Quillen adjunctions, whenever the model structures in question exist:
    \begin{enumerate}
        \item $f^*:\cC^\Lambda_\inj\rightleftarrows \cC^\Gamma_\inj:f_*$
        \item $f_!:\cC^\Gamma_\proj\rightleftarrows \cC^\Lambda_\proj:f^*$
    \end{enumerate}
\end{lem}

\begin{rem}\label{rem:resExt}
    Let $\cC$ be a category with small coproducts and let $\Gamma$ be a small category. Fix $t\in\Gamma$ and let $\ForgetfulFunctor_t:\slice{\Gamma}{t}\to \Gamma$ be the forgetful functor. Then there is an adjunction:
    \[\begin{tikzcd}
            \Uexcl{t}:\cC^{(\slice{\Gamma}{t})^\op} \arrow[r,shift left]
            &
            \cC^{\Gamma^\op}:\Uast{t} \arrow[l,shift left]
        \end{tikzcd}\]
    Moreover, the pushforward functor $\Uexcl{t}$ can be explicitely described as follows. For $s\in\Gamma$, we have:
    \[\left[\Uexcl{t} F\right](s)=\coprod_{\alpha\in\Hom(s, t)}F(\alpha)\]
    For a morphism $g:s'\to s$ in $\Gamma$, $\left[\Uexcl{t} F\right](g)$ is the unique morphism making the following squares commute, where the vertical arrows are the structural morphisms into the coproducts:
    \[\xymatrix@C=7em{
        \displaystyle\coprod_{\alpha\in\Hom(s, t)}F(\alpha)\ar[r]^-{\left[\Uexcl{t} F\right](g)} & \displaystyle\coprod_{\alpha'\in\Hom(s', t)}F(\alpha') \\
        F(\beta)\ar[u]^-{\can_\beta}\ar[r]^-{F(g)} & F(\beta\circ g)\ar[u]_-{\can_{\beta\circ g}}
        }\]
\end{rem}

\begin{lem}[cf. \cite{luriehtt}*{Lemma A.2.8.10}]\label{lem:resLeftQuillen2}
    Let $\cC$ be a model category, $\Gamma$ be a small category, fix $t\in\Gamma$, and let $\ForgetfulFunctor_t:\slice{\Gamma}{t}\to\Gamma$ be the forgetful functor from the slice category. Then the following adjunctions are Quillen adjunctions, whenever the model structures in question exist:
    \begin{enumerate}
        \item\label{item:adj1} $\Uast{t}:\cC^{\Gamma}_\proj\rightleftarrows \cC^{\slice{\Gamma}{t}}_\proj:\Uastd{t}$
        \item\label{item:adj2} $\Uexcl{t}:\cC^{(\slice{\Gamma}{t})^\op}_\inj\rightleftarrows \cC^{\Gamma^\op}_\inj:\Uast{t}$
    \end{enumerate}
\end{lem}
\begin{proof}
    To prove \eqref{item:adj2} is a Quillen adjunction, let us show that $\Uexcl{t}$ preserves cofibrations and trivial cofibrations. Recall from Remark \ref{rem:resExt} that, for $F\in\cC^{(\slice{\Gamma}{t})^\op}$ and $s\in\Gamma$, we have:
    \[\left[\Uexcl{t} F\right](s)=\coprod_{\alpha\in\Gamma(s, t)}F(\alpha)\]
    Let $\eta:F\to F'$ be a morphism in $\cC^{(\slice{\Gamma}{t})^\op}$. For $s\in\Gamma$, $\Uexcl{t}(\eta)(s)$ is the coproduct of the morphisms:
    \[\{\eta(\alpha):F(\alpha)\to F'(\alpha)\}_{\alpha\in\Gamma(s, t)}\]
    If $\eta$ is a cofibration (resp. a trivial cofibration) in $\cC^{(\slice{\Gamma}{t})^\op}_\inj$, the latter are cofibrations (resp. trivial cofibrations) in $\cC$ and, thus, $\Uexcl{t}(\eta)(s)$ is again a cofibration (resp. trivial cofibration) in $\cC$. Since this holds for every $s\in\Gamma$, it follows that $\Uexcl{t}(\eta)$ is a cofibration (resp. a trivial cofibration) in $\cC^{\Gamma^\op}_\inj$.
    The proof of \eqref{item:adj1} is dual to that of \eqref{item:adj2}, using the fact that
    \[\left[\Uastd{t} F\right](s)=\prod_{\alpha\in\Gamma(t, s)}F(\alpha)\]
    for $F\in \cC^{\slice{\Gamma}{t}}$ and $s\in\Gamma$.
\end{proof}

\begin{lem}[\cite{luriehtt}*{Remark A.2.8.6}]\label{lem:htt}
    Let $F:\cC\rightleftarrows \cD:U$ be a Quillen adjunction of combinatorial model categories and let $\Gamma$ be a small category. Then composition with $F$ and $U$ determines a Quillen adjunction
    \[F^{\Gamma}:\cC^{\Gamma}\rightleftarrows \cD^{\Gamma}:U^{\Gamma}\]
    with respect to either the injective or the projective model structures.
\end{lem}

\begin{lem}\label{lem:deltaLQ}
    Let $\cte:\Spt^\cO_\proj\to \B(\cO, \Spt)$ be the functor defined in \eqref{eq:defiDelta}. Then $\cte$ is a left Quillen functor.
\end{lem}
\begin{proof}
    Let $\const:\Spt\to\Spt^{\cO^\op}_\inj$ be the functor that sends a spectrum $E$ to the constant functor on $E$. Then $\const$ is left Quillen since it clearly sends cofibrations (resp. trivial cofibrations) to cofibrations (resp. trivial cofibrations) and it has a right adjoint (taking limit). Note that we have:
    \[\cte=(\const)^\cO:\Spt^{\cO}_\proj\to \left[\Spt^{\cO^\op}_\inj\right]^\cO_\proj=\B(\cO,\Spt)\]
    Then $\cte$ is left Quillen by Lemma \ref{lem:htt}.
\end{proof}

\begin{lem}\label{lem:restIsLeftQuillen}
    Let $\cC$ be a combinatorial model category, $\Gamma$ be a small category, fix $t\in\Gamma$, and let $\ForgetfulFunctor_t:\slice{\Gamma}{t}\to\Gamma$ be the forgetful functor from the slice category. Then the restriction on both variables
    \[\Ustar{t}:\B(\Gamma,\cC)\to\B(\slice{\Gamma}{t}, \cC)\]
    is a left Quillen functor.
\end{lem}
\begin{proof}
    Consider the following commutative square:
    \[\xymatrix{B(\Gamma,\cC)=(\cC^{\Gamma^\op}_\inj)^{\Gamma}_\proj\ar[dr]^-{\Ustar{t}}\ar[r]^-{\Uast{t}}\ar[d]_-{\Uast{t}} & (\cC^{\Gamma^\op}_\inj)^{\slice{\Gamma}{t}}_\proj \ar[d]^-{\Uast{t}} \\
        \left[\cC^{(\slice{\Gamma}{t})^\op}_\inj\right]^{\Gamma}_\proj \ar[r]^-{\Uast{t}} & \left[\cC^{(\slice{\Gamma}{t})^\op}_\inj\right]^{\slice{\Gamma}{t}}_\proj=\B(\slice{\Gamma}{t}, \cC)}\]
    The horizontal morphisms are left Quillen by Lemma \ref{lem:resLeftQuillen2} \eqref{item:adj1} and the vertical morphisms are left Quillen by Lemma \ref{lem:resLeftQuillen1} and Lemma \ref{lem:htt}. Then $\Ustar{t}$ is left Quillen as well, for being the composite of left Quillen functors.
\end{proof}

\subsection{The natural zig-zag}\label{sec:zznat}
In this section we finally construct a zig-zag of $\cO$-spectra inducing \eqref{eq:natAdjGroups} upon taking homotopy groups. We begin with the following lemma, that shows that every $\cO$-spectrum can be canonically described as an objectwise coend.
\begin{lem}\label{lem:nat:isofun}
    Let $\cte:\Spt^\cO_\proj\to \B(\cO, \Spt)$ be the functor defined in \eqref{eq:defiDelta}. For $F\in\Spt^\cO$ and $t\in\cO$, the structural morphisms into the coends
    \begin{equation}\label{eq:nat:isofun}
        \xymatrix{F(t)=\left[\Ustar{t}\cte_F \right](\id_t,\id_t)\ar[r] & \textstyle\intg{t}\Ustar{t}\cte_F }
    \end{equation}
    are isomorphisms. Moreover, these are natural in $t\in\cO$ and in $F\in\Spt^\cO$.
\end{lem}
\begin{proof}
    Fix $t\in\cO$. Since $\Ustar{t}\cte_F $ is constant in the contravariant variable, we have:
    \[\textstyle\int^{\slice{\cO}{t}}\Ustar{t}\cte_F \cong \displaystyle\colim_{\alpha\in\slice{\cO}{t}}F(\ForgetfulFunctor_t(\alpha))\]
    Since $\id_t$ is a final object of $\slice{\cO}{t}$, the structural morphism
    \[\xymatrix{F(t)=F(\ForgetfulFunctor_t(\id_t))\ar[r] & \displaystyle\colim_{\alpha\in\slice{\cO}{t}}F(\ForgetfulFunctor_t(\alpha))}\]
    is an isomorphism. Combining the above we get the desired isomorphism:
    \[\xymatrix{F(t)\ar[r]^-{\cong} & \textstyle\int^{\slice{\cO}{t}}\Ustar{t}\cte_F }\]
    It is easily verified that this is natural in $t$ and in $F$.
\end{proof}

\begin{lem}\label{lem:natUnit}
    Let $F\in\Spt^\cO$ and fix $t\in\cO$. Let $\Uexcl{t}$ and $\Uast{t}$ be the functors that form the adjunction \eqref{eq:adjResPF}. Then there is a morphism of bifunctors $\Ustar{t}\cte_F \to \Ustar{t}\Uexcl{t}\Uast{t}\cte_F $
    described as follows. For objects $\alpha:r\to t$ and $\beta:s\to t$ of $\slice{\cO}{t}$, the component corresponding to the pair $(\beta,\alpha)$ is the structural morphism into the coproduct corresponding to $\beta$:
    \begin{equation}\label{eq:lemNatUnit}\xymatrix{
        \left[\Ustar{t}\cte_F \right](\beta,\alpha)=F(r)\ar[r]^-{\can_\beta} &
        \displaystyle\coprod_{\cO(s,t)}F(r)=\left[\Ustar{t}\Uexcl{t}\Uast{t}\cte_F \right](\beta, \alpha)
        }\end{equation}
    Moreover, upon taking coend we get a morphism of spectra
    \begin{equation}\label{eq:nat:unit}\xymatrix{
            \intg{t}\Ustar{t}\cte_F \ar[r] & \intg{t}\Ustar{t}\Uexcl{t}\Uast{t}\cte_F
        }\end{equation}
    that is natural in $t$.
\end{lem}
\begin{proof}
    The fact that the morphisms \eqref{eq:lemNatUnit} are natural in $\alpha$ and $\beta$ is easily verified using the explicit description of the bifunctor $\Uexcl{t}\Uast{t}\cte_F $ given in Remark \ref{rem:constr}.

    Let us now prove that \eqref{eq:nat:unit} is natural in $t$. Let $h:t\to t'$ be a morphism in $\cO$. By the universal property of the coend, it suffices to show that the outer square in the following diagram commutes for every $\alpha:r\to t$ in $\slice{\cO}{t}$:
    \[\xymatrix@C=2em@R=2em{F(r)=[\Ustar{t}\cte_F ](\alpha, \alpha)\ar[dr]^-{\can_\alpha}\ar@/^1pc/[rrd]\ar@/_2pc/[ddr] & & \\
        & \intg{t}\Ustar{t}\cte_F \ar[d]_-{h_*}\ar[r]^-{\eqref{eq:nat:unit}} &
        \intg{t}\Ustar{t}\Uexcl{t}\Uast{t}\cte_F \ar[d]^-{h_*} \\
        & \intg{t'}\Ustar{t'}\cte_F \ar[r]^-{\eqref{eq:nat:unit}} &
        \intg{t'}\Ustar{t'}\Uexcl{t'}\Uast{t'}\cte_F
        }\]
    Unravelling the definitions of $h_*$ (see Lemma \ref{lem:nat} and Remark \ref{rem:constr}), it is straightforward to verify that both ways from $F(r)$ to $\intg{t'}\Ustar{t'}\Uexcl{t'}\Uast{t'}\cte_F $ in the diagram above equal the composite:
    \[\xymatrix@C=4em{F(r)\ar[r]^-{\can_{h\circ\alpha}} &
        \displaystyle\coprod_{\cO(r, t')}F(r)\ar[r]^-{\can_{h\circ\alpha}} &
        \intg{t'}\Ustar{t'}\Uexcl{t'}\Uast{t'}\cte_F
        }\]
    This finishes the proof.
\end{proof}

\begin{thm}\label{thm:ZZAG}
    Let $q:Q\xrightarrow{\sim} \id$ and $\bar{q}:\bQ\xrightarrow{\sim}\id$ be, respectively, cofibrant replacements in $\B(\cO,\Spt)$ and $\Spt^{\cO}_{\proj}$. Then we have a zig-zag of $\cO$-spectra as follows:
    \[
        \xymatrix@C=1.6em{
        \intg{t} \Ustar{t} \cte_\J  &
        \intg{t} \Ustar{t} \cte_{\bQ \J}\ar[l]^-{\sim}_-{\bar{q}}\ar[r]^-{\eqref{eq:nat:unit}} &
        \intg{t} \Ustar{t} \Uexcl{t} \Uast{t} \cte_{\bQ \J} \\
        \intg{t} \Ustar{t} Q (\Iuno_t) &
        \intg{t} \Ustar{t} Q \Uexcl{t} \Uast{t} \cte_\J \ar[l]_-{\varphi}  &
        \intg{t} \Ustar{t} Q \Uexcl{t} \Uast{t} \cte_{\bQ \J}\ar[u]_-{q}^-{\rotatebox{90}{$\scriptstyle\sim$}}\ar[l]_-{\bar{q}} \\
        \intg{t} \Ustar{t} Q (\Idos_t )\ar[r]^-{q}\ar[u]^-{\psi}_-{\rotatebox{90}{$\scriptstyle\sim$}} &
        \intg{t} \Ustar{t} \Idos_t \ar[r]^-{\zeta} &
        \bbK^G(A^{(t)}, M_GB)}
    \]
    Moreover, upon identifying
    \[\bbK(A, \cR(B\rtimes t))=\J(t)\overset{\eqref{eq:nat:isofun}}{\cong} \intg{t}\Ustar{t}\cte_\J \]
    and then taking homotopy groups, this zig-zag induces the isomorphism \eqref{eq:natAdjGroups}.
\end{thm}
\begin{proof}
    Let us first show that the morphisms appearing in the zig-zag are indeed natural in $t$. For \eqref{eq:nat:unit} this is part of Lemma \ref{lem:natUnit} and for $\zeta$ it is Lemma \ref{lem:natZeta}. The rest of the morphisms are natural in $t$ because they result from aplying the functor $\funcoend$ of Lemma \ref{lem:nat} to appropriate morphisms in $\B(\cO, \Spt)^\cO$.

    The morphisms in the zig-zag labeled with $\sim$ are indeed weak equivalences by Ken Brown's Lemma \cite{hoveymodel}*{Lemma 1.1.12}: they result from applying a left Quillen functor to an appropriate weak equivalence between cofibrant objects. Here we use that the functors $\cte$, $\Ustar{t}$ and $\intg{t}$ are left Quillen by Lemmas \ref{lem:deltaLQ}, \ref{lem:restIsLeftQuillen} and Theorem \ref{thm:coendquillen} respectively.

    The fact that the zig-zag induces \eqref{eq:natAdjGroups} upon taking homotopy groups follows from the commutativity of the following diagram of spectra.

    \[
        \xymatrix{
        \framebox{$\bbK(A, \cR(B\rtimes t))$}
        \ar@{=}[r]
        \ar[dd]^-{\can_{\id_t}} &
        \J(t)
        \ar[r]^-{\can_{\id_t}}_-{\text{$\cong$ by Lem. \ref{lem:nat:isofun}}} &
        \intg{t}\Ustar{t}\cte_\J  \\
        & (\bQ\J)(t)
        \ar[u]^-{\bar{q}}_-{\rotatebox{90}{$\scriptstyle\sim$}}
        \ar[r]^-{\can_{\id_t}}_-{\text{$\cong$ by Lem. \ref{lem:nat:isofun}}}
        \ar[d]_-{\can_{\id_t}} &
        \intg{t}\Ustar{t}\cte_{\bQ \J}
        \ar[u]^-{\bar{q}}_-{\rotatebox{90}{$\scriptstyle\sim$}}
        \ar[d]_-{\eqref{eq:nat:unit}} \\
        \displaystyle\coprod_{\gamma\in\cO(t,t)}\J(t)
        \ar[ddd]^-{\varphi} &
        \displaystyle\coprod_{\gamma\in\cO(t,t)}(\bQ\J)(t)
        \ar[r]^-{\can_{\id_t}}
        \ar[l]_-{\bar{q}} &
        \intg{t}\Ustar{t}\Uexcl{t}\Uast{t}\cte_{\bQ \J} \\
        & \left[Q\Uexcl{t}\cte_{\bQ \J}\right](t,t)
        \ar[u]^-{q}_-{\rotatebox{90}{$\scriptstyle\sim$}}
        \ar[d]_-{\bar{q}}
        \ar[r]^-{\can_{\id_t}} &
        \intg{t}\Ustar{t}Q\Uexcl{t}\Uast{t}\cte_{\bQ \J}
        \ar[u]^-{q}_-{\rotatebox{90}{$\scriptstyle\sim$}}
        \ar[d]_-{\bar{q}} \\
        & \left[Q\Uexcl{t}\cte_\J \right](t,t)
        \ar[d]_-{\varphi}
        \ar[r]^-{\can_{\id_t}} &
        \intg{t}\Ustar{t}Q\Uexcl{t}\Uast{t}\cte_\J
        \ar[d]_-{\varphi} \\
        \Iuno_t(t,t) & (Q\Iuno_t)(t,t)
        \ar[r]^-{\can_{\id_t}}
        \ar[l]_-{q}^-{\sim} &
        \intg{t}\Ustar{t}Q(\Iuno_t) \\
        & (Q\Idos_t )(t,t)
        \ar[r]^-{\can_{\id_t}}
        \ar[u]^-{\psi}_-{\rotatebox{90}{$\scriptstyle\sim$}}
        \ar[d]_-{q} &
        \intg{t}\Ustar{t}Q(\Idos_t )
        \ar[u]^-{\psi}_-{\rotatebox{90}{$\scriptstyle\sim$}}
        \ar[d]_-{q} \\
        & \Idos_t (t,t)
        \ar[r]^-{\can_{\id_t}}
        \ar@/^2pc/[uul]^-{\psi}_-{\hspace{-0.5em}\rotatebox{135}{$\scriptstyle\sim$}}
        \ar@/_1pc/[dr]_-{\zeta_t} &
        \intg{t}\Ustar{t}\Idos_t
        \ar[d]^-{\zeta} \\
        & & \framebox{$\bbK^G(A^{(t)}, M_GB$)}
        \ar@(l,dr) "5,2";"3,1"^-{q}_-{\hspace{-0.5em}\rotatebox{135}{$\scriptstyle\sim$}}
        \ar@(l,l)@/_3pc/ "1,1";"6,1"_-{(-)^{(t)}}
        }
    \]
    Upon taking homotopy groups, the morphisms labeled with $\sim$ in the diagram above become isomorphisms and can be inverted. The zig-zag equals the morphism in the top row followed by the composite of the morphisms in the rightmost column. By Lemma \ref{lem:natcounit} \eqref{item:couniadj}, the composite of the bent morphisms on the left equals, upon taking homotopy groups, the isomorphism \eqref{eq:natAdjGroups}.
\end{proof}

%%%%%%%%%%%%%%%%%%%%%%%%%%%%%%%%%%%%%%%%%%%%%%%%%%%%%%%%%%%%%%%%%%%%%%%%%%%%%%%%%%%%%
%%%%%%%%%%%%%%%%%%%%%%%%%%%%%% Main theorem %%%%%%%%%%%%%%%%%%%%%%%%%%%%%%%%%%%%%%%%%
%%%%%%%%%%%%%%%%%%%%%%%%%%%%%%%%%%%%%%%%%%%%%%%%%%%%%%%%%%%%%%%%%%%%%%%%%%%%%%%%%%%%%
\section{Main theorem}\label{sec:prin}

In this section we prove Theorem \ref{thm:mainthm}. It will turn out to be an easy consequence of the technical Theorem \ref{thm:ZZAG} and the following two lemmas.

\begin{lem}\label{lem:defiAlphaX}
    Let $A$ be an algebra, $B$ be a $G$-algebra and $X$ be a $\GFin$-finite complex. Then there exists a weak equivalence of spectra
    \begin{equation}\label{eq:transNat}
        \alpha_{X}: H^G(X; \bbK^G(A^{(-)}, B))\rightarrow
        \bbK^G(A^{(X)}, B)\end{equation}
    that is natural in $X$.
\end{lem}
\begin{proof}
    Let $\bE:\OrGFin\to\Spt$ be defined by $\bE(G/H)= \bbK^G(A^{(G/H)}, B)$ and let $X$ be a $\GFin$-complex. Since coends and smash products commute with colimits, we have:
    \begin{align*}
        H^G(X; \bE) & = \int^{G/H}X^H_+\wedge \bE(G/H)                                                                                                                \\[1em]
                    & \cong \int^{G/H} \colim_{G/K\times \Delta^n\downarrow X} (G/K\times\Delta^n)^H_+ \wedge \bE(G/H)      &  & \text{(by Lemma \ref{lem:fixcolim})} \\[1em]
                    & \cong \colim_{G/K\times \Delta^n\downarrow X} \Delta^n_+ \wedge \int^{G/H}  (G/K)^H_+ \wedge \bE(G/H) &  & \text{(by Lemma \ref{lem:prodCat})}  \\[1em]
                    & = \colim_{G/K\times \Delta^n\downarrow X} \Delta^n_+ \wedge H^G(G/K; \bE)                                                                       \\[1em]
                    & \cong \displaystyle\colim_{G/K\times \Delta^n\downarrow X} \Delta^n_+ \wedge \bE(G/K)
    \end{align*}
    Thus, to construct the morphism \eqref{eq:transNat} it suffices to define compatible morphisms \begin{equation}\label{eq:compatMorph}\Delta^n_+\wedge \bbK^G(A^{(G/K)},B)\to \bbK^G(A^{(X)},B)\end{equation}
    for every $f:G/K\times \Delta^n\to X$. Define \eqref{eq:compatMorph} as the composite:
    \[\xymatrix{\Delta^n_+\wedge\bbK^G(A^{(G/K)},B)\ar[r]^-{\eqref{eq:lemMult}} &
        \bbK^G(A^{(G/K\times \Delta^n)},B)\ar[r]^-{f_*} &
        \bbK^G(A^{(X)},B)}\]
    The compatibility of these morphisms is immediate from the naturality of \eqref{eq:lemMult} in $G/K$. This defines the morphism \eqref{eq:transNat}.

    We claim that $\alpha_{X}$ is a weak equivalence of spectra. We will prove the claim by considering the following increasingly more general cases:
    \begin{enumerate}
        \item\label{item:0celda} $X=G/H$ with $H\in \FIN$
        \item\label{item:ncelda} $X=G/H \times \Delta^{n}$ with $H\in \FIN$ and $n\in\N$
        \item\label{item:uniondisjunta} $\displaystyle{X=\bigsqcup^{m}_{i=1} G/H_{i}\times \Delta^{n} }$ with $H_{i}\in \FIN$ and $m,n\in\N$
        \item\label{item:general} $X$ any $\GFin$-finite complex
    \end{enumerate}

    To prove the case \eqref{item:0celda}, note that the identity of $X\cong G/H\times \Delta^0$ is a final object among those $G/K\times\Delta^n\downarrow X$. Then, taking colimit over $G/K\times\Delta^n\downarrow X$ boils down to evaluating at the final object $\id_{G/H\times \Delta^0}$ and the result follows from Lemma \ref{lem:mult}.

    For the case \eqref{item:ncelda}, let $\pi:G/H\times \Delta^{n}\to G/H$ be the projection and consider the following commutative diagram:
    \[
        \xymatrix@C=5em{ H^G(G/H\times \Delta^{n}; \bE) \ar[r]^-{\alpha_{G/H\times \Delta^n}} \ar[d]^-{\rotatebox{90}{$\scriptstyle\sim$}}_{\pi_{*}}& \bbK^G(A^{(G/H\times \Delta^{n})}, B) \ar[d]_-{\rotatebox{90}{$\scriptstyle\sim$}}^{\pi_{*}}\\
        H^G(G/H; \bE) \ar[r]_-{\sim}^-{\alpha_{G/H}} & \bbK^G(A^{(G/H)}, B)
        }
    \]
    The left vertical arrow is an equivalence by homotopy invariance of equivariant homology theories. The right vertical arrow is a weak equivalence by homotopy invariance of $\kk_*^G(-,B)$, since we have:
    \[A^{(G/H\times \Delta^{n})}\cong A^{\left(\bigsqcup_{G/H}\Delta^n\right)}\cong \bigoplus_{G/H} A^{\Delta^n} \cong\left(\bigoplus_{G/H}A\right)\otimes \ell^{\Delta^n}\cong A^{(G/H)}\otimes \ell^{\Delta^n}\]
    The bottom arrow is a weak equivalence by the case \eqref{item:0celda}.

    Now let $X$ be as in case \eqref{item:uniondisjunta}. This case follows from the previous one since we have:
    \begin{align*}
        H^{G}_*(X;\bE)      & \cong \bigoplus^{n}_{i=1}H^{G}_{*}(G/H_{i} \times \Delta^{n};\bE)   \\
        \kk^G_*(A^{(X)}, B) & \cong  \bigoplus_{i=1}^{n}\kk_*^G(A^{(G/H_{i}\times \Delta^{n})},B) \\
        \alpha_X            & =\oplus \alpha_{G/H_i\times \Delta^n}
    \end{align*}

    To prove the case \eqref{item:general} we proceed by induction on $n=\dim X$. The base case $n=0$ holds by the case \eqref{item:uniondisjunta}. Now suppose $n\geq 1$ and assume that $\alpha_Z$ is a weak equivalence for every $\GFin$-finite complex $Z$ with $\dim Z\leq n-1$. Since $X$ is $\GFin$-finite, there is a pushout diagram as follows:
    \[
        \xymatrix{\bigsqcup_{i=1}^{n}G/H_{i} \times \partial \Delta^{n} \ar[r] \ar@{->}[d] & \sk_{n-1}X \ar@{->}[d] & \\
        \bigsqcup_{i=1}^{n}G/H_{i} \times \Delta^{n} \ar[r] & X
        }
    \]
    By Lemma \ref{lem:milnor}, upon applying the functor $A^{(-)}$ we obtain a Milnor square:
    \[
        \xymatrix{A^{(\bigsqcup_{i=1}^{n}G/H_{i} \times \partial \Delta^{n} )}  & A^{(\sk_{n-1}X)}  \ar[l]  \\
        \displaystyle A^{(\bigsqcup_{i=1}^{n}G/H_{i} \times \Delta^{n})} \ar@{->>}[u] & A^{(X)} \ar@{->>}[u] \ar[l]
        }
    \]
    This induces a morphism of Mayer-Vietoris sequences as follows; see Lemma \ref{lem:excision}:
    \[
        \xymatrix@C=5em{
        H_{*}^{G}(\bigsqcup_{i=1}^{n}G/H_{i} \times \partial \Delta^{n}; \bE) \ar[d]  \ar[r]^{\alpha_{\bigsqcup_{i=1}^{n}G/H_{i} \times \partial \Delta^{n}}} & \kk_*^G(A^{(\bigsqcup_{i=1}^{n}G/H_{i} \times \partial \Delta^{n})}, B) \ar[d] \\
        {\begin{array}{@{}c@{}}
            H_{*}^{G}(\bigsqcup_{i=1}^{n}G/H_{i} \times \Delta^{n}; \bE) \\
            \oplus                                                       \\  H_{*}^{G}(\sk_{n-1}X; \bE)
        \end{array}}
        \ar[d] \ar[r]^{\scriptsize{{\begin{array}{@{}c@{}}
                        \alpha_{\bigsqcup_{i=1}^{n}G/H_{i} \times \Delta^{n}} \\
                        \oplus                                                \\  \alpha_{\sk_{n-1}X}
                    \end{array}}}}&
        {\begin{array}{@{}c@{}}
                    \kk_*^G(A^{(\bigsqcup_{i=1}^{n}G/H_{i} \times\Delta^{n})}, B) \\
                    \oplus                                                        \\  \kk^G_*(A^{(\sk_{n-1}X)},B)
                \end{array}}
        \ar[d]  \\
        H_{*}^{G}(X; \bE) \ar[d] \ar[r]^{\alpha_{X}} & \kk^G_*(A^{(X)},B) \ar[d] \\
        H_{*-1}(\bigsqcup_{i=1}^{n}G/H_{i} \times \partial \Delta^{n}; \bE) \ar[r]  &  \kk_{*-1}^G(A^{(\bigsqcup_{i=1}^{n}G/H_{i} \times \partial \Delta^{n})},B ) }
    \]
    The morphisms $\alpha_{\bigsqcup_{i=1}^{n}G/H_{i} \times \partial \Delta^{n}}$ and $\alpha_{\sk_{n-1}X}$ are isomorphisms by the inductive hypothesis.
    The morphism $\alpha_{\bigsqcup_{i=1}^{n}G/H_{i} \times \Delta^{n}}$ is an isomorphism by the case \eqref{item:uniondisjunta}. We conclude by the Five Lemma that $\alpha_{X}$ is an isomorphism.
\end{proof}

\begin{rem}
    The morphism $\alpha_X$ of Lemma \ref{lem:defiAlphaX} can be defined for any $\GFin$-complex $X$. However, the hypothesis of $X$ being $\GFin$-finite is needed for stating that $\alpha_X$ is natural in $X$---the codomain of $\alpha_X$ is a functor on $X$ only when restricted to the full subcategory of $\S^G$ whose objects are $\GFin$-finite complexes.
\end{rem}

\begin{lem}\label{lem:homologiesCoincide}
    Let $G$ be a group  satisfying \eqref{eq:hipG}. Let $A$ be an algebra, $B$ be a $G$-algebra and $Z$ be a $\GFin$-complex. Then there is a natural isomorphism
    \[H^G_*(Z; \bbK(A, \cR(B\rtimes -)))\cong H^G_*(Z; \bbK^G(A^{(-)},B))\]
    induced by a natural zig-zag of spectra.
\end{lem}
\begin{proof}
    Recall the notation and definitions introduced in Section \ref{sec:bomba}. By Theorem \ref{thm:ZZAG}, we have a zig-zag of $\OrGFin$-spectra as follows, where the morphisms labelled with $\sim$ are objectwise weak equivalences of spectra:
    \begin{equation}\label{eq:zzagmain}\begin{gathered}
            \xymatrix@C=2em@R=1.5em{
            \bbK(A, \cR(B\rtimes t))&
            \intg{t} \Ustar{t} \cte_{\bQ \J}\ar[l]_-{\sim}\ar[r] &
            \intg{t} \Ustar{t} \Uexcl{t} \Uast{t} \cte_{\bQ \J} \\
            \intg{t} \Ustar{t} Q (\Iuno_t) & &
            \ar[ll]_-{\varphi\circ\bar{q}}
            \intg{t} \Ustar{t} Q \Uexcl{t} \Uast{t} \cte_{\bQ \J}\ar[u]^-{\rotatebox{90}{$\scriptstyle\sim$}}\\
            \intg{t} \Ustar{t} Q (\Idos_t )\ar[rr]^-{\zeta\circ q}\ar[u]_-{\rotatebox{90}{$\scriptstyle\sim$}} &&
            \bbK^G(A^{(t)}, M_GB)}
        \end{gathered}
    \end{equation}
    We also have a zig-zag of $\OrGFin$-spectra as follows, induced by the $G$-stability zig-zag \eqref{eq:zigzag}:
    \begin{equation}\label{eq:zzagGstab}\xymatrix{\bbK^G(A^{(t)}, M_GB)\ar[r]^-{\iota'}_-{\sim} & \bbK^G(A^{(t)}, M_{G_+}B) & \bbK^G(A^{(t)}, B)\ar[l]_-{\iota}^-{\sim}}\end{equation}
    In \eqref{eq:zzagGstab}, both morphisms are objectwise weak equivalences of spectra since the $G$-algebra homomorphisms of \eqref{eq:zigzag} induce isomorphisms in $\kk^G$.
    To ease notation, we name the $\OrGFin$-spectra appearing in the zig-zags above as follows:
    \begin{align*}
        \bF(t)   & := \bbK(A, \cR(B\rtimes t))                                       \\
        \bR_1(t) & :=\textstyle\intg{t} \Ustar{t} \cte_{\bQ \J}                      \\
        \bR_2(t) & :=\textstyle\intg{t} \Ustar{t} \Uexcl{t} \Uast{t} \cte_{\bQ \J}   \\
        \bR_3(t) & :=\textstyle\intg{t} \Ustar{t} Q \Uexcl{t} \Uast{t} \cte_{\bQ \J} \\
        \bR_4(t) & :=\textstyle\intg{t} \Ustar{t} Q (\Iuno_t)                        \\
        \bR_5(t) & :=\textstyle\intg{t} \Ustar{t} Q (\Idos_t )                       \\
        \bR_6(t) & := \bbK^G(A^{(t)}, M_GB)                                          \\
        \bR_7(t) & := \bbK^G(A^{(t)}, M_{G_+}B)                                      \\
        \bE(t)   & := \bbK^G(A^{(t)}, B)
    \end{align*}
    Upon concatenating \eqref{eq:zzagmain} and \eqref{eq:zzagGstab} we get a zig-zag of $\OrGFin$-spectra as follows, where the morphisms labelled with $\sim$ are objectwise weak equivalences of spectra:
    \begin{equation}\label{eq:mainThmZzag}
        \xymatrix@C=2em{ \bF & \bR_1\ar[l]_-{\sim}\ar[r] & \bR_2 & \bR_3\ar[l]_-{\sim}\ar[r] & \bR_4 & \bR_5\ar[l]_-{\sim}\ar[r] & \bR_6\ar[r]^-{\sim} & \bR_7 & \bE\ar[l]_-{\sim}}
    \end{equation}
    Moreover, by Theorem \eqref{thm:ZZAG}, this zig-zag induces the isomorphism \eqref{eq:natAdjGroups} upon taking homotopy groups.

    Let $Z$ be a $\GFin$-complex. After applying $H_*^G(Z;-)$ to \eqref{eq:mainThmZzag} we get a zig-zag of graded abelian groups:
    \begin{equation}\label{eq:zzaghomology}\xymatrix@C=1.7em{
        H_*^G(Z; \bF) & H_*^G(Z; \bR_1)\ar[l]_-{\cong}\ar[r] & H_*^G(Z; \bR_2) & \cdots \ar[l]_-{\cong} & H_*^G(Z; \bE)\ar[l]_-{\cong}
        }\end{equation}
    Each arrow in \eqref{eq:zzaghomology} is natural in $Z$ since it is induced by a morphism of $\OrGFin$-spectra; see Remark \ref{rem:natHG}.
    Moreover, by Remark \ref{rem:natHG}, those arrows labelled with $\sim$ in \eqref{eq:mainThmZzag} induce natural isomorphisms in \eqref{eq:zzaghomology} and thus can be uniquely inverted. Upon inverting the natural isomorphisms in \eqref{eq:zzaghomology} we get a chain of natural transformations as follows:
    \begin{equation}\label{eq:beta}
        \xymatrix@C=1.7em{
        H_*^G(Z; \bF)\ar[r] & H_*^G(Z; \bR_1)\ar[r] & H_*^G(Z; \bR_2)\ar[r] & \cdots \ar[r] & H_*^G(Z; \bE)
        }
    \end{equation}
    Write $\beta_Z:H_*^G(Z;\bF)\to H_*^G(Z;\bE)$ for the composite of the morphisms above. It is clear that $\beta_Z$ is natural in $Z$ since each one of the morphisms appearing in \eqref{eq:beta} is.
    We claim, moreover, that $\beta_Z$ is an isomorphism. Since homology commutes with filtered unions, it suffices to prove this for $\GFin$-finite $Z$. The claim holds for $Z=G/N$ since, in this case, $\beta_Z$ is the isomorphism \eqref{eq:natAdjGroups} by Theorem \ref{thm:ZZAG}. Now we can continue as in the proof of Lemma \ref{lem:defiAlphaX}: the case $Z=G/H\times\Delta^n$ follows from homotopy invariance and the general case follows from excision upon considering the skeletal filtration of $Z$.
\end{proof}

Combining Lemma \ref{lem:defiAlphaX} and Lemma \ref{lem:homologiesCoincide} with the fact that homology commutes with filtered unions we get the main result of this paper:

\begin{thm}\label{thm:mainthm}
    Let $G$ be a group  satisfying \eqref{eq:hipG}. Let $A$ be an algebra, $B$ be a $G$-algebra and $Z$ be a $\GFin$-complex. Then there is a natural isomorphism
    \[H^G_*(Z; \bbK(A, \cR(B\rtimes -)))\cong \colim_{\substack{X\subseteq Z\\\text{$G$-finite}}} \kk^G_*(A^{(X)}, B)\]
    induced by a natural zig-zag of spectra.
\end{thm}

\begin{rem}\label{rem:mainthm}
    Let $G$ be a group  satisfying \eqref{eq:hipG} and let $B$ be a $G$-algebra. Define an $\OrG$-spectrum $\bKH_B$ by $\bKH_B(G/H):=\bbK(\ell,\cR(B\rtimes G/H))$ and note that it satisfies \eqref{eq:KHBproperty}. By Theorem \ref{thm:mainthm}, for every $\GFin$-complex $Z$ we have:
    \[H^G_*(Z;\bKH_B)\cong \colim_{\substack{X\subseteq Z \\ \text{$G$-finite}}}\kk^G_*(\ell^{(X)}, B)\]
    When taking $Z=\EFinG$, this shows that the domain of the $\KH$-assembly map can be expressed in terms of $\kk^G$-groups in a way that is completely analogous to that of the Baum-Connes assembly map.
\end{rem}

%%%%%%%%%%%%%%%%%%%%%%%%%%%%%%%%%%%%%%%%%%%%%%%%%%%%%%%%%%%%%%%%%%%%%%%%%%%%%%%%%%%%%
%%%%%%%%%%%%%%%%%%%%%%%%% Towards a kk-theoretic assembly map%%%%%%%%%%%%%%%%%%%%%%%%
%%%%%%%%%%%%%%%%%%%%%%%%%%%%%%%%%%%%%%%%%%%%%%%%%%%%%%%%%%%%%%%%%%%%%%%%%%%%%%%%%%%%%
\section{Towards a \texorpdfstring{$\kk$}{kk}-theoretic assembly map}\label{sec:kktheoreticASS}

\subsection{The Baum-Connes assembly map}\label{sec:BCass}
Let us briefly recall the definition of the Baum-Connes assembly map as formulated in \cite{bacohi}*{Section 9}. Fix a $G$-$C^*$-algebra $B$ and write $B\rtimes_rG$ for the reduced crossed product. Let $\EFinG$ be a topological model of the classifying space of $G$ with respect to $\FIN$. To every proper $G$-compact $G$-space $X$ one associates a canonical class $e_X\in \KK(\C, C_0(X)\rtimes G)$, where $C_0(X)$ denotes the algebra of continuous functions $X\to\C$ that vanish at infinity. Now consider the composites:
\[
    \xymatrix{
    \KK^G_*(C_0(X),B)\ar[r]^-{-\rtimes_rG} & \KK_*(C_0(X)\rtimes_r G,B\rtimes_r G)\ar[r]^-{e_X^*} & \KK_*(\C,B\rtimes_rG)
    }
\]
As $X$ varies over the closed $G$-compact subsets of $\EFinG$, these maps are compatible and define the \emph{Baum-Connes assembly map}:
\begin{equation}\label{eq:BCassKK}
    \colim_{\substack{X\subseteq \EFinG \\ \text{$G$-compact}}}\KK_*^G(C_0(X), B)\to \KK_*(\C, B\rtimes_rG)\cong \Ktop_*(B\rtimes_r G)
\end{equation}
The canonical class $e_X$ can be obtained as we proceed to explain. Let $C_c(X)$ denote the algebra of continuous functions $X\to\C$ with compact support. A \emph{cut-off function} for $X$ is a nonnegative $f\in C_c(X)$ such that $\sum_{g\in G}f(g\cdot x)=1$ for all $x\in X$; such functions always exist for proper and $G$-compact $X$. Let $f$ be a cut-off function for $X$ and define $p_f\in C_c(G\times X)$ by:
\[p_f(g, x):=\sqrt{f(x)f(g^{-1}\cdot x)}\]
If we consider $p_f \in C_0(X)\rtimes G$, then $p_f$ is a projection and its class $e_X=[p_f]\in \KK_0(\C, C_0(X)\rtimes G)$ is independent of the choice of $f$ \cite{land}*{Lemma 2.2}.

\begin{exa}
    Let $G$ be a (discrete) group and let $H\subseteq G$ be a finite subgroup. For every $uH\in G/H$ we have a cut-off function $f_{uH}:=\frac{1}{|H|}\chi_{uH}\in C_c(G/H)$. The corresponding projection is:
    \begin{equation}\label{eq:puH}p_{uH}:=\frac{1}{|H|}\chi_{uH}\rtimes \sum_{h\in H}uhu^{-1}\in C_0(G/H)\rtimes G\end{equation}
\end{exa}

\subsection{The \texorpdfstring{$\KH$}{KH}-assembly map}
Let $G$ be a group  satisfying \eqref{eq:hipG} and let $B$ be a $G$-algebra. Taking Remark \ref{rem:mainthm} into account, it is a natural question whether the $\KH$-assembly map admits a $\kk$-theoretic description analogous to \eqref{eq:BCassKK}. We could expect to find, for each $\GFin$-finite complex $X$, a natural class $e_X\in\kk_0(\ell, \ell^{(X)}\rtimes G)$ such that the composites
\[
    \xymatrix{
    \kk^G_*(\ell^{(X)},B)\ar[r]^-{-\rtimes G} & \kk_*(\ell^{(X)}\rtimes G,B\rtimes G)\ar[r]^-{e_X^*} & \kk_*(\ell,B\rtimes G)
    }
\]
are natural in $X$ and define a morphism $\mathcal{A}$ making the following diagram commute:
\[
    \xymatrix{H_*^G(\EFinG; \bbK(\ell, \cR(B\rtimes -))) \ar[r]^-{\pr} \ar@{-->}[d]_-{\text{Thm. \ref{thm:mainthm}}}^-{\cong} & \kk_*(\ell, B\rtimes G)\cong\KH_*(B\rtimes G)\\
    \displaystyle\colim_{\substack{X\subseteq \EFinG \\ \text{$G$-compact}}} \kk_*^G(\ell^{(X)},B) &
    \ar@/_1pc/ "2,1";"1,2"+<-2.8em,-0.8em>_-{\mathcal{A}}
    }
\]
The authors have no results in this generality although it is clear how to define $e_X$ for $X=G/H$, as we explain below.

Let $H\subseteq G$ be a finite subgroup. The formula \eqref{eq:puH} defines an idempotent
\begin{equation}\label{eq:pH}
    p_{uH}=\frac{1}{|H|}\chi_{uH}\rtimes\sum_{h\in H}uhu^{-1}\in\ell^{(G/H)}\rtimes G
\end{equation}
and thus an element $[p_{uH}]\in\kk_0(\ell, \ell^{(G/H)}\rtimes G)$. It is immediate that the latter does not depend upon $uH$ since we have:
\[(1\rtimes g)p_H(1\rtimes g^{-1})=p_{uH}\in \ell^{G/H}\rtimes G\]

\begin{lem}
    Let $G$ be a group satisfying \eqref{eq:hipG} and let $B$ be a $G$-algebra. Let $H\subseteq G$ be a finite subgroup and let $p_H\in \ell^{(G/H)}\rtimes G$ be the idempotent defined in \eqref{eq:pH}. Then the following diagram commutes:
    \begin{equation}\label{eq:lemapH}
        \begin{gathered}
            \xymatrix{
            H_*^G(G/H; \bbK(\ell, \cR(B\rtimes -))) \ar[r]^-{\pr} \ar@{-->}[d]_-{\text{Thm. \ref{thm:mainthm}}}^-{\cong} & \kk_*(\ell, B\rtimes G)\cong\KH_*(B\rtimes G)\\
            \kk_*^G(\ell^{(G/H)},B) &
            \ar@/_1pc/ "2,1";"1,2"+<-2.8em,-0.8em>_-{(-\circ [p_{H}])\circ (- \rtimes G)}
            }
        \end{gathered}
    \end{equation}
\end{lem}
\begin{proof}
    For every algebra $A$ we have a morphism $p_H\otimes A:A\to A^{(G/H)}\rtimes G$. We will prove that the following diagram commutes for every algebra $A$; note that the commutativity of the outer square for $A=\ell$ is equivalent to that of \eqref{eq:lemapH}.
    \begin{equation}\label{eq:lemapHd1}
        \begin{gathered}
            \xymatrix@C=3em{
            H^G_*(G/H; \bbK(A, \cR(B\rtimes -)))\ar[dd]^-{\cong}_-{\text{Thm. \ref{thm:mainthm}}}\ar[rr]^-{\pr} & & H^G_*(G/G; \bbK(A, \cR(B\rtimes -)))\ar[d]^-{\cong} \\
            & \save[]+<-2em,0em>*+{\kk_*(A, \cR(B\rtimes G/H))}\ar[dl]_-{\cong}^-{\eqref{eq:natAdjGroups}}\ar[r]^-{\pr}\ar@{<-}[ul]_-{\cong}\restore & \kk_*(A, \cR(B\rtimes G/G))\ar@{=}[d] \\
            \kk^G_*(A^{(G/H)}, B)\ar[rr]^-{(p_H\otimes A)^*\circ (-\rtimes G)} & & \kk_*(A, B\rtimes G)
            }
        \end{gathered}
    \end{equation}
    To see  that the triangle in \eqref{eq:lemapHd1} commutes, recall that the isomorphism of Theorem \ref{thm:mainthm} is induced by the zig-zag of $\OrGFin$-spectra from Theorem \ref{thm:ZZAG} and that the latter recovers the isomorphism \eqref{eq:natAdjGroups} upon taking homotopy groups. As the upper right square in \eqref{eq:lemapHd1} clearly commutes, we have to prove the commutativity of the lower right one. After identifying $\kk_*(A, \cR(B\rtimes G/H))\cong \kk_*(A, B\rtimes H)$ using \eqref{eq:jrtimesH} and unravelling the details of the isomorphism \eqref{eq:natAdjGroups}, we are left to show that the following diagram commutes for all $n$:
    \begin{equation}\label{eq:lemapHd2}
        \begin{gathered}
            \xymatrix@C=8em{
            \kk_n(A, B\rtimes H)\ar[r]^-{\incl_*}\ar[d]_-{(-)^{(G/H)}} & \kk_n(A, B\rtimes G)\ar[d]_-{\iota_*}^-{
            \scriptsize{{\begin{array}{@{}c@{}}
                            \cong \\
                            \eqref{eq:zigzag}
                        \end{array}}}
            } \\
            \ar[d]_-{(\psi_B)_*}^-{\eqref{eq:defipsi}}
            \kk^G_n(A^{(G/H)}, (R\rtimes H)^{(G/H)}) & \kk_n(A, (M_{G_\plus}B)\rtimes G) \\
            \kk^G_n(A^{(G/H)}, M_GB)\ar[r]^-{(p_H\otimes A)^*\circ (-\rtimes G)} & \kk_n(A, (M_GB)\rtimes G)\ar[u]^-{(\iota')_*}_-{
            \scriptsize{{\begin{array}{@{}c@{}}
                            \cong \\
                            \eqref{eq:zigzag}
                        \end{array}}}
            }
            }
        \end{gathered}
    \end{equation}
    It suffices to consider the case $n=0$ since the general case follows upon replacing $A$ by $\Sigma^nA$ if $n>0$ and by $\Omega^{-n}A$ if $n<0$; see \cite{cortho}*{Corollary 6.4.2}.
    Let $\alpha\in\kk(A, B\rtimes H)$ and put:
    \begin{align*}
        \alpha_1 & :=\left(\iota_*\circ\incl_*\right)(\alpha)                                                                 \\
        \alpha_2 & :=\left((\iota')_*\circ (p_H\otimes A)^*\circ (-\rtimes G)\circ (\psi_B)_*\circ (-)^{(G/H)}\right)(\alpha)
    \end{align*}
    We will show that $\alpha_1=\alpha_2$. Recall from \eqref{eq:R} that we have an isomorphism:
    \[(R_{G_\plus,B})_*:\kk(A, (M_{G_\plus}B)\rtimes G)\to \kk(A, M_{|G_\plus|}(B\rtimes G))\]
    We claim that $(R_{G_\plus,B})_*(\alpha_1)=(R_{G_\plus,B})_*(\alpha_2)$. First suppose that $\alpha$ is represented by an algebra homomorphism $f:A\to B\rtimes H$. Fix $a\in A$ and write:
    \[f(a)=\sum_{k\in H}b_k\rtimes k\]
    It is straightforward to verify that $(R_{G_+,B})_*(\alpha_i)$ is represented by the algebra homomorphism $f_i:A\to M_{|G_+|}(B\rtimes G)$, where $f_1$ and $f_2$ are given by:
    \begin{align*}
        f_1(a) & = \displaystyle\sum_{k\in H}e_{\plus,\plus}\otimes (b_k\rtimes k)                       \\
        f_2(a) & =\displaystyle\frac{1}{|H|}\sum_{p,t,k\in H}e_{p,t}\otimes (p\cdot b_k\rtimes pkt^{-1})
    \end{align*}
    An easy computation shows that
    \[(V_H\otimes 1)V^{-1} f_2 V (V^{-1}_H\otimes 1) =f_1\]
    where $V$ is defined in \eqref{eq:V} and $V_{H}$ in \eqref{eq:vsH}. Thus, $f_1$ and $f_2$ induce the same morphism in $\kk$.
    Suppose now that $\alpha$ is represented by an algebra homomorphism $f:\J^rA\to (B\rtimes H)^{S^r}$. Let $\Omega:\kk\to\kk$ denote the translation functor in $\kk$ and recall that we have natural isomorphisms as follows; see \cite{cortho}*{Lemma 6.3.11} and \cite{loopsthtpy}*{Lemma 7.10}:
    \[\kk(A, B)\cong \kk(\Omega^rA, \Omega^rB)\cong \kk(\J^rA, B^{S^r})\]
    Since these are compatible with all the morphisms appearing in \eqref{eq:lemapHd2}, we may replace $A$ by $\J^rA$, $B$ by $B^{S^r}$ and reduce to the case $r=0$ that we have already addressed.
\end{proof}

\appendix

%%%%%%%%%%%%%%%%%%%%%%%%%%%%%%%%%%%%%%%%%%%%%%%%%%%%%%%%%%%%%%%%%%%%%%%%%%%%%%%%%%%%%
%%%%%%%%%%%%%%%%%%%%%%%%%%%%%% G-Simplicial sets %%%%%%%%%%%%%%%%%%%%%%%%%%%%%%%%%%%%
%%%%%%%%%%%%%%%%%%%%%%%%%%%%%%%%%%%%%%%%%%%%%%%%%%%%%%%%%%%%%%%%%%%%%%%%%%%%%%%%%%%%%
\section{\texorpdfstring{$G$}{G}-Simplicial sets}\label{sec:Gsset}

Let $G$ be a group. We recall some definitions and properties concerning the $G$-simplicial sets. A \emph{$G$-simplicial set} is a simplicial set with a left action of $G$. We write $\S^G$ for the category of $G$-simplicial sets with equivariant morphisms. Every $G$-simplicial set $X$ has a skeletal filtration such that the $n$-skeleton $\sk_nX$ is obtained from $\sk_{n-1}X$ upon attaching cells of the form $G/H\times \Delta^n$ with $H$ a subgroup of $G$. We say that $X$ is \emph{$G$-finite} if $X$ can be built from a finite number of these cells; it is easily verified that $X$ is $G$-finite if and only if $G\backslash X$ is a finite simplicial set. Let $\cF$ be a nonempty family of subgroups of $G$ closed under conjugation and subgroups; we are interested in the family $\FIN$ of finite subgroups. A $G$-simplicial set $X$ is called a \emph{$\GF$-complex} if $X$ can be built from cells of the form $G/H\times\Delta^n$ with $H\in\cF$. The $\GF$-complexes are the cofibrant objects for a certain model structure on $\S^G$; see \cite{corel}*{Proposition 2.3}. A $G$-simplicial set $X$ is called \emph{$\GF$-finite} if it can be built from a finite number of cells of the form $G/H\times\Delta^n$ with $H\in\cF$. It is easily verified that $X$ is $\GF$-finite if and only if $X$ is a $G$-finite $\GF$-complex. In the rest of this section we gather some technical results that are used in Section \ref{sec:prin}.

\begin{lem}\label{lem:prodCat}
    Let $G$ be a group and let $\S_c\subset \S$ denote the full subcategory of connected simplicial sets. Let $G/H, G/K\in\OrG$ and $X,Y\in\S_c$. Then there is a natural isomorphism:
    \[\Hom_{\S^G}(G/H\times X, G/K\times Y)\cong \Hom_\OrG(G/H, G/K)\times \Hom_{\S}(X, Y)\]
    In other words, the full subcategory of $\S^G$ whose objects are $G/H\times X$ with $G/H\in\OrG$ and $X\in\S_c$ is equivalent to the product category $\OrG\times \S_c$.
\end{lem}
\begin{proof}
    Let $f:G/H\times X\to G/K\times Y$ be a morphism in $\S^G$. We claim that there exist a unique coset $uK\in G/K$ and a unique morphism $h:X\to Y$ that fit into a commutative square as follows, as we proceed to explain.
    \begin{equation}\label{eq:prodcat}\begin{gathered}\xymatrix{\displaystyle\bigsqcup_{G/H} X\ar[r]^-{f} &
            \displaystyle\bigsqcup_{G/K} Y\\
            X\ar[u]^-{\can_H}\ar[r]^-{h} &
            Y\ar[u]_-{\can_{uK}}}\end{gathered}\end{equation}
    Since $Y$ is a connected simplicial set, the set of connected components of $\bigsqcup_{G/K}Y$ is $\{\can_{uK}(Y)\}_{uK\in G/K}$. Since $X$ is a connected simplicial set, there is a unique connected component $\can_{uK}(Y)$ of $\bigsqcup_{G/K}Y$ such that $(f\circ\can_H)(X)\subseteq \can_{uK}(Y)$. Now define $h$ as the composite:
    \[\xymatrix@C=3em{X\ar[r]^-{f\circ \can_H} & \can_{uK}(Y)\ar[r]^-{(\can_{uK})^{-1}}_-{\cong} & Y}\]
    It is clear that $h$ makes the square \eqref{eq:prodcat} commute.
    Moreover, it follows from the equivariance of $f$ that $g(tH)=tuK$ defines a morphism $g:G/H\to G/K$. Conversely, every pair $(g,h)\in \Hom_\OrG(G/H, G/K)\times \Hom_\S(X, Y)$ defines a unique $G$-equivariant morphism $f$ making the following squares commute for all $t$:
    \[\xymatrix{\displaystyle\bigsqcup_{G/H} X\ar[r]^-{f} &
        \displaystyle\bigsqcup_{G/K} Y\\
        X\ar[u]^-{\can_{tH}}\ar[r]^-{h} &
        Y\ar[u]_-{\can_{g(tH)}}}\]
    It is easily verified that both constructions are mutually inverse.
\end{proof}

\begin{lem}\label{lem:gfcofi}
    Let $\cF$ be a family of subgroups of $G$ and let $\psi:Y\to X$ be a morphism of $G$-simplicial sets. If $X$ is a $\GF$-complex then $Y$ is a $\GF$-complex too.
\end{lem}
\begin{proof}
    Let $\sigma\in Y_n$; it is easily verified that $\Stab(\sigma)\subseteq\Stab(\psi(\sigma))\in\cF$.
\end{proof}

\begin{lem}\label{lem:proper}
    Let $G$ be an infinite group, let $H\subseteq G$ be a subgroup, let $X$ be a $\GFin$-complex and let $K$ be a finite simplicial set. Then every $G$-equivariant morphism $\psi:G/H\times K\to X$ is proper, i.e. $\psi^{-1}(L)$ is a finite simplicial set for every finite simplicial subset $L\subseteq X$.
\end{lem}
\begin{proof}
    First of all notice that $H\in\FIN$ by Lemma \ref{lem:gfcofi}. Let $L\subseteq X$ be a finite simplicial set and suppose that $\psi^{-1}(L)$ is not finite. Then there is an infinite number of non-degenerate simplices in $\psi^{-1}(L)\subseteq G/H\times K$. Since every non-degenerate simplex of $G/H\times K$ has dimension $\leq d:=\dim K$, there exists $0\leq p\leq d$ such that there is an infinite number of non-degenerate $p$-simplices in $\psi^{-1}(L)$. Let $\{g_i,i\in I\}\subseteq G$ be a system of representatives for the cosets in $G/H$; notice that $I$ is infinite because $H$ is finite. Every non-degenerate $p$-simplex of $G/H\times K$ is of the form $(g_iH,\sigma)$ for some $i\in I$ and some non-degenerate $p$-simplex $\sigma$ of $K$. Since $K$ has finitely many non-degenerate $p$-simplices, there exist a non-degenerate $p$-simplex $\sigma$ of $K$ and an infinite subset $J\subseteq I$ such that $\{(g_iH,\sigma),i\in J\}\subseteq \psi^{-1}(L)$. Then
    \[\{\psi(g_iH,\sigma),i\in J\}\subseteq L_p.\]
    Since $L_p$ is a finite set, replacing $J$ by a smaller but still infinite subset, we can assume without loss of generality that there is $\tau\in L_p$ such that $\psi(g_iH,\sigma)=\tau$ for every $i\in J$. Fix $i_0\in J$. Then
    \[g_{i_0}\cdot\psi(H,\sigma)=\psi(g_{i_0}H,\sigma)=\tau=\psi(g_iH,\sigma)=g_i\cdot \psi(H,\sigma)\]
    for every $i\in J$ and it follows that $\{g_i^{-1}g_{i_0},i\in J\}\subseteq\Stab(\psi(H,\sigma))\in \FIN$; this is a contradiction since $J$ is infinite.
\end{proof}

\begin{lem}\label{lem:proper2}
    Let $X$ be a $\GFin$-complex and $Y$ be a $\GFin$-finite simplicial set. Then every morphism $\phi:Y\to X$ is proper.
\end{lem}
\begin{proof}
    Let $K\subseteq X$ be a finite simplicial subset and suppose that $\phi^{-1}(K)$ is not finite. Note that $\phi^{-1}(K)$ has finite dimension since $\phi^{-1}(K)\subseteq Y$ and $Y$ has finite dimension. It follows that, for some $n\geq 0$, $\phi^{-1}(K)$ has infinitely many non-degenerate $n$-simplices. Let $\{\sigma_k\}_{k\in\N}\subseteq \phi^{-1}(K)$ be a list of distinct non-degenerate $n$-simplices. Consider a pushout diagram
    \[\begin{gathered}\xymatrix{
                \bigsqcup_{i=1}^rG/H_i\times \partial\Delta^n\ar[r]\ar[d] & \sk_{n-1}Y\ar[d] \\
                \bigsqcup_{i=1}^rG/H_i\times \Delta^n\ar[r] &  \sk_nY\\
            }\end{gathered}
    \]
    where the $H_i$ are finite. For every $k\in\N$, there exist $1\leq i_k\leq r$ and a non-degenerate $n$-simplex $\tau_k$ of $G/H_{i_k}\times \Delta^n$ such that $\sigma_k$ is the image of $\tau_k$ under $G/H_{i_k}\times \Delta^n\to \sk_nY$. Then there exists some $1\leq j\leq r$ such that $i_k=j$ for infinitely many values of $k$. We may assume without loss of generality that $i_k=j$ for all $k$. Let $\psi$ be the composite:
    \[\xymatrix{G/H_j\times \Delta^n\ar[r] & \sk_nY\ar[r]^-{\incl} & Y\ar[r]^-{\phi} & X}\]
    Then $\{\tau_k\}_{k\in\N}$ is a list of distinct non-degenerate $n$-simplices of $\psi^{-1}(K)\subseteq G/H_j\times \Delta^n$. But the latter is not possible since $\psi$ is proper by Lemma \ref{lem:proper}.
\end{proof}

\begin{lem}\label{lem:fixcolim}
    Let $G$ be a group, $K\subseteq G$ be a subgroup and $X$ be a $G$-simplicial set. Then:
    \[X^K\cong \colim_{G/H\times\Delta^n\downarrow X}(G/H\times\Delta^n)^K\]
\end{lem}
\begin{proof}
    There is a natural morphism:
    \begin{equation}\label{eq:fixcolim}
        \colim_{G/H\times\Delta^n\downarrow X}(G/H\times\Delta^n)^K\to X^K\end{equation}
    Let us prove that it is surjective. Let $\sigma\in(X^K)_p=(X_p)^K$ and put $H:=\Stab(\sigma)$; note that we have $K\subseteq H$. Let $f:G/H\times\Delta^p\to X$ be the $G$-equivariant morphism determined by $(H, \iota_p)\mapsto \sigma$. Since $(H,\iota_p)\in(G/H\times \Delta^p)^K$, we have that $\sigma = f^K(H, \iota_p)$, showing that $\sigma$ is in the image of \eqref{eq:fixcolim}. We still have to prove that \eqref{eq:fixcolim} is injective. Let us first show that every $p$-simplex of
    \begin{equation}\label{eq:fixcolim2}\colim_{G/H\times\Delta^n\downarrow X}(G/H\times\Delta^n)^K\end{equation}
    is represented by one of the form $(L,\iota_p)\in (G/L\times \Delta^p)^K$. Let $f:G/H\times \Delta^n\to X$ be a $G$-equivariant morphism and let $(gH,\tau)$ be a $p$-simplex of $(G/H\times \Delta^n)^K$. Then $K\subseteq gHg^{-1}$ and $\tau=\tau_*(\iota_p)$ for some nondecreasing function $\tau:[p]\to [n]$. The commutativity of the following triangle implies that $(gH, \tau)$ and $(gH, \iota_p)$ represent the same simplex of \eqref{eq:fixcolim2}:
    \[\xymatrix{(gH, \tau) & G/H\times\Delta^n\ar[rr]^-{f} & & X \\
        (gH,\iota_p)\ar@{|->}[u] & G/H\times \Delta^p\ar[u]^-{\id\times \tau_*}\ar@/_1pc/[urr]_-{f\circ (\id\times \tau_*)} & &
        }\]
    Write $L:=gHg^{-1}$ and note that there is a $G$-equivariant bijection $\beta:G/H\to G/L$ determined by $\beta(H)=g^{-1}L$. The commutativity of the following triangle implies thata $(gH,\iota_p)$ and $(L,\iota_p)$ represent the same simplex of \eqref{eq:fixcolim2}:
    \[\xymatrix{(gH, \iota_p)\ar@{|->}[d] & G/H\times\Delta^p\ar[d]^-{\cong}_-{\beta\times \id}\ar[rr]^-{f\circ (\id\times \tau_*)} & & X \\
        (L,\iota_p) & G/L\times \Delta^p\ar@/_1pc/[urr] & &
        }\]
    Now let $(L,\iota_p)$ and $(M,\iota_p)$ represent two $p$-simplices of \eqref{eq:fixcolim2} having the same image $\sigma\in(X^K)_p$ under \eqref{eq:fixcolim}. Put $S:=\Stab(\sigma)$ and note that $L,M\subseteq S$. The commutativity of the following diagram shows that $(L,\iota_p)$ and $(M,\iota_p)$ represent the same simplex of \eqref{eq:fixcolim}:
    \[\xymatrix{(L, \iota_p)\ar@{|->}[d] & G/L\times\Delta^p\ar[d]\ar@/^1pc/[drr] & & \\
        (S,\iota_p) & G/S\times \Delta^p\ar[rr] & & X
        \\
        (M,\iota_p)\ar@{|->}[u] & G/M\times\Delta^p\ar@/_1pc/[urr]\ar[u] & &}
    \]
    This finishes the proof.
\end{proof}

%%%%%%%%%%%%%%%%%%%%%%%%%%%%%%%%%%%%%%%%%%%%%%%%%%%%%%%%%%%%%%%%%%%%%%%%%%%%%%%%%%%%%
%%%%%%%%%%%%%%%%%%%%% A Mayer-Vietoris sequence in kk-theory %%%%%%%%%%%%%%%%%%%%%%%%
%%%%%%%%%%%%%%%%%%%%%%%%%%%%%%%%%%%%%%%%%%%%%%%%%%%%%%%%%%%%%%%%%%%%%%%%%%%%%%%%%%%%%
\section{A Mayer-Vietoris sequence in \texorpdfstring{$\kk$}{kk}-theory }

\begin{defi}
    A \emph{Milnor square of $G$-algebras} is a pullback square
    \[
        \xymatrix{A\ar[r]\ar[d] & B\ar[d]^-{f} \\
        C\ar[r] & D}
    \]
    where $f$ is surjective and has a $G$-linear section.
\end{defi}

\begin{lem}\label{lem:milnor}
    Let $A$ be an algebra and let $X$ be a $\GFin$-finite $G$-simplicial set. For each $n\geq 1$ there is a pushout diagram
    \begin{equation}\label{eq:poskeleta}\begin{gathered}\xymatrix{
                \bigsqcup_{i=1}^rG/H_i\times \partial\Delta^n\ar[r]\ar[d] & \sk_{n-1}X\ar[d] \\
                \bigsqcup_{i=1}^rG/H_i\times \Delta^n\ar[r] & \sk_nX\\
            }\end{gathered}
    \end{equation}
    of $G$-simplicial sets with $H_i\in \FIN$ for all $i$. Then all the morphisms appearing in \eqref{eq:poskeleta} are proper and this diagram induces a Milnor square of $G$-algebras:
    \begin{equation}\label{eq:pbskeleta}\begin{gathered}\xymatrix{
                A^{\left(\bigsqcup_{i=1}^rG/H_i\times \partial\Delta^n\right)} & A^{(\sk_{n-1}X)}\ar[l] \\
                A^{\left(\bigsqcup_{i=1}^rG/H_i\times \Delta^n\right)}\ar[u] & A^{(\sk_nX)}\ar[u]\ar[l]\\
            }\end{gathered}
    \end{equation}
\end{lem}
\begin{proof}
    All the morphisms in \eqref{eq:poskeleta} are proper by Lemma \ref{lem:proper2}. Then we can apply $A^{(-)}$ to get a commutative diagram of $G$-algebras like \eqref{eq:pbskeleta} that it is easily seen to be a pullback; see Remark \ref{rem:polynomialfcoprod}. Write $i:\partial\Delta^n\to\Delta^n$ for the inclusion. By Lemma \ref{lem:cortholX}, the morphism $i^*:A^{ \Delta^n}\to A^{\partial\Delta^n}$ admits a linear section. Then the left vertical morphism in \eqref{eq:pbskeleta} admits a $G$-linear section, since it identifies with
    \[\bigoplus_i \ell^{(G/H_i)}\otimes i^*:\bigoplus_i\ell^{(G/H_i)}\otimes A^{\Delta^n}\longrightarrow \bigoplus_i\ell^{(G/H_i)}\otimes A^{\partial\Delta^n}\]
    by Remark \ref{rem:polynomialfcoprod}.
\end{proof}

\begin{lem}\label{lem:excision}
    Let $E$ be a $G$-algebra. Then every Milnor square of $G$-algebras
    \[\xymatrix{A\ar[r]\ar[d] & B\ar[d] \\
            C\ar[r] & D}\]
    induces a long exact Mayer-Vietoris sequence:
    \[\xymatrix{\kk^G_*(D,E)\ar[r] & \kk^G_*(B,E)\oplus \kk^G_*(C,E)\ar[r] & \kk^G_*(A,E)\ar[r] & \kk^G_{*-1}(D,E)}
    \]
\end{lem}
\begin{proof}
    It follows from excision in $\kk^G$ \cite{euge}*{Theorem 4.1.1}
    and from the argument explained in \cite{ralf}*{Theorem 2.41}.
\end{proof}

%%%%%%%%%%%%%%%%%%%%%%%%%%%%%%%%%%%%%%%%%%%%%%%%%%%%%%%%%%%%%%%%%%%%%%%%%%%%%%%%%%%%%
%%%%%%%%% The model category of spectra and spectra representing kk-theory %%%%%%%%%%
%%%%%%%%%%%%%%%%%%%%%%%%%%%%%%%%%%%%%%%%%%%%%%%%%%%%%%%%%%%%%%%%%%%%%%%%%%%%%%%%%%%%%
\section{The model category of spectra and spectra representing \texorpdfstring{$\kk$}{kk}-theory}\label{sec:mcs}

\subsection{The stable model category of spectra}

In this section we recall the definition of the stable model category of spectra and discuss some of its properties.

\begin{defi}[\cite{bousfried}*{Definition 2.1}]
    A \emph{spectrum} $X$ is a sequence of pointed simplicial sets $X^0, X^1, X^2, \dots$ together with pointed morphisms $S^1\wedge X^n\to X^{n+1}$ for all $n$, called \emph{bonding maps}. Here, $S^1=\Delta^1/\partial \Delta^1$. A \emph{morphism of spectra} $f:X\to Y$ is a sequence of pointed morphisms $f^n:X^n\to Y^n$ that commute with the bonding maps:
    \[\xymatrix{S^1\wedge X^n\ar[d]_-{S^1\wedge f^n}\ar[r] & X^{n+1}\ar[d]^-{f^{n+1}} \\
        S^1\wedge Y^n\ar[r] & Y^{n+1}}\]
    We write $\Spt$ for the category of spectra and morphisms of spectra.
\end{defi}

We endow $\Spt$ with its stable model structure, that we proceed to describe; see \cite{bousfried}*{Section 2} and \cite{hoveyspectra}*{Section 3} for details:
\begin{itemize}
    \item Let $X$ be a spectrum and let $m\in\Z$. The \emph{$m$-th stable homotopy group} of $X$ is defined as:
          \[\pi_m(X)=\colim_k \pi_{m+k}(X^k)\]
          A morphism of spectra $f:X\to Y$ is a \emph{weak equivalence} if $\pi_m(f)$ is an isomorphism for all $m\in\Z$.
    \item A morphism of spectra $f:X\to Y$ is a \emph{fibration} if $f^n:X^n\to Y^n$ is a fibration of simplicial sets for all $n$.
    \item A morphism of spectra is a \emph{cofibration} if it has the left lifting property with respect to trivial fibrations.
\end{itemize}

\begin{lem}[\cite{rosicky}*{Example 3.6 (iii)}]\label{lem:sptCombi}
    The stable model structure on $\Spt$ is combinatorial.
\end{lem}
\begin{proof}
    Recall from \cite{dugger}*{Definition 2.1} that a model category is \emph{combinatorial} if it is cofibrantly generated and its underlying category is locally presentable.
    The stable model structure on $\Spt$ is cofibrantly generated by \cite{hoveyspectra}*{Definition 3.3 and Corollary 3.5}. We have to show that the underlying category is locally presentable. We claim that $\Spt$ is locally \emph{finitely} presentable. By \cite{ros}*{Theorem 1.11}, to prove this claim we have to show that $\Spt$ has a strong generator formed by finitely presentable objects. By \cite{ros}*{Example 1.12}, the set $\{\Delta^m:m\geq 0\}$ is a strong generator for $\S$ formed by finitely presentable objects. Since the forgetful functor $\S_*\to\S$ commutes with filtered colimits, the latter is easily seen to imply that $\{\Delta^m_+:m\geq 0\}$ is a strong generator for $\S_*$ formed by finitely presentable objects. For $n\geq 0$, let $F_n:\S_*\to\Spt$ be the left adjoint to evaluation at $n$. Explicitely, for a pointed simplicial set $X$, let $F_n(X)$ be the spectrum whose level $k$ is $(S^1)^{\wedge(k-n)}\wedge X$ if $k\geq n$ and $\ast$ otherwise. The bonding maps are the obvious ones. It is easily verified that $\{F_n(\Delta^m_+):m,n\geq 0\}$ is a strong generator for $\Spt$ formed by finitely presentable spectra.
\end{proof}

\subsection{Bivariant \texorpdfstring{$K$}{K}-theory spectra}\label{sec:sptkkg}

In this section we recall the definitions of spectra representing $\kk$-theory \cite{garku}*{Theorem 9.8} and $\kk^G$-theory \cite{tesisema}*{Theorem 5.3.11}.

Let $\cC$ denote either $\Algl$ or $\GAlgl$. For two objects $A$ and $B$ of $\cC$, the \emph{bivariant $K$-theory space} of the pair $(A,B)$  \cite{htpysimp}*{Definition 4.10} is defined as the fibrant simplicial set:
\[\scrK(A,B):=\colim_n \Omega^n\Ex^\infty \Hom_\cC(J^nA, B^\Delta)\]
This definition is equivalent to the original one given in \cite{garku}*{Section 4}. By \cite{garku}*{Theorem 5.1} there is a natural isomorphism of simplicial sets:
\[\scrK(A,B)\cong \Omega\scrK(JA, B)\]
Thus, we have an $\Omega$-spectrum $\bbK(A,B)$ defined by the sequence:
\[\scrK(A,B),\; \scrK(JA, B),\; \scrK(J^2A, B),\; \dots\]
The spectrum $\bbK(A,B)$ (denoted by $\bbK^{\mathrm{unst}}(A,B)$ in \cite{garku}) represents a universal bivariant $K$-theory introduced by Garkusha that is excisive, homotopy invariant but matrix-unstable \cite{garku}*{Comparison Theorem B}. Different matrix-stabilizations can be perfomed in order to obtain spectra representing $\kk$- and $\kk^G$-theories:
\begin{enumerate}
    \item \emph{Stabilization by finite matrices.} For two objects $A$ and $B$ of $\cC$ put
          \[\bbK_f(A,B):=\colim_n \bbK(A, M_nB)\]
          where the transition maps are induced by the inclusion $M_nB\to M_{n+1}B$ into the upper left corner. These spectra represent a universal bivariant $K$-theory that is excisive, homotopy invariant and stable by finite matrices \cite{garku}*{Theorem 9.8}; see \cite{garku}*{Section 9} and \cite{tesisema}*{Section 5.1} for details.
    \item \emph{Stabilization by finite matrices indexed on an infinite set.} Let $X$ be an infinite set. For two objects $A$ and $B$ of $\cC$ put \cite{tesisema}*{Definition 5.2.21}:
          \[\bbK_X (A,B):=\bbK_f(A, M_X B)\]
          The spectra $\bbK_X(A,B)$ represent a universal bivariant $K$-theory that is excisive, homotopy invariant and $M_X$-stable \cite{tesisema}*{Theorem 5.2.22}. For any $X$, Weibel's homotopy $K$-theory $\KH$ is the functor represented by the base ring $\ell$ \cite{tesisema}*{Theorem 5.2.20}. For $X=\N$, this theory coincides with the $\kk$-theory defined in \cite{cortho}.
    \item \emph{$G$-stabilization. } Let $G$ be a group and let $X=\N\times |G|$. For two $G$-algebras $A$ and $B$ put:
          \[\bbK^G(A, B):=\bbK_X(M_GA, M_GB)\]
          These spectra represent $\kk^G$-theory \cite{tesisema}*{Theorem 5.3.11}.
\end{enumerate}

\begin{lem}[cf. \cite{tesisema}*{Section 4.4}]\label{lem:mult}
    Let $\S_f\subset \S$ denote the full subcategory of finite simplicial sets. Let $G$ be a group, $X$ be an infinite set, $G/K\in\OrGFin$, $A, B\in\GAlgl$ and $S\in\S_f$. Then, for $\bbE\in \{\bbK, \bbK_f, \bbK_X, \bbK^G\}$, there is a morphism of spectra
    \begin{equation}\label{eq:lemMult}S_+\wedge \bbE(A^{(G/K)}, B)\to \bbE(A^{(G/K\times S)}, B)\end{equation}
    that is natural in $A$, $B$, $G/K$ and $S$. Moreover, for $S=\Delta^0$ this is an isomorphism.
\end{lem}
\begin{proof}
    It suffices to prove the Lemma for $\bbE=\bbK$: the case $\bbE=\bbK_f$ follows from this upon taking colimit over the inclusions $M_nB\to M_{n+1}B$ and the rest of the cases follow from the latter upon replacing $A$ and $B$ with appropriate matrix algebras.

    Let us prove the case $\bbE=\bbK$. We will define \eqref{eq:lemMult} levelwise. At level $p\geq 0$, we have to define a morphism of simplicial sets:
    \begin{equation}\label{eq:lemMultiLevel}S_+\wedge \scrK(J^p(A^{(G/K)}),B)\to \scrK(J^p(A^{(G/K\times S)}), B)\end{equation}
    Let us describe it in dimension $q\geq 0$. A $q$-simplex of $S_+\wedge \scrK(J^p(A^{(G/K)}),B)$ is represented by a pair $(\sigma,\alpha)$ where $\sigma$ is a $q$-simplex of $S$ and $\alpha$ is a $q$-simplex of $\scrK(J^p(A^{(G/K)}),B)$.
    Below we use the notation of \cite{htpysimp}*{Section 2.9}.
    Let $\alpha$ be represented by a $G$-algebra homomorphism
    \[\alpha:J^{p+v}(A^{(G/K)})\to B^{(I^v\times \Delta^q, \partial I^v\times \Delta^q)}_r\]
    for some $v,r\geq 0$. Then the morphism \eqref{eq:lemMultiLevel}
    sends the pair $(\sigma, \alpha)$ to the $q$-simplex of $\scrK(J^p(A^{(G/K\times S)}), B)$ represented by the following composite in $\GAlgl$:
    \[\xymatrix@R=1.5em{
        J^{p+v}\left(A^{(G/K\times S)}\right)\ar@{=}[r] & J^{p+v}\left(\left(A^{(G/K)}\right)^{S}\right)\ar[r]^-{\clas} &
        \left[J^{p+v}\left(A^{(G/K)}\right)\right]^{S}\ar[d]^-{\alpha_*} \\
        B^{(I^v\times\Delta^q\times \Delta^q, \partial I^v\times \Delta^q\times \Delta^q)}_r\ar[d]^-{\diag^*} & \left[B^{(I^v\times \Delta^q, \partial I^v\times\Delta^q)}_r \right]^{\Delta^q}\ar[l]_-{\mu} &
        \left[B^{(I^v\times \Delta^q, \partial I^v\times\Delta^q)}_r\right]^{S}\ar[l]_-{\sigma^*} \\
        B^{(I^v\times \Delta^q, \partial I^v\times \Delta^q)}_r & &
        }\]
    Here, $\mu$ is the morphism defined in \cite{htpysimp}*{Remark 3.4}.
    This clearly defines a morphism \eqref{eq:lemMultiLevel} that is natural in $A$, $B$, $G/K$ and $S$. Let us now show that \eqref{eq:lemMultiLevel} is an isomorphism for $S=\Delta^0$. First note that the classifying map
    \[\clas:J^{p+v}\left(\left(A^{(G/K)}\right)^{\Delta^0}\right)\to
        \left[J^{p+v}\left(A^{(G/K)}\right)\right]^{\Delta^0}\]
    is an isomorphism. Moreover, it follows from the naturality of $\mu$ that the composite $\diag^*\circ\mu\circ \sigma^*$ equals the obvious isomorphism:
    \[\left[B^{(I^v\times \Delta^q, \partial I^v\times\Delta^q)}_r\right]^{\Delta^0}\xrightarrow{\cong}B^{(I^v\times \Delta^q, \partial I^v\times\Delta^q)}_r \]
    Together, these observations imply that, for $S=\Delta^0$ and making the obvious identifications, the morphism \eqref{eq:lemMultiLevel} is the identity of $\scrK(J^p(A^{(G/K)}),B)$. This finishes the proof.
\end{proof}

\begin{lem}\label{lem:sqSptComm}
    Let $A$ and $B$ be two $G$-algebras and let $f:C\to D$ be a morphism of $G$-algebras. Then the following square of spectra commutes:
    \[\xymatrix{\bbK^G(A, B)\ar[r]^-{-\otimes C}\ar[d]_-{-\otimes D} &
        \bbK^G(A\otimes C, B\otimes C)\ar[d]^-{f_*} \\
        \bbK^G(A\otimes D, B\otimes D)\ar[r]^-{f^*} &
        \bbK^G(A\otimes C, B\otimes D)
        }\]
\end{lem}
\begin{proof}
    Unravelling the definitions of the spectra $\bbK^G$, $\bbK_X$ and $\bbK_f$, it suffices to show that the following square commutes:
    \[\xymatrix{\bbK(A, B)\ar[r]^-{-\otimes C}\ar[d]_-{-\otimes D} &
        \bbK(A\otimes C, B\otimes C)\ar[d]^-{f_*} \\
        \bbK(A\otimes D, B\otimes D)\ar[r]^-{f^*} &
        \bbK(A\otimes C, B\otimes D)
        }\]
    At level $p\geq 0$ the latter is the following square of simplicial sets:
    \begin{equation}
        \begin{gathered}
            \label{eq:compatotimes}\xymatrix{\scrK(J^pA, B)\ar[r]^-{-\otimes C}\ar[d]_-{-\otimes D} &
            \scrK(J^p(A\otimes C), B\otimes C)\ar[d]^-{f_*} \\
            \scrK(J^p(A\otimes D), B\otimes D)\ar[r]^-{f^*} &
            \scrK(J^p(A\otimes C), B\otimes D)
            }
        \end{gathered}
    \end{equation}
    Let $q\geq 0$. Write $B^{S^n\times\Delta^q}_r$ instead of $B^{(I^n\times\Delta^q,\partial\Delta^n\times\Delta^q)}_r$ to ease notation. Let $\alpha$ be a $q$-simplex of $\scrK(J^pA,B)$, represented by an algebra homomorphism $\alpha:J^{p+n}A\to B^{S^n\times\Delta^q}_r$ for some $n,r\geq 0$. Consider the following commutative diagram of algebras:
    \[\xymatrix{J^{p+n}(A\otimes C)\ar[r]^-{\clas}\ar[d]^-{J^{p+n}(\id\otimes f)} &
        J^{p+n}(A)\otimes C\ar[r]^-{\alpha\otimes \id}\ar[d]^-{J^{p+n}(\id)\otimes f} &
        B^{S^n\times\Delta^q}_r\otimes C\ar@{=}[r]\ar[d]^-{\id\otimes f} &
        (B\otimes C)^{S^n\times\Delta^q}_r\ar[d]^-{(\id\otimes f)^{S^n\times\Delta^q}_r} \\
        J^{p+n}(A\otimes D)\ar[r]^-{\clas} &
        J^{p+n}(A)\otimes D\ar[r]^-{\alpha\otimes \id} &
        B^{S^n\times\Delta^q}_r\otimes D\ar@{=}[r] &
        (B\otimes D)^{S^n\times\Delta^q}_r
        }\]
    The composite of the top morphisms followed by the rightmost vertical morphism represents the $q$-simplex $f_*(\alpha\otimes C)$ of $\scrK(J^p(A\otimes C), B\otimes D)$. The leftmost vertical morphism followed by the composite of the bottom morphisms represents the $q$-simplex $f^*(\alpha\otimes D)$ of $\scrK(J^p(A\otimes C), B\otimes D)$. The commutativity of the diagram shows that $f_*(\alpha\otimes C)=f^*(\alpha\otimes D)$ and, thus, that \eqref{eq:compatotimes} commutes.
\end{proof}

\section{Equivariant \texorpdfstring{$\kk$}{kk}-theory as a universal \texorpdfstring{$\delta$}{delta}-functor into a graded category}
Let $G$ be a group. Equivariant $\kk^G$-theory was introduced in \cite{euge}*{Theorem 4.1.1} as the universal homotopy invariant, $G$-stable and excisive functor from $\GAlgl$ into a triangulated category. This universal property allows us, for example, to define the crossed product with a subgroup $H\subseteq G$ --- or better, with an orbit space $G/H$ --- at the level of $\kk$-theory. Indeed, by Proposition \ref{prop:existencertimesGH}, for every subgroup $H\subseteq G$ there exists a unique triangulated functor $\overline{-\rtimes G/H}:\kk^G\to \kk$ making the following diagram commute:
\[
    \xymatrix@C=6em{
    \GAlgl\ar[d]_-{j^G}\ar[r]^-{\cR(-\rtimes G/H)} & \Algl\ar[d]^-{j} \\
    \kk^G\ar[r]^-{\overline{-\rtimes G/H}} & \kk
    }
\]
A morphism $f:G/H\to G/K$ of $G$-sets induces a natural transformation $f_*:\cR(-\rtimes G/H)\to \cR(-\rtimes G/K)$ of functors $\GAlgl\to\Algl$. We would like $f$ to induce as well a natural transformation $f_*:\overline{-\rtimes G/H}\to \overline{-\rtimes G/K}$ of triangulated functors $\kk^G\to\kk$. Such a natural transformation can be thought of as a functor $\kk^G\to \kk^I$, where $\kk^I$ is the category of functors from the interval category $I$ into $\kk$. We cannot expect to define a functor $\kk^G\to \kk^I$ by the universal property of $\kk^G$ mentioned above \cite{euge}*{Theorem 4.1.1} since $\kk^I$ is not a triangulated category in an obvious way. To get around this problem, we work in the setting of graded categories.

\begin{defi}[\cite{loopsthtpy}*{Sec. 10}]
    A \emph{graded category} is a pair $(\scrA, \Omega)$ where $\scrA$ is an additive category and $\Omega$ is an automorphism of $\scrA$. A \emph{graded functor} $F:(\scrA,\Omega)\to(\scrA',\Omega')$ is an additive functor $F:\scrA\to\scrA'$ such that $F\circ \Omega = \Omega'\circ F$. Let $F,G:(\scrA,\Omega)\to (\scrA',\Omega')$ be graded functors. A \emph{graded natural transformation} $\nu:F\to G$ is a natural transformation $\nu$ such that $\Omega'(\nu_X)=\nu_{\Omega(X)}:\Omega'F(X)\to \Omega'G(X)$ for all $X\in\scrA$.
\end{defi}

Every triangulated category is a graded category. If $\scrA$ is a graded category, then $\scrA^I$ is a graded category as well.

\begin{defi}[cf. \cite{loopsthtpy}*{Def. 10.6}]\label{def:deltafunctor}
    Let $\cC$ denote either $\Algl$ or $\GAlgl$ and let $(\scrA,\Omega)$ be a graded category. A \emph{$\delta$-functor} with values in $\scrA$ consists of the following data:
    \begin{enumerate}
        \item a functor $X:\cC\to\scrA$ that preserves finite products;
        \item a morphism $\delta_{\cE}\in\Hom_\scrA(\Omega X(C), X(A))$ for every extension $\cE$ as in \eqref{eq:extension}.
    \end{enumerate}
    These morphisms $\delta_{\cE}$ are subject to the following conditions:
    \begin{enumerate}
        \item $\delta_{\cE}:\Omega X(C)\to X(A)$ is an isomorphism if $X(B)=0$;
        \item the morphisms $\delta_{\cE}$ are natural with respect to morphisms of extensions.
    \end{enumerate}
\end{defi}

Any excisive homology theory in the sense of Corti\~nas and Thom \cite{cortho}*{Sec. 6.6} is a $\delta$-functor. By \cite{loopsthtpy}*{Theorem 10.15 and Remark 10.17}, the excisive, homotopy invariant and matrix-unstable bivariant $K$-theory introduced by Garkusha in \cite{garkuni} is the universal homotopy invariant $\delta$-functor with values in a graded category. By \cite{tesisema}*{Theorem 5.2.15}, for any infinite set $X$, $M_X$-stable algebraic $\kk$-theory is the universal homotopy invariant and $M_X$-stable $\delta$-functor with values in a graded category. A similar result holds in the equivariant context:

\begin{thm}\label{thm:molesto}
    Let $G$ be a group, let $(\scrA,\Omega)$ be a graded category and let $X:\GAlgl\to\scrA$ be a homotopy invariant and $G$-stable $\delta$-functor. Then there exists a unique graded functor $\bar{X}:\kk^G\to \scrA$ such that $\bar{X}(\partial^G_{\cE})=\delta_{\cE}$ for every extension $\cE$ and such that the following diagram commutes:
    \begin{equation}\label{eq:molesto}\begin{gathered}\xymatrix{\GAlgl\ar@/_1pc/[dr]_{X}\ar[r]^-{j^G} & \kk^G\ar@{-->}[d]^-{\exists !\bar{X}} \\
            & \scrA}\end{gathered}\end{equation}
    Here, the morphisms $\partial^G_{\cE}$ are those that make $j^G:\GAlgl\to\kk^G$ into an excisive homology theory.
\end{thm}
\begin{proof}
    The proof is similar to that of \cite{euge}*{Thm. 4.1.1}.
    We start by recalling some details of the construction of $\kk^G$.
    Following \cite{cortho}*{Sec. 6.6}, an \emph{excisive homology theory} of $G$-algebras consists of a triangulated category $(\tri,\Omega)$, a functor $X:\GAlgl\to\tri$ and a morphism $\delta_{\cE}:\Omega X(C)\to X(A)$ for every extension $\cE$ as in \eqref{eq:extension} such that:
    \begin{enumerate}
        \item For every extension $\cE$, the following is a distinguished triangle in $\tri$:
              \[\xymatrix{\Omega X(C)\ar[r]^-{\delta_{\cE}} & X(A) \ar[r] & X(B)\ar[r] & X(C)}\]
        \item The morphisms $\delta_{\cE}$ are natural with respect to morphisms of extensions.
    \end{enumerate}
    Fix a universal excisive, homotopy invariant and $M_{\N\times |G|}$-stable homology theory, $j:\GAlgl\to \kk^{\GAlgl}$, $\{\partial_{\cE}\}_{\cE}$.
    The existence of such a homology theory follows from \cite{cortho}*{Thm. 6.6.2} in the case $G=1$, from \cite{euge}*{Thm. 2.6.5} for countable $G$ and from \cite{tesisema}*{Thm. 5.2.16} in the general case. Since the composite
    \[\xymatrix{\GAlgl\ar[r]^-{M_G\otimes -} & \GAlgl\ar[r]^-{j} & \kk^{\GAlgl}}\]
    is an excisive, homotopy invariant and $M_{\N\times |G|}$-stable homology theory, there exists a unique triangulated functor $M_G:\kk^{\GAlgl}\to\kk^{\GAlgl}$ making the following square commute:
    \[\xymatrix{\GAlgl\ar[r]^-{j}\ar[d]_-{M_G\otimes -} & \kk^{\GAlgl}\ar[d]^-{M_G} \\
        \GAlgl\ar[r]^-{j} & \kk^{\GAlgl}}\]
    Now define $\kk^G$ (see \cite{euge}*{Sec. 4.1} and \cite{tesisema}*{Def. 5.3.4}) as the category whose objects are those of $\kk^{\GAlgl}$ and whose morphism sets are given by:
    \[\Hom_{\kk^G}(A,B):=\Hom_{\kk^{\GAlgl}}(M_GA, M_GB)\]
    Let $f\in\Hom_{\kk^{\GAlgl}}(M_GA, M_GB)$. We will often consider $f$ both as a morphism in $\kk^{\GAlgl}$ and as a morphism $A\to B$ in $\kk^G$. To avoid ambiguity, we will write $[f]$ instead of $f$ when considering $f$ as a morphism in $\kk^G$. There is a functor $t^G:\kk^{\GAlgl}\to\kk^G$ that is the identity on objects and that sends $f\in\Hom_{\kk^{\GAlgl}}(A, B)$ to $[M_G(f)]\in\kk^G(A,B)$. Let $j^G$ be the composite:
    \[\xymatrix{\GAlgl\ar[r]^-{j} & \kk^{\GAlgl}\ar[r]^-{t^G} & \kk^G}\]
    By \cite{euge}*{Thm. 4.1.1} and \cite{tesisema}*{Thm. 5.3.8}, $j^G:\GAlgl\to\kk^G$ endowed with $\{\partial^G_{\cE}:=t^G(\partial_{\cE})\}_{\cE}$ is the universal excisive, homotopy invariant and $G$-stable homology theory.

    Let us show that $j^G:\GAlgl\to\kk^G$, $\{\partial^G_{\cE}\}_{\cE}$, is also the universal homotopy invariant and $G$-stable $\delta$-functor into a graded category. Let $X:\GAlgl\to(\scrA,\Omega)$, $\{\delta_{\cE}\}$, be a homotopy invariant and $G$-stable $\delta$-functor. Since $X$ is $G$-stable, then it is $M_{\N\times |G|}$-stable; see \cite{euge}*{Section 3}. By \cite{tesisema}*{Theorem 5.2.15} there exists a unique graded functor $\hat{X}:\kk^{\GAlgl}\to\scrA$ making the following diagram commute and such that $\hat{X}(\partial_{\cE})=\delta_{\cE}$ for every extension $\cE$:
    \[\xymatrix{\GAlgl\ar[r]^-{j}\ar@/_1pc/[dr]_-{X} & \kk^{\GAlgl}\ar@{-->}[d]^-{\exists ! \hat{X}} \\
        & \scrA}\]
    Suppose that there exists a graded functor $\bar{X}:\kk^G\to\scrA$ making the triangle \eqref{eq:molesto} commute and such that $\bar{X}(\partial^G_{\cE})=\delta_{\cE}$ for every extension $\cE$. Then we have $X=\bar{X}\circ j^G = (\bar{X}\circ t^G)\circ j$ and $\delta_{\cE}=\bar{X}(\partial^G_{\cE})=(\bar{X}\circ t^G)(\partial_{\cE})$ for every extension $\cE$. By the uniqueness of $\hat{X}$ we must have $\hat{X}=\bar{X}\circ t^G$. By \cite{tesisema}*{Lemma 5.3.6}, for any $f\in\Hom_{\kk^{\GAlgl}}(M_GA,M_GB)$ we have a commutative diagram in $\kk^G$ as follows:
    \[\xymatrix{M_GA\ar[r]^-{t^G(f)}\ar[d]^-{\cong}_-{t^G(\iota_A')} & M_GB\ar[d]^-{t^G(\iota_B')}_-{\cong} \\
        M_{G_+}A & M_{G_+}B \\
        A\ar[u]_-{\cong}^-{t^G(\iota_A)}\ar[r]^-{[f]} & B\ar[u]_-{t^G(\iota_B)}^-{\cong}
        }\]
    Here $\iota$ and $\iota'$ are the morphisms appearing in the $G$-stability zig-zag \eqref{eq:zigzag}. Upon applying $\bar{X}$ to the diagram above we get:
    \[\bar{X}([f])=\hat{X}(\iota_B)^{-1}\circ \hat{X}(\iota_B')\circ \hat{X}(f)\circ \hat{X}(\iota_A')^{-1}\circ \hat{X}(\iota_A)\]
    This shows that $\bar{X}$ is uniquely determined. Moreover, it is straightforward to verify that the last equation defines a graded functor $\bar{X}:\kk^G\to\scrA$ with the desired properties.
\end{proof}

%%%%%%%%%%%%%%%%%%%%%%%%%%%%%%%%%%%%%%%%%%%%%%%%%%%%%%%%%%%%%%%%%%%%%%%%%%%%%%%%%%%%%
%%%%%%%%%%%%%%%%%%%%%%%%%%%%%%%%%% Bibliography %%%%%%%%%%%%%%%%%%%%%%%%%%%%%%%%%%%%%
%%%%%%%%%%%%%%%%%%%%%%%%%%%%%%%%%%%%%%%%%%%%%%%%%%%%%%%%%%%%%%%%%%%%%%%%%%%%%%%%%%%%%

\end{document}